 \documentclass[11pt]{article}

 \usepackage{hyperref} 
 \usepackage{amssymb}
 \usepackage{amsthm}
 \usepackage{mathrsfs}
 \usepackage{amsmath}
  \usepackage{lineno}
\usepackage[text={180mm,220mm},centering]{geometry}

\newcommand{\be}{\begin{equation}}
\newcommand{\ee}{\end{equation}}
\newcommand{\bes}{\begin{equation}\begin{aligned}}
\newcommand{\ees}{\end{aligned}\end{equation}}
\newcommand{\ben}{\begin{equation}\nonumber\begin{aligned}}

\newcommand{\dist}{{\rm dist}}

\renewcommand{\leq}{\leqslant}
\renewcommand{\geq}{\geqslant}  

 \newtheorem{theorem}{Theorem}[section]
\newtheorem{lemma}[theorem]{Lemma}

\newtheorem{definition}[theorem]{Definition}

\newtheorem{corollary}[theorem]{Corollary}
\numberwithin{equation}{section}

\allowdisplaybreaks

\begin{document}

\baselineskip=1.0\baselineskip

\pagestyle{plain}

\title{\bf
Uniform Large Deviations of Mckean-Vlasov
Stochastic Fractional $(\alpha,p)$-Laplacian
Equations Driven by Superlinear Noise on $\mathbb{R}^d$
  }

\author{
Renhai  Wang\footnote{School of Mathematical Sciences, Guizhou Normal University, Guiyang 550025, China,
 Email:   rwang-math@outlook.com},\;
 Zhang Chen\footnote{School of Mathematics,
 Shandong University, Jinan 250100,   China,
 Email:   zchen@sdu.edu.cn},\;
 Bixiang Wang\footnote{
Department of Mathematics, New Mexico Institute of Mining and
Technology,  Socorro,  NM~87801, USA, \
Email: bwang@nmt.edu} }

\date{}

\maketitle

\medskip

\begin{abstract}
The global-in-time well-posedness and uniform
large deviation principles (LDPs) are investigated for a wide class of
\emph{Mckean-Vlasov}
stochastic \emph{non-local fractional} $(\alpha,p)$-Laplacian equations\
with $\alpha \in (0,1)$ and $p>2$ driven by \emph{superlinear}
multiplicative noise
 defined on the whole space $\mathbb{R}^d$, where the non-local nonlinear fractional  $(\alpha,p)$-Laplace operator
is defined by a singular, symmetrical and
translation invariant kernel function,
the distribution-dependent drift terms have arbitrary polynomial growth
  and the distribution-dependent diffusion terms have \emph{superlinear}  growth. The global-in-time well-posedness is established under these  conditions by using the
monotone method and an   domain expansion   argument.
Under additional conditions on
  the growth  of   diffusion  terms,
  we  establish the Freidlin-Wentzell and Dembo-Zeitouni uniform LDPs by using the generalized weak convergence method developed by Salins (Probab.
Surv., 16:99-142, 2019). The idea of uniform tail-ends estimates, the pseudo monotone technique and the Arzel\`{a}-Ascoli theorem are combined to prove the \emph{weak-to-strong} continuity of solution operators of the controlled
equations in order to overcome many difficulties caused by the \emph{noncompactness} of Sobolev embeddings on $\mathbb{R}^d$
and  the \emph{nonlinearity} of the fractional $(\alpha,p)$-Laplace
operator.
The superlinearly growing diffusion
 term is  carefully controlled by
 using the
  dissipative drift terms and several algebraic inequalities.
\end{abstract}



\section{Introduction}
\subsection{Background of uniform large
deviation principles}
The theory of large
deviation principles (LDPs), developed by
Freidlin, Wentzell and others, plays a central role in probability,
statistics and other related fields. It describes the exponential
decay rate of rare probabilities of solutions of small noise driven stochastic systems as the noise intensity tends to zero. An elementary understanding in this area is that the LDPs are equivalent to the so-called Laplace principles when the underlying space is a Polish one. Starting from this understanding, the LDP problems are therefore reduced, and then studied by the \emph{weak convergence method} which is based
on a variational representation theorem for positive functionals of Brownian motions \cite{Budhiraja}.
For more detailed discussions on the weak convergence method solving LDP problems, we refer the reader to Budhiraja and Dupuis \cite{Budhiraja}, Dupuis and Ellis
\cite{Dupuis-}, and Freidlin \cite{Freidlin1}.
Compared with the classical method which requires certain exponential tightness criteria \cite{Cerrai1,Chow,Freidlin,Peszat1,Sowers1}, the weak convergence method is more effective, and has been used to discuss the
LDPs of  finite and infinite
dimensional stochastic evolution problems, see \cite{
Brzezniak1,Budhiraja,Budhiraja1,Cerrai2,Chueshov,
LR4,Ren1,Rockner1} and the references therein.

The \emph{uniform} LDP plays a
crucial role
for understanding
the  asymptotic properties of invariant measures
  and the  exit time  of solutions   of a stochastic system   from a
  given  domain  as the noise intensity approaches zero, see \cite{Biswasa,Chenal1,Dembo,Freidlin,Martinelli1}.
  There are several different types of definitions of
  uniform LDPs  introduced
  in the literature.
  The first type of  definition of a uniform LDP is the Freidlin-Wentzell uniform LDP proposed by
Freidlin and Wentzell \cite{Freidlin}.
The second type of definition of a uniform LDP is
the Dembo-Zeitouni uniform LDP introduced by Dembo and Zeitouni \cite{Dembo}. The third type of definition of uniform LDP is called the uniform Laplace principle which can be founded in \cite{Budhiraja1,Dupuis-}. The forth  type of definition of uniform LDP is called the equicontinuous uniform Laplace principle introduced by Salins \cite{Salins2}. These types of uniform LDPs of finite and infinite dimensional stochastic evolution systems have been investigated in \cite{Budhiraja1,Gautier,Salins1, Sowers1}
 and many others. A natural question is   whether these
 concepts  of uniform LDPs are equivalent,
 which   has been
well-answered by Salins in his
paper \cite{Salins2}. He pointed out that the Freidlin-Wentzell uniform LDP and the uniform Laplace principle are not
equivalent in general, but they are equivalent
with an extra compactness condition, see \cite[Theorem 2.5]{Salins2}. He proved that the Freidlin-Wentzell uniform LDP and the equicontinuous uniform Laplace principle are equivalent
without adding any extra conditions, see \cite[Theorem 2.9]{Salins2}. In addition, he demonstrated that the Freidlin-Wentzell uniform LDP is not equivalent to the Dembo-Zeitouni uniform LDP over bounded initial data, but they are equivalent over \emph{compact} initial data under a continuity condition on the level sets of the rate functions, see \cite[Theorem 2.9]{Salins2}.

\subsection{Mckean-Vlasov
stochastic fractional $(\alpha,p)$-Laplacian equations on $\mathbb{R}^d$}
Our main interest in this paper is to discuss
the global-in-time well-posedness as well as the Freidlin-Wentzell and Dembo-Zeitouni uniform
LDPs
for the
\emph{Mckean-Vlasov}
stochastic \emph{non-local fractional} $(\alpha,p)$-Laplacian equations on  $\mathbb{R}^d$ driven by \emph{superlinear}
multiplicative noise:
\begin{equation}\label{Int19-}
\left\{
  \begin{aligned}
  &d \psi(t)+ \mathfrak{L}_{\mathcal{K}_{p}^\alpha} \psi(t) dt
=\sum_{i=1}^2H_i(t,x,\psi(t),\mathcal{L}_{
\psi(t)})dt
+ \sqrt{\epsilon}\sum_{k\in\mathbb{N}} \sigma_{k}(t,x,\psi(t),\mathcal{L}_{
\psi(t)}) d\mathcal{W}_k, \ \ t>0, \\
  &\psi(0,x)=\psi_0(x),\\
  \end{aligned}
\right.
\end{equation}
where $d\in \mathbb{N}$ is arbitrary, $\epsilon\in(0,1]$ is the noise intensity,
$\mathcal{L}_{
\psi(t)}$ denotes the law of the unknown
function $\psi(t)$, $(\mathcal{W}_{k})_{k\in\mathbb{N}}$ is a sequence of independent real-valued
Wiener processes on a given complete filtered probability space $(\Omega,\mathfrak{F}, \{\mathfrak{F}_t\}_{t\in\mathbb{R}^+},\textbf{P})$, the nonlinear functions\footnote{We use $\mathcal{P}_2(L^2(\mathbb{R}^d))$ to denote the space of all probability measures on $L^2(\mathbb{R}^d)$ with finite second moments}
$H_1,H_2:\mathbb{R}^+\times\mathbb{R}^d
\times\mathbb{R}\times\mathcal{P}_2(L^2(\mathbb{R}^d))\rightarrow\mathbb{R}$ have polynomial growth of arbitrary orders
$p-1$ $(p>2)$ and $q-1$ $(q>2)$ in their third arguments, respectively. The diffusion coefficient $\{\sigma_k\}_{k=1}^\infty:\mathbb{R}^+
 \times\mathbb{R}^d
\times\mathbb{R}\times
\mathcal{P}_2(L^2(\mathbb{R}^d))
\rightarrow\mathbb{R}$ is a family of \emph{locally Lipschitz} continuous functions which have \emph{superlinear}
 growth rates $p^*\in[2,p)$ and $q^*\in[2,q)$ in their third arguments, respectively. The \emph{non-local} and \emph{nonlinear} fractional  $(\alpha,p)$-Laplace
operator $\mathfrak{L}_{\mathcal{K}_{p}^\alpha}$
is defined by means of the singular, symmetrical and translation kernel function $\mathcal{K}_{p}^\alpha:\mathbb{R}^d\times
\mathbb{R}^d\rightarrow\mathbb{R}^+$ with $\alpha\in(0,1)$ and $p>2$, that is\footnote{ ``P.V.'' denotes the Cauchy  principal value.},
\begin{align}\label{integro-differential}
\mathfrak{L}_{\mathcal{K}_{p}^\alpha} \psi
(x)
&= \mbox{P.V.}\int_{\mathbb{R}^d}
|\psi(x)-\psi(y)|^{p-2}\big(\psi(x)-\psi(y)\big) \mathcal{K}_{p}^\alpha(x,y)dy,\ \ x\in\mathbb{R}^d.
 \end{align}

\subsection{Background and  recent  results of \eqref{Int19-}}
The McKean-Vlasov stochastic systems are
referred to as distribution-dependent or mean field stochastic systems, which are related to
the limits of interacting particle systems and other related topics in probability and physical sciences, see Jara \cite{Jara}, Lasry and Lions \cite{Lasry1}, McKean \cite{McKean} and  Buckdahn et al.\cite{Buckdahn}.
Partial differential equations (PDEs) with fractional Laplace operators have rich applications in
physics, chemistry, biology, finance, engineering, materials science, probability and
many other fields. In particular,
the fractional Laplace operator is related to the anomalous diffusion (the
diffusion process of complex systems) that do not follow the Gaussian statistics, and it represents the infinitesimal generator of
L\'{e}vy stable diffusion process in the theory of probability, see \cite{Abe,Jara,Sato}.
Several
 topics such as the existence, regularity, H\"{o}lder continuity, free boundary problems, nonlocal tug-of-war, Harnack inequalities, chaotic dynamics, fractional integration by parts formulas, and fractional eigenvalues of PDEs with fractional Laplace
operators have been extensively discussed in \cite{Byun,Brasco1,Barrios,Bjorland,Caffarelli,
Caffarelli-,
Castro1,Di2012bull,
Grigorenko,Lindgren,Muratori,Maz,Moen,Silvestre,Servadei}
and the related references.

When $p>2$, $\alpha\in(0,1)$ and $\mathcal{K}_{p}^\alpha(x,y)=C(d,p,\alpha)|x-y|^{-d-\alpha p}$ with $C(d,p,\alpha)>0$, the generalized fractional $(\alpha,p)$-Laplace operator $\mathfrak{L}_{\mathcal{K}_{p}^\alpha}$ is reduced to the well-known fractional $(\alpha,p)$-Laplace
operator $(-\Delta)_{p}^\alpha$:
$$(-\Delta)_{p}^\alpha\psi(x)
=C(d,p,\alpha)\ \mbox{P.V.}\int_{\mathbb{R}^d}|\psi(x)-\psi(y)|^{p-2}
\big(\psi(x)-\psi(y)\big)|x-y|^{-d-p\alpha }dy, \ \ x\in\mathbb{R}^d.
 $$
The existence and robustness of
pathwise random attractors of stochastic PDEs involving fractional $(\alpha,p)$-Laplace
operators $(-\Delta)_{p}^\alpha$ have been explored in \cite{Chenpa,Wang-wangb,Wang-wangb2}.
In particular, the global-in-time well-posedness, mean random attractors and invariant measures of stochastic PDEs with fractional $(\alpha,p)$-Laplace $\mathfrak{L}_{\mathcal{K}_{p}^\alpha}$ and
$(-\Delta)_{p}^\alpha$ were recently studied in \cite{Wangjde,RWangsubmitted}.

When $p=2$, $\alpha\in(0,1)$ and $\mathcal{K}_{p}^\alpha(x,y)
=C(d,\alpha)|x-y|^{-d-2\alpha}$ with
$C(d,\alpha)>0$, the
nonlinear fractional $(\alpha,p)$-Laplace operator
$\mathfrak{L}_{\mathcal{K}_{p}^\alpha}$ becomes the classical \emph{linear} fractional Laplace operator $(-\Delta)^\alpha$:
$$(-\Delta)^\alpha\psi(x)
=C(d,\alpha)\ \mbox{P.V.}\int_{\mathbb{R}^d}
\big(\psi(x)-\psi(y)\big)|x-y|^{-d-2\alpha }dy, \ \ x\in\mathbb{R}^d.
 $$
The pathwise random attractors, mean random attractors and invariant measures of
stochastic PDEs involving fractional Laplace operator $(-\Delta)^\alpha$ have been investigated in \cite{Chenzhang4,Gu2018jde,wangjde2019,rwang2,WangGUOWANG,Xu1}.

Next,  we discuss the  results on the   LPDs and
 uniform LPDs of \eqref{Int19-} available
 in the literature:

\underline{Case 1:} $\alpha=1$ and $p=2$. If the non-local and nonlinear fractional  $(\alpha,p)$-Laplace
operator $\mathfrak{L}_{\mathcal{K}_{p}^\alpha}$
is replaced by the standard Laplace
operator $-\Delta$, then problem \eqref{Int19-}
becomes the standard stochastic reaction-diffusion equation. The LPDs and uniform LPDs of such stochastic equations defined on bounded domains with globally or locally Lipschitz
drift and diffusion terms have been examined in the literature aforementioned.

\underline{Case 2:}
$\alpha\in(0,1)$ and $p=2$. If the nonlinear fractional $(\alpha,p)$-Laplace
operator $\mathfrak{L}_{\mathcal{K}_{p}^\alpha}$
is reduced to the fractional Laplace
operator $(-\Delta)^\alpha$, then problem \eqref{Int19-} is reduced to the standard
stochastic fractional reaction-diffusion equation. The LDPs of such stochastic fractional equations defined on \emph{unbounded} domains were discussed in \cite{wangLargedeviationsjde} when the
drift term has an arbitrarily
polynomial growth rate and the
diffusion term is globally
Lipschitz. The results of \cite{wangLargedeviationsjde} were recently
improved in \cite{Wang2024submitted} for superlinear diffusion terms, and in
\cite{Chenzhang4b} for distribution dependent drift and diffusion terms. We also mention that
the LDPs of stochastic fractional reaction-diffusion equations on bounded domains driven by L\'{e}vy noise were recently examined by Xu and Caraballo \cite{Xu2}.

\underline{Case 3:}
$\alpha\in(0,1)$ and $p>2$.
No results available for LDPs of stochastic PDEs involving the \emph{nonlinear} fractional  $(\alpha,p)$-Laplace
operators $(-\Delta)_{p}^\alpha$ and $\mathfrak{L}_{\mathcal{K}_{p}^\alpha}$.

\subsection{Motivations}
To the best of our knowledge, there is only one result \cite{Wang2024submitted} in the literature  on the study of LDPs of
superlinear noise driven stochastic PDEs on unbounded domains, and there are no results reported in the literature regarding the LDPs of
\emph{Mckean-Vlasov}
stochastic non-local \emph{fractional} $(\alpha,p)$-Laplacian equation \eqref{Int19-} driven by \emph{superlinear}
multiplicative noise
 defined on the
\emph{unbounded} space $\mathbb{R}^d$ even
 in the case where $s=1$ and $p=2$. In the present article, we want to examine, in a more general setting, the Freidlin-Wentzell and Dembo-Zeitouni uniform LDPs of
Mckean-Vlasov stochastic fractional $(\alpha,p)$-Laplacian equation \eqref{Int19-} on $\mathbb{R}^d$ with polynomial drift terms and  superlinear diffusion terms for any $\alpha\in(0,1)$ and $p>2$. As far as we know, this is the first work on the LDPs and uniform LDPs of \emph{Mckean-Vlasov} stochastic  PDEs with \emph{nonlinear} fractional Laplace operators and \emph{superlinear} diffusion terms even when the underlying domain is bounded. We remark that our results in this paper are also
valid for the case of $\alpha\in(0,1]$
and $p=2$ when the nonlinear fractional  $(\alpha,p)$-Laplace
operator $(-\Delta)_{p}^\alpha$ becomes
the linear fractional Laplace operator $(-\Delta)^\alpha$
and  the standard Laplace operator $-\Delta$. However, we will
focus on    the case  where $\alpha\in(0,1)$ and $p>2$.

\subsection{Main results}
To introduce our main results, we let $\mathbb{H}:=L^2(\mathbb{R}^d)$
with norm $\| \cdot \|$,
$\mathbb{V}_1=
:W^{\alpha,p}(\mathbb{R}^d)$ and
$\mathbb{V}_2:= L^q(\mathbb{R}^d)$,
where $W^{\alpha,p}(\mathbb{R}^d)$ with $\alpha\in(0,1)$ and $p>2$ (see \cite{Di2012bull}) is the
fractional Sobolev space:
\begin{align*}
W^{\alpha,p}(\mathbb{R}^d):= \bigg\{\psi\in L^p(\mathbb{R}^d):\int_{\mathbb{R}^d}
\int_{\mathbb{R}^d}|\psi(x)-\psi(y)|^p|x-y|^{-(d+\alpha p)}
dxdy<\infty  \bigg\}
\end{align*}
with the  norm
\begin{align}\label{int4}
\|\psi\|_{W^{\alpha,p}(\mathbb{R}^d)}&
=\bigg(\int_{\mathbb{R}^d}|\psi(x)|^pdx+
\int_{\mathbb{R}^d}
\int_{\mathbb{R}^d}|\psi(x)-
\psi(y)|^p|x-y|^{-(d+\alpha p)}
dxdy\bigg)^{1/p},\ \ \psi\in W^{\alpha,p}(\mathbb{R}^d).
\end{align}

The first result is the global-in-time well-posedness of \eqref{Int19-}
when the diffusion coefficient $\sigma_k$ has
 \emph{superlinear} growth
  in its third argument.

\begin{theorem}\label{Theorem4}(Global-in-time well-posedness)
Let conditions {\bf A} and {\bf B}
in Section 2 with $p^*\in[2,p)$ and $q^*\in[2,q)$ be satisfied.
If $\psi_0\in L^{2}(\Omega,\mathfrak{F}_0,\mathbb{H})$, then problem \eqref{Int19-} has a unique solution $\psi$ in the sense of Definition \ref{Defsolutionsolution} such that
\begin{align}\label{G8}
&
\mathbf{E}\bigg[\sup_{r\in[0,T]}\|\psi
(r)
\|^{2}
+\int_{0}^{T}\Big(\|\psi(r)
\|_{\mathbb{V}_1}^{p} +\|\psi(r)\|^q_{\mathbb{V}_2}\Big)
 dr\bigg]\leq M(T) (1+\mathbf{E}[\|\psi_0\|^2]),\ \ \forall\ T>0,
\end{align}
where $M(T)>0$ is a constant independent of $\epsilon$.
\end{theorem}

The second result is the Freidlin-Wentzell uniform LDPs of \eqref{Int19-} when the diffusion coefficient $\sigma_k$ has  \emph{superlinear} growth rate $p^*\in\big[2,\frac{p+2}{2}\big]\subseteq[2,p)$ or
$q^*\in\big[2,\frac{q+2}{2}\big]\subseteq[2,q)$ in its third argument:
\begin{theorem}\label{MianI}
(Freidlin-Wentzell uniform LDPs) Let conditions  {\bf A}
in section 2  and {\bf B1}
in Section 4    with $p^*\in\big[2,\frac{p+2}{2}\big]$ and $q^*\in\big[2,\frac{q+2}{2}\big]$
be satisfied. Then for every bounded set $\mathcal{B}\subseteq \mathbb{H}$, the family of solutions
$\{\psi^\epsilon(\cdot,\psi_0): \epsilon\in(0,1),\ \psi_0\in \mathcal{B}\}$ of \eqref{Int19-}
satisfies the Freidlin-Wentzell uniform LDPs in $C([0,T],\mathbb{H} )\bigcap\\ L^p([0,T], \mathbb{V}_1  )\bigcap L^q([0,T], \mathbb{V}_2)$ uniformly on $\mathcal{B}$ with the good rate function defined by \eqref{main1}.
\end{theorem}

The third result is the Dembo-Zeitouni uniform LDPs of \eqref{Int19-}
when the diffusion coefficient $\sigma_k$ has  \emph{superlinear} growth rate
 $p^*\in\big[2,\frac{p+2}{2}\big]\subseteq[2,p)$
 or
 $q^*\in\big[2,\frac{q+2}{2}\big]\subseteq[2,q)$ in its third argument:
\begin{theorem}\label{MianII}
(Dembo-Zeitouni uniform LDPs) Let conditions  {\bf A}
in section 2  and {\bf B1}
in Section 4  with $p^*\in\big[2,\frac{p+2}{2}\big]$ and $q^*\in\big[2,\frac{q+2}{2}\big]$
be satisfied.
Then for every compact set $\mathcal{K}\subseteq \mathbb{H}$, the family of solutions
$\{\psi^\epsilon(\cdot,\psi_0): \epsilon\in(0,1),\ \psi_0\in \mathcal{K}\}$ of \eqref{Int19-}
satisfies the Dembo-Zeitouni uniform LDPs in $C([0,T],\mathbb{H} )\bigcap L^p([0,T], \mathbb{V}_1  )\bigcap \\ L^q([0,T], \mathbb{V}_2)$ uniformly on $\mathcal{K}$ with the good rate function defined by \eqref{main1}.
\end{theorem}

\subsection{Difficulties and methods}
The difficulties and methods
for solving the
global-in-time well-posedness and uniform
LDPs of the
\emph{Mckean-Vlasov}
stochastic fractional $(\alpha,p)$-Laplacian equation \eqref{Int19-} on $\mathbb{R}^d$ driven by \emph{superlinear} noise are:

\begin{itemize}
  \item \emph{Unboundedness of $\mathbb{R}^d$}. Because of the unboundedness of $\mathbb{R}^d$, we do not have the Gelfand triples for the three spaces $\mathbb{H}$, $\mathbb{V}_1$ and $\mathbb{V}_2$, and hence we can not directly use the variational
frameworks known in \cite{Barbu2010,Caraballo4,Krylov,LR1,
Minty,rwangma} (adapted to PDEs defined on bounded domains) to discuss the existence of solutions of \eqref{Int19-} defined on $\mathbb{R}^d$. To solve this problem, we consider a Faedo-Galerkin approximation equation of \eqref{Int19-} on the bounded ball
$\mathcal{Q}_{k}:=\{x\in\mathbb{R}^d: |x|< k\}$. Then we establish the existence of global-in-time
solutions $\psi_k$ of the approximation equation under the Gelfand triples:
$\mathbb{V}_{1,k}\subseteq \mathbb{H}_k\equiv \mathbb{H}_k^*\subseteq \mathbb{V}_{1,k}^*$ and $ \mathbb{V}_{2,k}\subseteq \mathbb{H}_k\equiv \mathbb{H}_k^*\subseteq \mathbb{V}^*_{2,k}$,  where
$\mathbb{H}_k:=\big\{\psi\in \mathbb{H}: \psi=0 \ \mbox{a.e. on}\ \ \mathbb{R}^d\backslash\mathcal{Q}_k\big\}$,
$\mathbb{V}_{1,k}:=\big\{\psi\in\mathbb{V}_{1}: \psi=0 \ \mbox{a.e. on}\ \ \mathbb{R}^d\backslash\mathcal{Q}_k\big\}$ and
$\mathbb{V}_{2,k}:=\big\{\psi\in \mathbb{V}_{2}: \psi=0 \ \mbox{a.e. on}\ \ \mathbb{R}^d\backslash\mathcal{Q}_k\big\}$. By deriving  several uniform estimates of  $\psi_k$ with respect to $k$ and establishing
the weak-star convergence of the approximation solutions in $\big(L^ 2(\Omega,
				L^1([0,T],\mathbb{H} ))
				\big)
				^*$, we show
that the weak limit of $\psi_k$ as $k\rightarrow\infty$ in a certain sense
is the desired solution
to \eqref{Int19-} on $\mathbb{R}^d$.  There are many  results
in the literature
 on the LDPs for  stochastic PDEs defined on \emph{bounded} domains, where the solution
operators of the corresponding controlled equations are \emph{compact} due to compact Sobolev embeddings
on bounded domains. This, in fact, plays a crucial role in establishing the tightness and
convergence of solutions of stochastic PDEs on bounded domains, especially for treating the
polynomial drift terms of
arbitrary orders (\cite{Cerrai1,Ren1}). For the controlled equation of \eqref{Int19-} defined on the unbounded domain $\mathbb{R}^d$, the solution
operators fail to gain such compactness property since Sobolev embeddings on $\mathbb{R}^d$ are no longer compact,
and hence one can not use such a compact embedding argument to discuss the LDPs of
\eqref{Int19-}. This is an essential obstacle in establishing the LDPs of
\eqref{Int19-} on $\mathbb{R}^d$. This obstacle is surmounted
by the idea of uniform
tail-ends estimates developed in \cite{Wangphysd}. In fact, by choosing an appropriate smooth cut-off function, we prove that the solutions of the
controlled equation of \eqref{Int19-} are sufficiently
small in $L^2(
 \mathbb{R}^d\setminus
 \mathcal{Q}_{n})$ for
 large enough $n$ uniformly for finite times, compact
initial data and bounded controls, see Lemma \ref{tightness2ddfhtt}. As a result of this and the compactness of the Sobolev embedding
$W^{\alpha,p}(\mathcal{Q}_n)\hookrightarrow L^2(\mathcal{Q}_{n})$, we prove the  continuity of solutions of the controlled equation with respect to the control
$u$ from the weak topology of
$L^2([0,T],\ell^2)$ to the strong topology of the  Polish space $C([0,T],\mathbb{H} )\bigcap L^p([0,T], \mathbb{V}_1  ) \bigcap L^q([0,T], \mathbb{V}_2)$, see Lemma \ref{Weak-to-strog}.
Note that the weak-to-strong continuity of solutions of the controlled equation plays a vital role to prove that the rate function defined in \eqref{main1} is  a
\emph{good} one. This is a
 necessary step to establish the uniform LDPs of \eqref{Int19-}.

\item \emph{Nonlinearity
and nonlocality of $\mathfrak{L
}_{\mathcal{K}_{p}^\alpha}$}. Since the nonlinear operator
$\mathfrak{L}_{\mathcal{K}_{p}^\alpha}$ with $p>2$ fails to generate a
$C_0$-semigroup, the semigroup method \cite{DaPrato1} and
the regularization method \cite{Mohan1,wangjde2019,WangGUOWANG}
can not be used to show the existence of solutions of \eqref{Int19-} involving $\mathfrak{L}_{\mathcal{K}_{p}^\alpha}$.
We will use the monotone method
\cite{Krylov,LR1,Wangjde} to prove the existence of solutions of \eqref{Int19-}.
Unlike the case of the linear fractional Laplace operator \cite{Chenzhang4b,wangLargedeviationsjde,
Wang2024submitted}, the nonlinearity of $\mathfrak{L}_{\mathcal{K}_{p}^\alpha}$ also introduces several difficulties in deriving the weak-to-strong continuity of solutions of the controlled equation. These difficulties are overcome by the \emph{pseudo monotone technique} which is based on the monotonicity and hemicontinuity of $\mathfrak{L}_{\mathcal{K}_{p}^\alpha}$.
In addition, the uniform tail-ends
estimates in Lemma \ref{tightness2ddfhtt} involving the
non-local and nonlinear operator $\mathfrak{L}_{\mathcal{K}_{p}^\alpha}$ for $p>2$ are much more complicated than the linear case for $p=2$ (see \cite{Chenzhang4b,wangLargedeviationsjde,
Wang2024submitted}). By carefully analyzing the
nonlocal and nonlinear structures of $\mathfrak{L}_{\mathcal{K}_{p}^\alpha}$, we eventually  derive uniform tail-ends
estimates of \eqref{Int19-}.

\item \emph{Superlinearity of} $\sigma_k$.
The superlinearly growing $\sigma_k$ introduces several difficulties in proving the global-in-time well-posedness of \eqref{Int19-} and the weak-to-strong
continuity of solutions of the controlled equation.
In this paper, we will
carefully control the
superlinearly growing diffusion term $\sigma_k$ by the \emph{polynomially dissipative} drift terms $H_1$ and $H_2$. In addition, we do not need the \emph{time-H\"{o}lder continuity} condition on the distribution-dependent diffusion terms when proving the uniform LDPs of \eqref{Int19-}
For a technique reason in proving the weak-to-strong
continuity of solutions of the controlled equation, we restrict the growing rate
 of $\sigma_k$ to
 $p^*\in\big[2,\frac{p+2}{2}\big]$
 or  $q^*\in\big[2,\frac{q+2}{2}\big]$.
However, our method fails when  $p^*\in\big(\frac{p+2}{2},p)$ or  $q^*\in\big(\frac{p+2}{2},q)$.
\end{itemize}

\subsection{Outline of the paper}
In the next section, we prove the global-in-time well-posedness of \eqref{Int19-} when the diffusion coefficient has \emph{superlinear} growth   $p^*\in[2,p)$
or  $q^*\in[2,q)$ in its third argument. In section 3, we recall the generalized weak convergence theory
for Freidlin-Wentzell and Dembo-Zeitouni uniform LDPs. In the last section, we establish the Freidlin-Wentzell and Dembo-Zeitouni uniform LDPs of \eqref{Int19-} when the diffusion coefficient has \emph{superlinear} growth   $p^*\in\big[2,\frac{p+2}{2}\big]\subseteq[2,p)$
or  $q^*\in\big[2,\frac{q+2}{2}\big]\subseteq[2,q)$.

\section{Global-in-time well-posedness
of \eqref{Int19-}}
In this section, we combine a
domain expansion argument,  the
monotone method  and the fixed point theorem
to show  the  global-in-time well-posedness of the \emph{distribution dependent}
stochastic \emph{fractional} $(\alpha,p)$-Laplacian equation
\eqref{Int19-} on $\mathbb{R}^d$ with
\emph{superlinear} diffusion terms. Our ideas are motivated by the works in \cite{Browder,Krylov,Minty,Pardoux,
Zeidler} and \cite{Caraballo4,
Krylov,LR1,Wangjde,Wang-wangb,
Wang-wangb2,RWangsubmitted} for determistic and stochastic evolution equations, respectively.

\subsection{Fractional
$(\alpha,p)$-Laplacian and
Sobolev spaces
}Let
 $\mathcal{K}_{p}^\alpha:\mathbb{R}^d\times\mathbb{R}^d\rightarrow\mathbb{R}^+\bigcup{\{+\infty\}}$ with $\alpha\in(0,1)$ and $p>2$
be a \emph{kernel function} with
\begin{itemize}
  \item \emph{Singularity}: $\mathcal{K}_{p}^\alpha(x,y)$
  has singularity at $x=y\in\mathbb{R}^d$.
  \item \emph{Symmetry}: $\mathcal{K}_{p}^\alpha(x,y)=
      \mathcal{K}_{p}^\alpha(y,x)$ for all $x,y\in\mathbb{R}^d$ with $x\neq y$.
  \item \emph{Translation invariance}: $\mathcal{K}_{p}^\alpha(x+z,y+z)=
      \mathcal{K}_{p}^\alpha(x,y)$ for any $x,y,z\in\mathbb{R}^d$ and $x\neq y$.
  \item \emph{Growth property}: there exists a constant $K\geq1$ such that
   $$
   K^{-1}|x-y|^{-(d+\alpha p)}\leq \mathcal{K}_{p}^\alpha(x,y)\leq
   K|x-y|^{-(d+\alpha p)}\ \mbox{for all $x,y\in\mathbb{R}^d$ and $x\neq y$}
   $$
  \item \emph{Continuity}: the mapping $x\rightarrow \mathcal{K}_{p}^\alpha(x,y)$ is continuous in $\mathbb{R}^d\setminus \{y\}$ for any $y\in\mathbb{R}^d$.
\end{itemize}
Based on the definition of the
\emph{kernel function} $\mathcal{K}_{p}^\alpha$,
the non-local fractional $(\alpha,p)$-Laplace
operator $\mathfrak{L}_{\mathcal{K}_{p}^\alpha}$ is defined by the formula \eqref{integro-differential}.
The following Gagliardo semi-norm of $W^{\alpha,p}(\mathbb{R}^d)$ is also used
\begin{align}\label{int5}
\|\psi\|_{\dot{W} ^{\alpha,p}(\mathbb{R}^d)}&=
\Bigg(\int_{\mathbb{R}^d}
\int_{\mathbb{R}^d}|\psi(x)-
\psi(y)|^p|x-y|^{-(d+\alpha p)}
dxdy \Bigg)^{1/p},\ \ \psi\in W^{\alpha,p}(\mathbb{R}^d).
\end{align}
Note that \eqref{int4} and \eqref{int5}
imply that
$
\|\psi\|^p_{W^{\alpha,p}(\mathbb{R}^d)}=
 \|\psi\|_{L^{p}(\mathbb{R}^d)}^p+\|u
 \|^p_{W_S^{\alpha,p}(\mathbb{R}^d)}
$
for $\psi\in W^{\alpha,p}(\mathbb{R}^d)$.

In
what follows,
 for
  simplicity, we  will
   write $\mathbb{H}:=L^2(\mathbb{R}^d)$, $\mathbb{V}_1:=W^{\alpha,p}(\mathbb{R}^d)$  and  $\mathbb{V}_2:=L^q(\mathbb{R}^d)$, where $q>2$
is the number associated with the growth order of $H_2$. The norm and inner product of $\mathbb{H}$ are
written as $\|\cdot\|:=\|\cdot\|_{\mathbb{H}}$
and $\langle \cdot,\cdot\rangle:=\langle \cdot ,\cdot\rangle_{\mathbb{H}}$, respectively. For $r\in[1,\infty]$, we may write $\|\cdot\|_r:=\|\cdot\|_{L^r(\mathbb{R}^d)}$.
For a Banach space $(X,\|\cdot\|_{X})$ and its dual space $X^*$, we use $_{X^*}\langle \cdot,
\cdot\rangle_{X}$ to denote the
duality product of $X$ and $X^*$. In addition, we also introduce a $X$-valued sequence space   $\ell^r(\mathbb{N},X):=\big\{u=(u_k)_{k\in \mathbb{N}} : u_k\in X\ \mbox{and}\ \sum_{k\in \mathbb{N}}\|u_k\|_X^r<+\infty\big\}$ with $r\in[1,\infty)$, which is a Banach space equipped with the norm $\|u\|_{\ell^r(\mathbb{N},X)}
=\big(\sum_{k\in \mathbb{N}}\|u_k\|_X^r\big)^{1/r}$.
Note that $\ell^{r_1}(\mathbb{N},X)\subseteq
\ell^{r_2}(\mathbb{N},X)$ for $1\leq r_1\leq r_2<\infty$. If $(X,\|\cdot\|_{X})=
(\mathbb{R},|\cdot|)$,
then we  write $\ell^r:=\ell^r(\mathbb{N},\mathbb{R})$.

\subsection{$L^r$-Wasserstein distance of probability measures}
Let $\mathcal{P}(\mathbb{H})$ be the space of all probability measures on $\mathbb{H}$ equipped with the weak topology,
which is
a
Polish space.
Given $r\in[1,\infty)$, we consider a subspace $\mathcal{P}_r(\mathbb{H})$ of $\mathcal{P}(\mathbb{H})$:
$$\mathcal{P}_r(\mathbb{H}):
=\Bigg\{ \mu\in  \mathcal{P}(\mathbb{H}):  \mu\big(\|\cdot\|^r\big)
:=\int_{\mathbb{H}}\|x\|^r\mu(dx) <\infty \Bigg\}.$$
It follows from Villani \cite[Definition 6.1]{Villani} that $\mathcal{P}_r(\mathbb{H})$ is also a
 Polish space with the $L^r$-Wasserstein distance, that is, for $\mu_1,\mu_2\in\mathcal{P}_r(\mathbb{H})$,
 \begin{align}\label{Wasserstein}
 d_{\mathcal{P}_r}
 (\mu_1,\mu_2):&=\Bigg(
 \inf_{\mu\in \mathcal{C}(\mu_1,\mu_2)}
 \int_{\mathbb{H}\times\mathbb{H}}\|x-y
 \|^r\mu(dx,dy)\Bigg)^{1/r},
\end{align}
where $\mathcal{C}(\mu_1,\mu_2)$ is
the set of all couplings for  $\mu_1$ and $\mu_2$.
Note  that
 \begin{align}\label{Wassersteinaa}
 d_{\mathcal{P}_r}
 (\mu_1,\mu_2)
= \inf_{\mathcal{L}_{\psi_1}=\mu_1,\        \mathcal{L}_{\psi_2}=\mu_2}
\Big(\mathbf{E}\big[\|\psi_1-
\psi_2\|^r\big]\Big)^{1/r},
\end{align}
where $\psi_i$ is a random variable
in
$\mathbb{H}$
with law
$\mathcal{L}_{\psi_i}=\mu_i$ for $i=1,2$.

\subsection{Conditions on measure-dependent drift and diffusion terms}
We assume that the  nonlinear
drift and diffusion terms
in \eqref{Int19-} satisfy the
conditions:

Condition \textbf{A}. The
functions
$H_1,H_2:\mathbb{R}^+\times\mathbb{R}^d
\times\mathbb{R}\times
\mathcal{P}_2(\mathbb{H})
\rightarrow\mathbb{R}$ are
continuous such that
for all $t\in\mathbb{R}^+$, $x\in\mathbb{R}^d$,
$s,s_1,s_2\in\mathbb{R}$ and $\mu,\mu_1,\mu_2\in\mathcal{P}_2(\mathbb{H})$,
\begin{subequations} \begin{align}\label{f1}&
H_1(t,x,s,\mu)s\leq-\beta|s|^p+
\phi_1(t,x)
\mu\big(\|\cdot\|^2
\big)+
\phi_2(t,x),
\\\label{f2}&|H_1(t,x,s,\mu)
|\leq\phi_3(t,x)|s|^{p-1}+
\phi_4(t,x)
\sqrt{\mu\big(\|\cdot\|^2\big)}
+\phi_5(t,x),
\\\label{f3}&(H_1(t,x,s_1,\mu_1)
-H_1(t,x,s_2,\mu_2))(s_1-s_2)
\leq-\beta
\big(|s_1|^{p-2}s_1-|s_2|^{p-2}s_2\big)(s_1-s_2)
\nonumber\\&\ \ \ \ \ \ \ \ \ \ \ \ \ \ \ \ \ \ \ \ \ \ \ \ \ \ \ \ \ \ \ \ \ \ \ \ \ \ \ \ \   \ \ \ \ \ \ \ \ \ \ \ \ \ \ \ \ \ \ +
\phi_6(t,x)|s_1-s_2|^2+
\phi_7(t,x)
d_{\mathcal{P}_2}^2
(\mu_1,\mu_2),
\\\label{f4}&|H_1(t,x,s_1,\mu_1)
-H_1(t,x,s_2,\mu_2)|
\leq\gamma (\phi_8(t,x) +   |s_1|^{p-2}+|s_2|^{p-2})|s_1-s_2|+
\phi_9(t,x)
d_{\mathcal{P}_2}
(\mu_1,\mu_2),
\end{align}
\end{subequations}
and
\begin{subequations}\begin{align}\label{ff1}&
H_2(t,x,s,\mu)s\leq-\widehat{\beta}|s|^q+
\widehat{\phi}_1(t,x)
\mu\big(\|\cdot\|^2
\big)+
\widehat{\phi}_2(t,x),
\\\label{ff2}&|H_2(t,x,s,\mu)
|\leq\widehat{\phi}_3(t,x)|s|^{q-1}+
\widehat{\phi}_4(t,x)
\sqrt{\mu\big(\|\cdot\|^2\big)}
+\widehat{\phi}_5(t,x),
\\\label{ff3}&(H_2(t,x,s_1,\mu_1)
-H_2(t,x,s_2,\mu_2))(s_1-s_2)
\leq-\widehat{\beta}
\big(|s_1|^{q-2}s_1-|s_2|^{q-2}
s_2\big)(s_1-s_2)\nonumber\\&\ \ \ \ \ \ \ \ \ \ \ \ \ \ \ \ \ \ \ \ \ \ \ \ \ \ \ \ \ \ \ \ \ \ \ \ \ \ \ \ \ \   \ \ \ \ \ \ \ \ \ \ \ \ \ \ \ \ \ \ +
\widehat{\phi}_6(t,x)|s_1-s_2|^2+
\widehat{\phi}_7(t,x)
d_{\mathcal{P}_2}^2
(\mu_1,\mu_2),
\\\label{ff4}&|H_2(t,x,s_1,\mu_1)
-H_2(t,x,s_2,\mu_2)|
\leq \widehat{\gamma}(\widehat{\phi}_8(t,x) +   |s_1|^{q-2}+|s_2|^{q-2})|s_1-s_2|+
\widehat{\phi}_9(t,x)
d_{\mathcal{P}_2}
(\mu_1,\mu_2),
\end{align}\end{subequations}where
$\beta>0$, $\widehat{\beta}>0$, $\gamma>0$, $\widehat{\gamma}>0$,
 $p>2$,
$q>2$, $1/p+1/\widehat{p}=1$ and $1/q+1/\widehat{q}=1$, and
 $\phi_1,\widehat{\phi}_1
 \in  L_{loc}^\infty(\mathbb{R}^+,
 L^1(\mathbb{R}^d))$, $\phi_2,\widehat{\phi}_2\in L_{loc}^1(\mathbb{R}^+,
   L^1(\mathbb{R}^d))$, $
 \phi_3,\widehat{\phi}_3\in  L_{loc}^\infty(\mathbb{R}^+,
 L^\infty(\mathbb{R}^d))$,
   $\phi_4\in L_{loc}^{\infty}(\mathbb{R}^+,L^{\widehat{p}}
(\mathbb{R}^d))$, $\widehat{\phi}_4\in L_{loc}^{\infty}(\mathbb{R}^+,
L^{\widehat{q}}
(\mathbb{R}^d))$, $\phi_5\in L_{loc}^{\widehat{p}}(\mathbb{R}^+,
 L^{\widehat{p}}(\mathbb{R}^d))$,
 $\widehat{\phi}_5\in L_{loc}^{\widehat{q}}(\mathbb{R}^+,
 L^{\widehat{q}}(\mathbb{R}^d))$, $\phi_6,\widehat{\phi}_6\in L_{loc}^1(\mathbb{R}^+,
 L^{\infty}(\mathbb{R}^d))$, $\phi_7,
 \widehat{\phi}_7\in L_{loc}^{1}(\mathbb{R}^+,
 L^{1}(\mathbb{R}^d))$, $\phi_8\in L_{loc}^{\infty}(\mathbb{R}^+, L^{\frac{p}{p-2}}(\mathbb{R}^d))$,
 $\widehat{\phi}_8\in L_{loc}^{\infty}(\mathbb{R}^+, L^{\frac{q}{q-2}}(\mathbb{R}^d))$,
 $\phi_9\in L_{loc}^{\infty}(\mathbb{R}^+, L^{\widehat{p}}(\mathbb{R}^d))$ and
 $\widehat{\phi}_9\in L_{loc}^{\infty}(\mathbb{R}^+, L^{\widehat{q}}(\mathbb{R}^d))$ are all nonnegative functions.

Condition \textbf{B}. For every $k\in \mathbb{N}$,
$\sigma_k:\mathbb{R}^+\times \mathbb{R}^d\times\mathbb{R}\times
\mathcal{P}_2(\mathbb{H})\rightarrow \mathbb{R}$ is
measurable which is
   locally Lipschitz continuous
 in the last two arguments and
  has a \emph{superlinear} growth in the third argument, that is, for
  all $t\in\mathbb{R}^+$,
$x\in\mathbb{R}^d$, $s,s_1,s_2\in\mathbb{R}$ and $\mu,\mu_1,\mu_2
\in\mathcal{P}_2(\mathbb{H})$,
\begin{align}\label{h2}
|\sigma_k(t,x,s,\mu)|^2\leq \sigma_{1,k}(x)
|s|^{p^*}+\sigma_{2,k}(x)
|s|^{q^*}
+\sigma_{3,k}(x)\big(1+\mu\big(\|\cdot\|^2)\big),
\end{align}
and
\begin{align}\label{h2a}
|\sigma_k(t,x,s_1,\mu_1)-\sigma_k(t,x,s_2,\mu_2)
|^2 \leq
& \sigma^2_{4,k}(x)\big( 1+ |s_1|^{p^*-2}+ |s_2|^{p^*-2} +|s_1|^{q^*-2}
   +|s_2|^{q^*-2} \big)|s_1-s_2|^2 \nonumber\\
& +\sigma_{5,k}(x)
d_{\mathcal{P}_2}^2(\mu_1,\mu_2),
\end{align}
where  $p^*\in[2,p)$, $q^*\in[2,q)$,  $\sigma_1\in
 \ell^{1}
 (\mathbb{N},L^{\frac{p}{p-p^*}}(\mathbb{R}^d))$,
$\sigma_2
 \in
 \ell^{1}
 (\mathbb{N},L^{\frac{q}{q-q^*}}(\mathbb{R}^d))$,
 $\sigma_3,\sigma_5
 \in
 \ell^{1}
 (\mathbb{N},L^1(\mathbb{R}^d))$, and $\sigma_4\in
 \ell^{1}
 (\mathbb{N},L^{\infty}
 (\mathbb{R}^d))   $.

\subsection{Algebraic inequalities}
We recall several algebraic inequalities from \cite{RWangsubmitted} (see also \cite{Brasco1}).
By utilizing these algebraic inequalities, we are able to control the superlinear
diffusion terms by using the  polynomial dissipativeness of the drift terms when proving the global well-posedness and LPDs of \eqref{Int19-}.

\begin{lemma}\label{AlgebraicA}
For any $b_1,b_2\in\mathbb{R}$, we have

(i) Upper bounds:
\begin{subequations}
\begin{align}&\label{Algebraic1}\big||b_1|^{r-2}b_1
		-|b_2|^{r-2}b_2\big|\leq (r-1)\big(|b_1|
		+|b_2|\big)^{r-2}|b_1-b_2|, \ \  \forall\ r\in[2,\infty),
\\&\label{Algebraic2}\big||b_1|^{r-2}b_1
		-|b_2|^{r-2}b_2\big|\leq \frac{2}{4^{r-2}} \big(|b_1|
		+|b_2|\big)^{r-2}|b_1-b_2|, \ \  \forall \ b_1\neq0, \ b_2 \neq 0, \ r\in(1,2],
\\&\label{Algebraic2a}\big||b_1|^{r-2}b_1
		-|b_2|^{r-2}b_2\big|\leq \frac{2}{4^{r-2}} |b_1-b_2|^{r-1}, \ \
		\forall \
		 b_1\neq0, \ b_2 \neq 0,
		 \
		  r\in(1,2].
\end{align}
\end{subequations}

(ii) Lower bounds:
\begin{subequations}
\begin{align}\label{Algebraic5}
&(|b_1|^{r-2}b_1-|b_2|^{r-2}b_2)(b_1-b_2)
\geq\frac{1}{2}(|b_1|^{r-2}+|b_2|^{r-2})
(b_1-b_2)^2,
\  \  \forall \  r\in[2,\infty),
\\&\label{Algebraic6}(|b_1|^{r-2}b_1-|b_2|
^{r-2}b_2)(b_1-b_2)
\geq \left \{
\begin{array}{ll}
 2^{2-r}|b_1-b_2|^r,  & \  \forall\ r\in[3,\infty),\\
  2^{1-r}|b_1-b_2|^r,  & \   \forall\ r\in[2,\infty),
  \end{array}
  \right.
\\&\label{Algebraic7}(|b_1|^{r-2}b_1-
|b_2|^{r-2}b_2)(b_1-b_2)
\geq
\left \{
\begin{array}{ll}
2^{1-r}|b_1-b_2|^r
+\frac{1}{4}(|b_1|^{r-2}+|b_2|^{r-2})(b_1-b_2)^2,
 &  \forall\ r\in[3,\infty),\\
 2^{-r}|b_1-b_2|^r
+\frac{1}{4}(|b_1|^{r-2}+|b_2|^{r-2})(b_1-b_2)^2,
 &  \forall\ r\in[2,\infty).
 \end{array}
 \right.
\end{align}
\end{subequations}

\end{lemma}

\subsection{Variational settings}
In this part, we establish several variational settings in order to discuss the
global-in-time well-posedness of \eqref{Int19-}.
We first establish some variational
 properties for the generalized fractional $(\alpha, p)$-Laplace operator $\mathfrak{L}_{\mathcal{K}_{p}^\alpha}$.
Given $t\in\mathbb{R}^+$ and
$\mu\in\mathcal{P}_2
(\mathbb{H})$, we
define two operators $\mathbf{A}_1(t,\cdot,\mu): \mathbb{V}_1\rightarrow
\mathbb{V}_1^*$ and $\mathbf{A}_2(t,\cdot,\mu): \mathbb{V}_2\rightarrow \mathbb{V}_2^*$ by
 \begin{align*}&
_{\mathbb{V}_1^*}\big\langle \mathbf{A}_1(t,\psi_1,\mu),
\psi_2\big  \rangle_{\mathbb{V}_1}
=\int_{\mathbb{R}^d}
H_1(t,x,\psi_1(x),\mu)\psi_2(x)dx
\\&\ \ \ -\frac{1}{2}\int_{\mathbb{R}^d}  \int_{\mathbb{R}^d}|\psi_1(x)-\psi_1(y)|^{p-2}\big
(\psi_1(x)-\psi_1(y)\big)
 \big( \psi_2(x)-\psi_2(y)\big)\mathcal{K}_{p}^\alpha(x,y)dxdy,\ \ \psi_1,\psi_2\in \mathbb{V}_1,
 \end{align*}
 and
\begin{align*}_{\mathbb{V}_2^*}
\langle \mathbf{A}_2(t,\psi_1,\mu) ,\psi_2\rangle_{\mathbb{V}_2}
=\int_{\mathbb{R}^d}H_2(t,x,\psi_1(x),\mu)
\psi_2(x)dx,\ \ \ \psi_1,\psi_2\in\mathbb{V}_2.
\end{align*}

Next, we discuss several useful properties for the nonlinear operators $\mathbf{A}_1(t,\cdot,\mu)$ and $\mathbf{A}_2(t,\cdot,\mu)$.

\begin{lemma}(\texttt{Dissipativeness and Growth})\label{Dissipativeness}
Let condition \textbf{A} be satisfied.
Then we have:

 (i) The operator $\mathbf{A}_1(t,\cdot,\mu):\mathbb{V}_1\rightarrow
\mathbb{V}_1^*$ is well-defined, and for all $t\in\mathbb{R}^+$,
$\mu\in\mathcal{P}_2
(\mathbb{H})$ and $\psi\in \mathbb{V}_1$,
\begin{align}\label{AAoperator1a}&
_{\mathbb{V}_1^*}\big\langle \mathbf{A}_1(t,\psi,\mu),
\psi\big\rangle_{\mathbb{V}_1}
\leq-\frac{1}{2K}\| \psi\|_{\dot{\mathbb{V}}_1}^p-
\beta\| \psi\|_p^{p} +\|\phi_1(t)
\|_1
\mu\big(\|\cdot\|^2\big)
+\|\phi_2(t)
\|_{1},
\end{align}
and
\begin{align}&\label{AAoperator1}
\|\mathbf{A}_1(t,\psi,\mu) \|_{\mathbb{V}^*_1}
\le\frac{K}{2}\| \psi
\|_{\dot{\mathbb{V}}_1
 }^{p-1}+\|\phi_3(t)
\|_\infty
\| \psi\|_p^{p-1}+\|\phi_4(t)
\|_{\widehat{p}}
\sqrt{\mu\big(\|\cdot\|^2\big)}
+\|\phi_5(t)
\|_{\widehat{p}}.
\end{align}

(ii) The operator $\mathbf{A}_2(t,\cdot,\mu):
\mathbb{V}_2\rightarrow
\mathbb{V}_2^*$ is well-defined, and for all $t\in\mathbb{R}^+$,
$\mu\in\mathcal{P}_2
(\mathbb{H})$ and $\psi\in \mathbb{V}_2$,
\begin{align}
&\label{aAAoperator1a}
_{\mathbb{V}^*_2}\big\langle \mathbf{A}_2(t,\psi,\mu),
\psi\big\rangle_{\mathbb{V}_2}
\leq-\widehat{\beta}\| \psi\|_{\mathbb{V}_2}^{q}+\|\widehat{\phi}_1(t)
\|_{1}
\mu\big(\|\cdot\|^2\big)
+\|\widehat{\phi}_2(t)
\|_{1},
\end{align}
and
\begin{align}
&\label{aAAAoperator1}
\|\mathbf{A}_2(t,\psi,\mu) \|_{\mathbb{V}^*_2}
\le \|\widehat{\phi}_3(t)
\|_{\infty}
\| \psi\|_{\mathbb{V}_2}^{q-1}+\|\widehat{\phi}_4(t)
\|_{\widehat{q}}
\sqrt{\mu\big(\|\cdot\|^2\big)}
+\|\widehat{\phi}_5(t)
\|_{\widehat{q}}.
\end{align}

\end{lemma}

\begin{proof}
By \eqref{f1} and the definition of $\mathbf{A}_1(t,\cdot,\mu)$, we have \eqref{AAoperator1a}.
By \eqref{f2} and the H\"{o}lder
inequality, we derive that for all $\psi_1,\psi_2\in L^{p}
(\mathbb{R}^d)$,
\begin{align*}&
\int_{\mathbb{R}^d}
H_1(t,x,\psi_1(x),\mu)\psi_2(x)dx\leq \Big(\|\phi_3(t)
\|_{\infty}
\|\psi_1\|^{p-1}_p
+\|\phi_4(t)
\|_{\widehat{p}}
\sqrt{\mu\big(\|\cdot\|^2\big)}
+
\|\phi_5(t)
\|_{\widehat{p}}\Big)
\|\psi_2\|_p.
  \end{align*}
By the definition of $\mathcal{K}^\alpha_{p}$ and the H\"{o}lder
inequality, we find that for all $\psi_1,\psi_2\in \dot{\mathbb{V}}_1$,
 \begin{align*}&
\frac{1}{2}\int_{\mathbb{R}^d}  \int_{\mathbb{R}^d}|\psi_1(x)
-\psi_1(y)|^{p-2}\big
(\psi_1(x)-\psi_1(y)\big)
 \big( \psi_2(x)-\psi_2(y)\big)
 \mathcal{K}_{p}^\alpha(x,y)dxdy
 \leq \frac{K}{2}\| \psi_1  \|_{\dot{\mathbb{V}}_1}^{p-1}\| \psi_2\|_{\dot{\mathbb{V}}_1}.
 \end{align*}
Above two inequalities implies  \eqref{AAoperator1}. By
\eqref{ff1}-\eqref{ff2} and the definition of $\mathbf{A}_2(t,\cdot,\mu)$, we can prove (ii).
\end{proof}

\begin{lemma}(\texttt{Strong monotonicity})
\label{Monotonicity}
Let condition \textbf{A} be satisfied. Then we have:

(i)The operator $\mathbf{A}_1(t,\cdot,\cdot): \mathbb{V}_1\times\mathcal{P}_2
(\mathbb{H})\rightarrow
\mathbb{V}^*_1$ is strongly monotone, that is, for all $t\in\mathbb{R}^+$, $\psi_1,\psi_2\in \mathbb{V}_1$ and $\mu_1,\mu_2\in\mathcal{P}_2
(\mathbb{H})$,
\begin{align*}
& _{\mathbb{V}^*_1}
\big\langle \mathbf{A}_1(t,\psi_1,\mu_1)-
\mathbf{A}_1(t,\psi_2,\mu_2),
\psi_1-\psi_2\big\rangle_{\mathbb{V}_1}
\leq -\frac{1}{2^{p+1}K}\|\psi_1-\psi_2
\|^p_{\dot{\mathbb{V}}_1}
 -2^{-p}\beta\|\psi_1-\psi_2\|^p_p
\nonumber\\& -\frac{1}{8K} \int_{\mathbb{R}^d}\int_{\mathbb{R}^d}
\big( |\psi_1(x)-\psi_1(y)|^{p-2}
+|\psi_2(x)-\psi_2(y)|^{p-2}  \big)
|\psi_1(x)-\psi_1(y)
-(\psi_2(x)-\psi_2(y))|^{2}|x-y|^{-(d+\alpha p)}dxdy
\\&\ \ -\frac{\beta}{4} \int_{\mathbb{R}^d}  \big(|\psi_1(x)|^{p-2}+|\psi_2(x)|^{p-2}\big)
|\psi_1(x)-\psi_2(x)|^2dx
+\|\phi_6(t)\|_{\infty}
\|\psi_1-\psi_2\|^2+
\|\phi_7(t)
\|_{1}
d_{\mathcal{P}_2}^2
(\mu_1,\mu_2).
\end{align*}

(ii)The operator $\mathbf{A}_2(t,\cdot,\cdot): \mathbb{V}_2\times\mathcal{P}_2
(\mathbb{H})\rightarrow
\mathbb{V}^*_2$
is strongly monotone, that is, for all   $t\in\mathbb{R}^+$, $\psi_1,\psi_2\in \mathbb{V}_2$ and $\mu_1,\mu_2\in\mathcal{P}_2
(\mathbb{H})$,
\begin{align*}&
_{\mathbb{V}^*_2}\langle \mathbf{A}_2(t,\psi_1,\mu_1)
-\mathbf{A}_2(t,\psi_2,\mu_2) ,\psi_1-\psi_2\rangle_{\mathbb{V}_2}
 \leq-2^{-q}\widehat{\beta}
\|\psi_1-\psi_2\|_{\mathbb{V}_2}^{q} \\& -\frac{\widehat{\beta}}{4}\int_{\mathbb{R}^d}  \big(|\psi_1(x)|^{q-2}+|\psi_2(x)|^{q-2}\big)
|\psi_1(x)-\psi_2(x)|^2dx
  +\|\widehat{\phi}_6(t)\|_{\infty}
\|\psi_1-\psi_2\|^2+
\|\widehat{\phi}_7(t)
\|_{1}
d_{\mathcal{P}_2}^2
(\mu_1,\mu_2).
\end{align*}
\end{lemma}
\begin{proof}
By \eqref{Algebraic7}, \eqref{f3},
\eqref{ff3} and the definitions of $\mathbf{A}_1(t,\cdot,\mu)$ and  $\mathbf{A}_2(t,\cdot,\mu)$, we can complete
the proof.
\end{proof}

\begin{lemma}(\texttt{Hemicontinuity})
\label{Hemicontinuity}
Let condition \textbf{A} hold. Then we have:

(i)The map $\mathbb{R}\ni s\mapsto\
_{\mathbb{V}^*_1}\big\langle \mathbf{A}_1(t,\psi_1+s\psi_2,
\mathcal{L}_{
\psi_1 +s \psi_2}) ,\widehat{\psi}\big\rangle_{\mathbb{V}_1}$ is,
$\mathbf{P}$-almost surely,   continuous
 for all $t\in\mathbb{R}^+$,
  $\psi_1,\psi_2\in
  L^p(\Omega, \mathbb{V}_1)
  \bigcap
  L^2(\Omega, \mathbb{H})$
  and
    $\widehat{\psi} \in
    L^p(\Omega,  \mathbb{V}_1)$.

(ii)The map $\mathbb{R} \ni s\mapsto\
_{\mathbb{V}^*_2}\big\langle \mathbf{A}_2(t,
\psi_1+s\psi_2,
\mathcal{L}_{\psi_1
+s \psi_2}) ,\widehat{\psi}\big\rangle_{\mathbb{V}_2}$ is,
$\mathbf{P}$-almost surely,   continuous
 for all $t\in\mathbb{R}^+$,
  $\psi_1,\psi_2\in
  L^q(\Omega, \mathbb{V}_2)
  \bigcap
  L^2(\Omega, \mathbb{H})$
  and
    $\widehat{\psi} \in
    L^q(\Omega,  \mathbb{V}_2)$.
\end{lemma}

\begin{proof}
We only show (ii) and the proof of (i) is
similar.
 Let $s\in\mathbb{R}$ and $s_n \to s$. Since
 $\psi_1,\psi_2\in
  L^2(\Omega, \mathbb{H})$, we find that
  $\psi_1+s_n  \psi_2
  \to \psi_1+s  \psi_2 $ in
  $L^2(\Omega, \mathbb{H})$,
and hence
  $\mathcal{L}_{
  \psi_1+s_n  \psi_2}
  \to
  \mathcal{L}_{
  \psi_1+s  \psi_2}$ in $\mathcal{P}_2 (\mathbb{H})$.
  Then by the continuity of
  $H_2$ we obtain that for  $t\in \mathbb{R}^+$
  and
  $x\in \mathbb{R}^d$,
  $\mathbf{P}$-almost surely,
\be\label{wanga}
  H_2 (t,x , \psi_1(x) +s_n  \psi_2(x),
  \mathcal{L}_{\psi_1+s_n  \psi_2})
  \to    H_2 (t,x , \psi_1(x)+s  \psi_2(x)  ),
  \mathcal{L}_{\psi_1+s  \psi_2 }).
\ee
On the other hand, by \eqref{ff2} we have, for all $t\in\mathbb{R}^+$,
  $\psi_1,\psi_2\in
  L^q(\Omega, \mathbb{V}_2)
  \bigcap
  L^2(\Omega, \mathbb{H})$
  and
    $\widehat{\psi} \in
    L^q(\Omega,  \mathbb{V}_2)$,
$$
\left | H_2 (t,x , \psi_1(x) +s_n  \psi_2(x),
  \mathcal{L}_{\psi_1+s_n  \psi_2})
  \widehat{\psi} (x)\right |
  $$
  $$
  \le
  \left (
   \widehat{\phi}_3 (t,x)
   |\psi_1(x) +s_n  \psi_2(x) |^{q-1}
   +
     \widehat{\phi}_4 (t,x)
     \sqrt{\mathbf{E}
     (\| \psi_1 +s_n  \psi_2\|^2)}
     +   \widehat{\phi}_5 (t,x)
  \right ) |   \widehat{\psi} (x)|
  $$
   $$
  \le
   {\frac 1{\widehat{q}}}
   |\widehat{\phi}_3 (t,x)|^{\widehat{q}}
   |\psi_1(x) +s_n  \psi_2(x) |^{q}
   +  {\frac 1{\widehat{q}}}
    | \widehat{\phi}_4 (t,x)|^{\widehat{q}}
    (  \mathbf{E}
       (\| \psi_1 +s_n  \psi_2\|^2) ) ^{\frac {\widehat{q}}2}
     +    {\frac 1{\widehat{q}}}
 |   \widehat{\phi}_5 (t,x)|^{\widehat{q}}
  +    {\frac 3{ {q}}}
   |   \widehat{\psi} (x)|^q
  $$
  $$
  \le
   c_1
   |\widehat{\phi}_3 (t,x)|^{\widehat{q}}
   (
   |\psi_1(x)|^q + | \psi_2(x) |^{q})
   + c_1
    | \widehat{\phi}_4 (t,x)|^{\widehat{q}}
    ( \mathbf{E}
     (\| \psi_1\|^2
      + \| \psi_2\|^2 ) ) ^{\frac {\widehat{q}}2}
     +   c_1
 |   \widehat{\phi}_5 (t,x)|^{\widehat{q}}
  +  c_1
   |   \widehat{\psi} (x)|^q,
  $$
  where $c_1>0$ is a number depending on $q$
  but not on $n$,
  which along with \eqref{wanga}
  the Lebesgue dominated convergence theorem
  yields (ii).
   \end{proof}

By the argument of  Lemma \ref{Hemicontinuity},
we also have the following results, which will be used
when proving the strong continuity of solutions
of the controlled equation with respect to
the weak topology of
control functions.

\begin{corollary}(\texttt{Hemicontinuity})
\label{cHemicontinuity}
Let condition \textbf{A} hold. Then we have:

(i)The map $\mathbb{R}\ni s\mapsto\
_{\mathbb{V}^*_1}\big\langle \mathbf{A}_1(t,\psi_1+s\psi_2,
\mu  ) ,\widehat{\psi}\big\rangle_{\mathbb{V}_1}$ is
   continuous
 for all $t\in\mathbb{R}^+$,
 $\mu \in \mathcal{P}_2 (\mathbb{H})$ and
  $\psi_1,\psi_2, \widehat{\psi} \in
   \mathbb{V}_1 $.

(ii)The map $\mathbb{R} \ni s\mapsto\
_{\mathbb{V}^*_2}\big\langle \mathbf{A}_2(t,
\psi_1+s\psi_2, \mu ) ,
\widehat{\psi}\big\rangle_{\mathbb{V}_2}$ is
 continuous
 for all $t\in\mathbb{R}^+$,
  $\mu \in \mathcal{P}_2 (\mathbb{H})$ and
  $\psi_1,\psi_2 ,
   \widehat{\psi} \in
    \mathbb{V}_2$.
\end{corollary}

\begin{lemma}(\texttt{Demicontinuity})
\label{Demicontinuity}
Let condition \textbf{A} hold. Then we have:

(i)The map $ \mathcal{P}_2
(\mathbb{H})\times \mathbb{V}_1 \ni(\psi, \mu )\mapsto\
_{\mathbb{V}^*_1}\big\langle \mathbf{A}_1(t,\psi, \mu) ,\widehat{\psi}\big\rangle_{\mathbb{V}_1}$ is continuous
for all $t\in\mathbb{R}^+$ and
$\widehat{\psi}\in \mathbb{V}_1$.

(ii)The map $\mathcal{P}_2
(\mathbb{H})\times \mathbb{V}_2\ni(\psi, \mu)\mapsto\
_{\mathbb{V}^*_2}\big\langle \mathbf{A}_2(t,\psi, \mu) ,\widehat{\psi}\big\rangle_{\mathbb{V}_2}$ is continuous
for all $t\in\mathbb{R}^+$ and
$\widehat{\psi}\in \mathbb{V}_2$.
\end{lemma}
\begin{proof}
Given $t\in\mathbb{R}^+$,
$\mu_1,\mu_2\in\mathcal{P}_2
(\mathbb{H})$ and
$\psi,\psi_1,\psi_2\in \mathbb{V}_1$, we
consider
$_{\mathbb{V}_1^*}\big\langle \mathbf{A}_1(t,\mu_1,\psi_1)-
\mathbf{A}_1(t,\mu_2,\psi_2) ,\psi\big\rangle_{\mathbb{V}_1}
=\mathcal{I}_1 -\mathcal{I}_2,$
where $\mathcal{I}_1:=\int_{\mathbb{R}^N}
\big(H_1(t,x,\mu_1,\psi_1(x))-
H_1(t,x,\mu_2,\psi_2(x))\big)\psi(x)dx$, and
 \begin{align*}&
\mathcal{I}_2:=\frac{1}{2}
\int_{\mathbb{R}^d}  \int_{\mathbb{R}^d}
\Big(|\psi_1(x)-\psi_1(y)|^{p-2}\big
(\psi_1(x)-\psi_1(y)\big) \\&\ \ \ \ \ \  \ \ \ \ \ \  \ \ \  \ \ \ \ \ \  \ \ \ -|\psi_2(x)-\psi_2(y)|^{p-2}\big
(\psi_2(x)-\psi_2(y)\big)     \Big)
 \big( \psi(x)-\psi(y)\big)
 \mathcal{K}_{p}^\alpha(x,y)
dxdy.
 \end{align*}
By condition \eqref{f4} and the H\"{o}lder
inequality, it follows that
 \begin{align*}
\mathcal{I}_1&\leq\gamma
\Big(\|\phi_8(t) \|_{\frac{p}{p-2}} +\|\psi_1 \|^{p-2}_{p} +\|\psi_2 \|^{p-2}_{p}        \Big)
\|\psi
\|_{p} \|\psi_1
-\psi_2\|_{p}
+\|\phi_9(t) \|_{\widehat{p}}  \|\psi
\|_{p}
d_{\mathcal{P}_2}
(\mu_1,\mu_2).
 \end{align*}
By the H\"{o}lder
inequality and \eqref{Algebraic1}, we see that
\begin{align*}
\mathcal{I}_2&\leq(p-1)2^{p-3}K\int_{\mathbb{R}^d}  \int_{\mathbb{R}^d}\big(|\psi_1(x)-\psi_1(y)|^{p-2}
+|\psi_2(x)-\psi_2(y)|^{p-2}\big)
\\& \ \ \  \ \ \times\big|\psi_1(x)-\psi_1(y)-
\big(\psi_2(x)-\psi_2(y)\big)\big| |\psi(x)-\psi(y)|
|x-y|^{-(d+\alpha p)}dxdy
\\&\leq(p-1)2^{p-3}K\Bigg(\bigg(
\int_{\mathbb{R}^d}  \int_{\mathbb{R}^d}|x-y|^{-(d+\alpha p)}|\psi_1(x)-\psi_1(y)|^{p}\bigg)^{(p-2)/p}\\& \ \ \  \ \ +
\bigg(
\int_{\mathbb{R}^d}  \int_{\mathbb{R}^d}|x-y|^{-(d+\alpha p)}|\psi_2(x)-\psi_2(y)|^{p}\bigg)^{(p-2)/p}\Bigg)
\\& \ \ \  \ \ \times
\bigg(\int_{\mathbb{R}^d}
\int_{\mathbb{R}^d}|\psi(x)-
\psi(y)|^p|x-y|^{-(d+\alpha p)}
dxdy \bigg)^{1/p}
\\& \ \ \  \ \ \times
\bigg(\int_{\mathbb{R}^d}
\int_{\mathbb{R}^d}|\psi_1(x)-\psi_1(y)-
\big(\psi_2(x)-\psi_2(y)\big)|^p|x-y|^{-(d+\alpha p)}
dxdy \bigg)^{1/p}
\\&=(p-1)2^{p-3}K\Big(   \|\psi_1\|^{p-2}_{\dot{\mathbb{V}}_1}     +  \|\psi_2\|^{p-2}_{\dot{\mathbb{V}}_1}
\Big)
\|\psi\|_{\dot{\mathbb{V}}_1}
\|\psi_1-\psi_2\|_{\dot{\mathbb{V}}_1}.
 \end{align*}
Above results imply (i). The proof of (ii) is similar to that of (i), and we will
not repeat it again.
\end{proof}

We now consider the diffusion terms in \eqref{Int19-}. Define $$\sigma(t,x,s,\mu):=\{\sigma_{k}(t,x, s,\mu)
\}_{k\in\mathbb{N}},\ \ \forall\ t\in\mathbb{R}^+, \ x\in\mathbb{R}^d,\ s\in\mathbb{R},\ \mu
\in\mathcal{P}_2(\mathbb{H}).$$
Then by \eqref{h2} and \eqref{h2a}, one can prove that
for every $M>0$, $\varepsilon_1>0$, $\varepsilon_2>0$ and $T>0$, and  for
any $t\in[0,T]$, $\psi,\psi_1,\psi_2\in
\mathbb{H}\cap L^p(\mathbb{R}^d)\cap \mathbb{V}_2$ and $\mu,\mu_1,\mu_2
\in\mathcal{P}_2(\mathbb{H})$,
\begin{align}\label{h2-}
&M\|\sigma(t,\cdot,\psi,\mu)\|^2_{
\ell^2(\mathbb{N},
\mathbb{H})}
=M\sum_{k\in\mathbb{N}}\int_{\mathbb{R}^d}
|\sigma_k(t,x,\psi(x),\mu)|^2dx
\nonumber\\
&\leq
M\sum_{k\in\mathbb{N}}\int_{\mathbb{R}^d}
\Big(
\sigma_{1,k}(x)
|\psi(x)|^{p^*}+
\sigma_{2,k}(x)
|\psi(x)|^{q^*}
 +
\sigma_{3,k}(x)
\big(1+\mu\big(\|\cdot\|^2\big)\big)\Big)dx
\nonumber\\
&\leq
 M
 \Bigg(
 \sum_{k\in\mathbb{N}}\|\sigma_{1,k} \|
 _{\frac p{p-p^*}}
 \| \psi\|^{p^*}_{p}
 +
 \sum_{k\in\mathbb{N}}\|\sigma_{2,k}\|
 _{\frac q{q-q^*} }
  \| \psi\|^{q^*}_{q}
 +
 \sum_{k\in\mathbb{N}}\|\sigma_{3,k}\|
 _{1}\big(1+\mu\big(\|\cdot\|^2
\big)\big)\Bigg),
\nonumber\\
&\le
\varepsilon_1
\|\psi\|_p^{p}+\varepsilon_2
\|\psi\|_{\mathbb{V}_2}^{q}\
+C_2(M,\varepsilon_1,\varepsilon_2)
\big(1+\mu\big(\|\cdot\|^2
\big)\big),
\end{align}
and
\begin{align}\label{h2--}
&M\|\sigma(t,\cdot,\psi_1,\mu_1)-
\sigma(t,\cdot,\psi_2,\mu_2)
\|^2_{
\ell^2(\mathbb{N},
\mathbb{H})}
\nonumber\\
&=M\sum_{k\in\mathbb{N}}\int_{\mathbb{R}^d}
|\sigma_k(t,x,\psi_1(x),\mu_1)
-\sigma_k(t,x,\psi_2(x),\mu_2)|^2dx
\nonumber\\
&
\leq M\sum_{k\in\mathbb{N}}\int_{\mathbb{R}^d}
\Big(\sigma^2_{4,k}(x)\big(1+ |\psi_1(x)|^{p^*-2}+ |\psi_2(x)|^{p^*-2} +|\psi_1(x)|^{q^*-2}
\nonumber\\& \ \ +|\psi_2(x)|^{q^*-2}
\big)|\psi_1(x)-\psi_2(x)|^2 +\sigma_{5,k}(x)
d_{\mathcal{P}_2}^2(\mu_1,\mu_2)\Big)dx
\nonumber\\
&\leq
\sum_{k\in\mathbb{N}}
\Big(\|\sigma_{4,k}\|^2_{\infty}+
C_3(M,\varepsilon_1)
\|\sigma_{4,k}\|_{\infty}^{\frac{2p-p^*-2}{p-p^*}}
+\widehat{C}_3(M,\varepsilon_2)
\|\sigma_{4,k}\|_{\infty}^{\frac {2q-q^*-2}{q-q^*}}\Big)
\|\psi_1-\psi_2\|^2
\nonumber\\
&\  +M\sum_{k\in\mathbb{N}}
\|\sigma_{5,k}\|_{1}
d_{\mathcal{P}_2}^2(\mu_1,\mu_2)
\nonumber\\
&\  + \frac{\varepsilon_1}{1+ \sum_{k\in\mathbb{N}}
\|\sigma_{4,k}\|_{\infty}}\sum_{k\in\mathbb{N}}
\int_{\mathbb{R}^d}
\sigma_{4,k}(x)\big(|\psi_1(x)|^{p-2}+ |\psi_2(x)|^{p-2}\big)|\psi_1(x)-\psi_2(x)|^2dx
\nonumber\\
&\ + \frac{\varepsilon_2}{1+ \sum_{k\in\mathbb{N}}
\|\sigma_{4,k}\|_{\infty}}
\sum_{k\in\mathbb{N}}\int_{\mathbb{R}^d}
\sigma_{4,k}(x)\big(|\psi_1(x)|^{q-2}+ |\psi_2(x)|^{q-2}\big)|\psi_1(x)-\psi_2(x)|^2dx
\nonumber\\&\leq C_4(M,\varepsilon_1,\varepsilon_2)
\|\psi_1-\psi_2\|^2
 +C_5(M)
d_{\mathcal{P}_2}^2(\mu_1,\mu_2)
\nonumber\\
&\ +\varepsilon_1\int_{\mathbb{R}^d}
 \big(|\psi_1(x)|^{p-2} +|\psi_2(x)|^{p-2}
\big)|\psi_1(x)-\psi_2(x)|^2dx
\nonumber\\
&\  +\varepsilon_2\int_{\mathbb{R}^d}
 \big(|\psi_1(x)|^{q-2} +|\psi_2(x)|^{q-2}
\big)|\psi_1(x)-\psi_2(x)|^2dx,
\end{align}
where
$C_1(M,\varepsilon_1,\varepsilon_2)$ and
$C_2(M,\varepsilon_1,\varepsilon_2)$
are positive constants depending on $M$, $\varepsilon_1$, $\varepsilon_2$, $p$, $p^*$, $q$ and $q^*$, and
\begin{align*}&
C_3(M,\varepsilon_1):=\frac{2(p-p^*)}{p-2}
M^{\frac{p-2}{p-p^*}}
\bigg( \frac{(p-2)\varepsilon_1}{(p^*-2)
(1+ \sum_{k\in\mathbb{N}}
\|\sigma_{4,k}\|_{\infty}) }
\bigg)^{\frac{2-p^*}{p-p^*}},
\\&
\widehat{C}_3(M,\varepsilon_2):=
\frac{2(q-q^*)}{q-2}M^{\frac{q-2}{q-q^*}}
\bigg( \frac{(q-2)\varepsilon_2}{(q^*-2)
(1+ \sum_{k\in\mathbb{N}}
\|\sigma_{4,k}\|_{\infty}) }
\bigg)^{\frac{2-q^*}{q-q^*}},
\\&C_4(M,\varepsilon_1,\varepsilon_2):=
\sum_{k\in\mathbb{N}}
\Big(\|\sigma_{4,k}\|^2_{\infty}+
C_3(M,\varepsilon_1)
\|\sigma_{4,k}\|_{\infty}^{\frac{2p-p^*-2}{p-p^*}}
+\widehat{C}_3(M,\varepsilon_2)
\|\sigma_{4,k}\|_{\infty}^{\frac {2q-q^*-2}{q-q^*}}\Big),
\end{align*}
and $C_5(M):=M\sum_{k\in\mathbb{N}}
\|\sigma_{5,k}\|_{1}$. We remark that, from the derivation of \eqref{h2--}, it is easy to see that $\sigma_4\in
 \ell^{1}
 (\mathbb{N},L^{\infty}
 (\mathbb{R}^d))$ is enough since $(2p-p^*-2)/ (p-p^*)>1$ and $(2q-q^*-2)/(q-q^*)>1$.

Given $t\in\mathbb{R}^+$, $\psi\in\mathbb{H}\cap \mathbb{V}_1 \cap \mathbb{V}_2$ and $\mu
\in\mathcal{P}_2(\mathbb{H})$, we now define an operator
$\mathfrak{L}_{H,S}(t,\psi,\mu):\ell^2\rightarrow \mathbb{H}$ by
\begin{align}\label{jhjfjnfhh}
(\mathfrak{L}_{H,S}(t,\psi,\mu)v)(x)
= \sum_{k\in\mathbb{N}}\sigma_{k}(t,x, \psi(x),\mu)v_k
,\ \ \forall\ v=(v_k)_{k\in \mathbb{N}}\in \ell^2,\ \ x\in \mathbb{R}^d.
\end{align}
Then one can verify that the operator
$\mathfrak{L}_{H,S}(t,\psi,\mu):\ell^2\rightarrow \mathbb{H}$ is
 a Hilbert-Schmidt operator with the Hilbert-Schmidt norm:
\begin{align}\label{ffnewA2a0}
\|\mathfrak{L}_{H,S}(t,\psi,\mu)
\|^2_{\mathcal{L}_2(\ell^2
,\mathbb{H})}
 =\|\sigma(t,\cdot,\psi,\mu)
\|^2_{
\ell^2(\mathbb{N},
\mathbb{H})}.
\end{align}
 In addition, we also have, for any $t\in\mathbb{R}^+$, $\psi_1,\psi_2\in\mathbb{H}\cap \mathbb{V}_1 \cap \mathbb{V}_2$ and $\mu_1,\mu_2
\in\mathcal{P}_2(\mathbb{H})$,
\begin{align}\label{h2----0}
\|\mathfrak{L}_{H,S}(t,\psi_1,\mu_1)-
\mathcal{L}_{H,S}(t,\psi_2,\mu_2)
\|^2_{\mathcal{L}_2(\ell^2
,\mathbb{H})}=\|\sigma(t,\cdot,\psi_1,
\mu_1)-
\sigma(t,\cdot,\psi_2,\mu_2)
\|^2_{
\ell^2(\mathbb{N},
\mathbb{H})}.
\end{align}

By above notations, one can rewrite problem \eqref{Int19-} as the following stochastic system
\begin{equation}\label{Int1--}
\left\{
  \begin{aligned}
  &d \psi(t)
=\sum_{i=1}^2\textbf{A}_i(t,\psi(t),\mathcal{L}_{
\psi(t)})dt
+ \sqrt{\epsilon}\mathfrak{L}_{H,S}(t,\psi(t),
\mathcal{L}_{
\psi(t)})d\mathcal{W}(t), \ \ t>0, \\
  &\psi(0)=\psi_0\in \mathbb{H},\\
  \end{aligned}
\right.
\end{equation}
where $\mathcal{W}(t)=(\mathcal{W}_{k}(t))_{k\in\mathbb{N}}$ is a cylindrical Wiener process
in $\ell^2$ defined on
$(\Omega,\mathfrak{F},
 \{\mathfrak{F}_t\}_{t\geq0},\textbf{P})$.

In this subsection, we present the main results on the global-in-time well-posedness for problem \eqref{Int1--}. In the sequel, we assume that $\epsilon =1$ for   simplicity.

\begin{definition}\label{Defsolutionsolution}
(\texttt{Definition of solutions}) Given
$\psi_0\in L^2(\Omega,\mathfrak{F}_0, \mathbb{H})$, a continuous $\mathbb{H}$-valued
$\mathfrak{F}_t$-adapted stochastic process $\psi$
is called a solution to \eqref{Int1--}
if
\begin{align}&\label{Defsolution0}\psi
 \in L^2(\Omega, C([0,T],\mathbb{H} ))\bigcap L^p([0,T]\times\Omega, \mathbb{V}_1)\bigcap L^q([0,T]\times\Omega, \mathbb{V}_2  ),\ \ \forall\ T>0,
\end{align}
and for all $t\geq0$ and $\xi\in\mathbb{H}\cap\mathbb{V}_1
\cap\mathbb{V}_2$, $\mathbf{P}$-a.s.,
\begin{align}\label{Defsolution1}
&\langle \psi(t) ,\xi\rangle=\langle \phi_0 ,\xi\rangle+\sum_{i=1}^2
 \int_{0}^t\ _{\mathbb{V}_i^*}\big\langle \mathbf{A}_i(r,\psi(r),\mathcal{L}_{
\psi(r)}),
\xi\big  \rangle_{\mathbb{V}_i}        dr
 +\int_{0}^t
 \big\langle \xi, \mathfrak{L}_{H,S}(r,\psi(r),\mathcal{L}_{
\psi(r)})d\mathcal{W}(r) \big\rangle.
\end{align}
\end{definition}

\subsection{Existence of solutions of \eqref{Int1--} on bounded domains}
To prove Theorem \ref{Theorem4}, we first prove the existence of global solutions of the corresponding stochastic equations defined on bounded domains.
Given $k\in\mathbb{N}$,
we consider the Mckean-Vlasov
stochastic non-local \emph{fractional} $(\alpha,p)$-Laplacian equation
 defined on a bounded ball   $\mathcal{Q}_{k}:=\{x\in\mathbb{R}^d: |x|< k\}$:
\begin{equation}\label{Int1aae}
\left\{
  \begin{aligned}
  &d \psi_k(t) + \mathfrak{L}_{\mathcal{K}_{p}^\alpha} \psi_k(t) dt
=\sum_{i=1}^2H_i(t,x,\psi_k(t),\mathcal{L}_{
\psi_k(t)})dt
+ \mathfrak{L}_{H,S}(t,\psi_k(t),
\mathcal{L}_{
\psi_k(t)})d\mathcal{W}(t), \ \ \ t>0, \\
  &\psi_k(0,x)=\varphi(x/k)\psi_0(x),\ \ x\in\mathbb{R}^d,
  \\
  &\psi_k(t,x)=0,\ x\in\mathbb{R}^d\setminus\mathcal{Q}_{k},\ \ t>0,
  \\
  \end{aligned}
\right.
\end{equation}
where
$\varphi:\mathbb{R}^d\mapsto[0,1]$ is a smooth function such that
$\varphi(x)=1$ for $|x|\leq1/2$; and
$\varphi(x)=0$ for $|x|\geq1$.

Now, we consider three spaces:
$
\mathbb{H}_k:=\big\{\psi\in \mathbb{H}: \psi=0 \ \mbox{a.e. on}\ \ \mathbb{R}^d\backslash\mathcal{Q}_k\big\}$,
$\mathbb{V}_{1,k}:=\big\{\psi\in\mathbb{V}_{1}: \psi=0 \ \mbox{a.e. on}\ \ \mathbb{R}^d\backslash\mathcal{Q}_k\big\}$ and
$\mathbb{V}_{2,k}:=\big\{\psi\in \mathbb{V}_{2}: \psi=0 \ \mbox{a.e. on}\ \ \mathbb{R}^d\backslash\mathcal{Q}_k\big\}$.
Then we get the variational triples:
$\mathbb{V}_{1,k}\subseteq \mathbb{H}_k\equiv \mathbb{H}_k^*\subseteq \mathbb{V}_{1,k}^*$ and $ \mathbb{V}_{2,k}\subseteq \mathbb{H}_k\equiv \mathbb{H}_k^*\subseteq \mathbb{V}^*_{2,k}$.
On the other hand, since $\mathbb{H}_k$, $\mathbb{V}_{1,k}$ and $\mathbb{V}_{2,k}$ are subspaces of $\mathbb{H}$,
$\mathbb{V}_{1}$ and $\mathbb{V}_{2}$, respectively, we find that, for every $t\in\mathbb{R}^+$ and $\mu\in\mathcal{P}_2
(\mathbb{H})$, the operators $\mathbf{A}_1(t,\cdot,\mu):
\mathbb{V}_{1,k}\rightarrow \mathbb{V}^*_{1,k}$, $\mathbf{A}_2(t,\cdot,\mu):
\mathbb{V}_{2,k}\rightarrow \mathbb{V}^*_{2,k}$
and $\mathfrak{L}_{H,S}(t,\psi,\mu)
:\ell^2\rightarrow \mathbb{H}_k$ are well-defined. In addition, we have the relations:
\begin{subequations}\begin{align}\label{Re1}
 &\langle \psi_1 ,\psi_2\rangle_{\mathbb{H}_{k}}:=
 \int_{\mathcal{Q}_k}\psi_1(x)\psi_2(x)dx=
 \langle \psi_1 ,\psi_2\rangle_{\mathbb{H}},\ \ \forall\ \psi_1,\psi_2\in\mathbb{H}_{k},
 \\&\label{Re2}
 _{\mathbb{V}_{1,k}^*}\big\langle \mathbf{A}_1(t,\psi_1,\mu),
\psi_2\big  \rangle_{\mathbb{V}_{1,k}}=_{\mathbb{V}_1^*}
\big\langle \mathbf{A}_1(t,\psi_1,\mu),
\psi_2\big  \rangle_{\mathbb{V}_1},\ \ \forall\ \psi_1,\psi_2\in \mathbb{V}_{1,k},
\\&\label{Re3}
 _{\mathbb{V}_{2,k}^*}\big\langle \mathbf{A}_2(t,\psi_1,\mu),
\psi_2\big  \rangle_{\mathbb{V}_{2,k}}=_{\mathbb{V}_2^*}
\big\langle \mathbf{A}_2(t,\psi_1,\mu),
\psi_2\big  \rangle_{\mathbb{V}_2},\ \ \forall\ \psi_1,\psi_2\in \mathbb{V}_{2,k},
\\&\label{Re4}\|\mathfrak{L}_{H,S}(t,\psi,\mu)
\|_{\mathcal{L}_2(\ell^2,
\mathbb{H}_{k})}
\leq\|\mathfrak{L}_{H,S}(t,\psi,\mu)
\|_{\mathcal{L}_2(\ell^2,
\mathbb{H})},\ \ \forall\ \psi\in\mathbb{H}_{k},
 \end{align}
and
 \begin{align}\label{Re5}
 \big\langle \psi_1, \mathfrak{L}_{H,S}(t,\psi_2,\mu)v
\big\rangle_{\mathbb{H}_{k}}=\big\langle \psi_1, \mathfrak{L}_{H,S}(t,\psi_2,\mu)v
\big\rangle, \ \ \forall\ \psi_1,\psi_2\in\mathbb{H}_{k},\ v\in \ell^2.
 \end{align}
\end{subequations}

Since $\mathbb{V}_{1,k}\cap \mathbb{V}_{2,k}$ is
separable and dense
in $\mathbb{H}_k$, for every $\psi_0\in L^2(\Omega,\mathfrak{F}_0,\mathbb{H})$ and $k\in\mathbb{N}$,
by applying the fixed point theorem
 as in \cite{Chenzhang4b}
 and  the arguments of
  \cite{BWangsubmitted,Wangjde},
   one can prove that  problem \eqref{Int1aae} has a unique solution $\psi_k$ which is a continuous $\mathbb{H}_k$-valued
$\mathfrak{F}_t$-adapted stochastic process such that
\begin{align}\label{eihkhnergy5new}
&\psi_k
 \in L^2(\Omega, C([0,T], \mathbb{H}_k ))\bigcap L^p([0,T]\times\Omega, \mathbb{V}_{1,k}  )\bigcap L^q([0,T]\times\Omega, \mathbb{V}_{2,k}),
\end{align}
and for all $t\geq0$ and $\xi\in \mathbb{V}_{1,k}\cap \mathbb{V}_{2,k}$, $\mathbf{P}$-a.s.,
\begin{align}
\label{Defsolution1jff}
\langle \psi_k(t) ,\xi\rangle_{\mathbb{H}_{k}}-\langle
  \varphi(\cdot/k)\phi_0 ,\xi\rangle_{\mathbb{H}_{k}}
&=
\sum_{i=1}^2\int_{0}^t \ _{\mathbb{V}_{i,k}^*}\big\langle \mathbf{A}_i(r,\psi_k(r),\mathcal{L}_{
\psi_k(r)}),
\xi\big  \rangle_{\mathbb{V}_{i,k}}
 dr
\nonumber\\&\ +\int_{0}^t
 \big\langle \xi, \mathfrak{L}_{H,S}(r,\psi_k(r),\mathcal{L}_{
\psi_k(r)})d\mathcal{W}(r) \big\rangle_{\mathbb{H}_{k}},
\end{align}
and $\mathbf{P}$-a.s.,
\begin{align}\label{energy5new}
& \|\psi_k(t)\|_{\mathbb{H}_{k}}^2-\|
  \varphi(\cdot/k)\psi_0\|_{\mathbb{H}_{k}}^2
  =2\sum_{i=1}^2\int_{0}^t \ _{\mathbb{V}_{i,k}^*}\big\langle \mathbf{A}_i(r,\psi_k(r),\mathcal{L}_{
\psi_k(r)}),
\psi_k(r)\big  \rangle_{\mathbb{V}_{i,k}}
 dr
\nonumber\\&  \  +\int_{0}^t\|\mathfrak{L}_{H,S}(r,
\psi_k(r),\mathcal{L}_{
\psi_k(r)})\|^2_{\mathcal{L}_2(\ell^2,
\mathbb{H}_{k})}dr
  +2\int_{0}^t
 \big\langle \psi_k(r), \mathfrak{L}_{H,S}(r,\psi_k(r),\mathcal{L}_{
\psi_k(r)})d\mathcal{W}(r) \big\rangle_{\mathbb{H}_{k}}.
\end{align}

\subsection{Uniform estimates of $\{\psi_k\}_{k\in\mathbb{N}}$}
The next lemma is concerned with the uniform estimates  of solutions $\{\psi_k\}_{k\in\mathbb{N}}$ to \eqref{Int1aae}.

\begin{lemma}\label{ffTheorem4kk}
Let conditions {\bf A}
and {\bf B} be satisfied.
If $\psi_0\in L^{2}(\Omega,\mathfrak{F}_0,\mathbb{H})$,
then for all $k\in\mathbb{N}$,
\begin{align*}
&\mathbf{E}\bigg[\sup_{r\in[0,T]}\|\psi_k(r)
\|^{2}\bigg]
+
\mathbf{E}\bigg[\int_{0}^{T}\Big(\|\psi_k(r)
\|_{\mathbb{V}_1}^{p} +\|\psi_k(r)\|^q_{\mathbb{V}_2}\Big)
 dr\bigg]\leq  C(T) (1+\mathbf{E}[\|\psi_0\|^2]),
\end{align*}
and
\begin{align*}
&
\mathbf{E}\bigg[\int_0^{T}
\|\mathbf{A}_1(r,\psi_k(r),\mathcal{L}_{
\psi_k(r)}) \|^{\widehat{p}}_{\mathbb{V}^*_1}dr\bigg]
+
\mathbf{E}\bigg[\int_0^{T}
\|\mathbf{A}_2(r,\psi_k(r),\mathcal{L}_{
\psi_k(r)}) \|^{\widehat{q}}_{\mathbb{V}^*_2}dr\bigg]
\nonumber\\&+\mathbf{E}\left[
\int_0^{T}
\|\mathfrak{L}_{H,S}(r,\psi_k(r),\mathcal{L}_{
\psi_k(r)})
\|^2_{\mathcal{L}_2(\ell^2,
\mathbb{H})}dr
\right]
\le  C(T) (1+\mathbf{E}[\|\psi_0\|^2]),
\end{align*}
where $C(T)>0$ is a constant independent of $k$.

\end{lemma}
\begin{proof}
By the energy equation \eqref{energy5new} and the relations \eqref{Re1}-\eqref{Re5}, we can use \eqref{h2-}, \eqref{ffnewA2a0} and Lemma
\ref{Dissipativeness} to derive the uniform
estimates.
The details of the proof  are  omitted here.
\end{proof}

\subsection{Proof of Theorem \ref{Theorem4}}
With the results in
the previous subsections, we are in a
 position to establish the global-in-time existence of solutions of \eqref{Int1--}  on $\mathbb{R}^d$ by examining the
limiting behavior of solutions
of \eqref{Int1aae} on $\mathcal{Q}_{k}$ as $k\rightarrow\infty$. The pathwise uniqueness of global-in-time solutions of \eqref{Int1--} is also proved.

\texttt{Proof of Theorem \ref{Theorem4}}.
\emph{Existence of global solutions}:
 Given $\xi\in \mathbb{H}\bigcap\mathbb{V}_1
\bigcap\mathbb{V}_2$,
we define $\xi_{m}:=\varphi(\cdot/m)\xi\in \mathbb{H}\bigcap\mathbb{V}_1
\bigcap\mathbb{V}_2$,
where $\varphi$ is the cut-off function as in \eqref{Int1aae}. Then, by \cite{Wang-wangb,RWangsubmitted}, we find that $\xi_m\rightarrow \xi \ \mbox{in}\ \mathbb{H}\bigcap\mathbb{V}_1
\bigcap\mathbb{V}_2$.
In addition, we have $\xi_m\in\mathbb{H}_k\bigcap\mathbb{V}_{1,k}
\bigcap\mathbb{V}_{2,k}$ for $m\leq k$.
Thus, by integration by parts and
the	Fubini theorem, we infer from \eqref{Defsolution1jff} and \eqref{Re1}-\eqref{Re5} that for any $\zeta\in L^\infty([0,T]
	\times\Omega,\mathbb{R})$,
	\begin{align}\label{Desyfsolutjk;jion25djdjd}
		& \mathbf{E}
		\bigg[\int_{0}^{T}\zeta(t)\langle \psi_k(t) ,\xi_m\rangle dt\bigg]-
		\mathbf{E}
		\bigg[\int_{0}^{T}\zeta(t) \langle
		\varphi(|\cdot|/k)\psi_0 ,\xi_m\rangle dt\bigg]
		\nonumber\\&=
		\sum_{i=1}^2\mathbf{E}
		\bigg[\int_{0}^{T}\zeta(t) \int_{0}^t\
		_{\mathbb{V}_i^*}\langle \mathbf{A}_i(r,\psi_k(r),\mathcal{L}_{\psi_k(r)}) ,\xi_m\rangle_{\mathbb{V}_i}
	dr dt\bigg]
\nonumber\\&\   +\mathbf{E}
		\bigg[\int_{0}^{T}\zeta(t)
\int_{0}^t
		\big\langle \xi_m, \mathfrak{L}_{H,S}(r,\psi_k(r),\mathcal{L}_{\psi_k(r)})d\mathcal{W}(r) \big\rangle dt\bigg]
		\nonumber\\&=\sum_{i=1}^2
		\mathbf{E}
		\bigg[\int_{0}^{T}
		\ _{\mathbb{V}_i^*}\Big\langle \mathbf{A}_i(t,\psi_k(t),\mathcal{L}_{\psi_k(r)}) ,\xi_m\int_{t}^{T}
		\zeta(r)dr\Big\rangle_{\mathbb{V}_i}dt\bigg]
		\nonumber\\&\   +\int_{0}^{T}
		\mathbf{E}
		\bigg[
		\Big\langle \zeta(t)\xi_m, \int_{0}^t\mathfrak{L}_{H,S}(r,\psi_k(r),
\mathcal{L}_{\psi_k(r)})d\mathcal{W}(r) \Big\rangle\bigg]
		dt.
	\end{align}

We want to take the limit in \eqref{Desyfsolutjk;jion25djdjd} as $k\rightarrow\infty$. To this end, we shall establish several weak convergence results for the approximate
 solutions $\{\psi_k\}_{k=1}^\infty$ of \eqref{Int1aae}.
By the uniform estimates
in  Lemma \ref{ffTheorem4kk} and the arguments of
 \cite{Wangjde,RWangsubmitted},
we find that there exist
	$\psi^*
	\in
	\big(L^ {2}(\Omega,
	L^1([0,T],\mathbb{H} ))\big)^*$,
	$\psi\in L^{2}(  [0,T]\times\Omega, \mathbb{H})
	\bigcap L^p([0,T]\times\Omega, \mathbb{V}_1)\bigcap L^q([0,T]\times\Omega, \mathbb{V}_2)$,
$\mathfrak{A}_1\in L^{\widehat{p}}([0,T]\times\Omega, \mathbb{V}^*_1)$,
$\mathfrak{A}_2\in L^{\widehat{q}}([0,T]\times\Omega, \mathbb{V}^*_2)$, $\widehat{\mathfrak{L}}_{H,S}\in
	L^2([0,T] \times\Omega,\mathcal{L}_2(
\ell^2, \mathbb{H}))$, and a subsequence which is not relabeled, such that
\begin{subequations}\begin{align}\label{Defia}
			&\psi_{k}\rightarrow \psi^*\ \
			\mbox{weak-star in $\big(L^ 2(\Omega,
				L^1([0,T],\mathbb{H} ))
				\big)
				^*$},
			\\&\label{Defibyy}
			\psi_{k}\rightarrow \psi\ \
			\mbox{weakly in } \
			L^{2}(  [0,T]\times\Omega, \mathbb{H}
			),
			\\&\label{Defsolution25b}
			\psi_{k}\rightarrow \psi\ \  \mbox{weakly in $L^p([0,T]\times\Omega, \mathbb{V}_1)$},		\\&\label{Defsolution25b1}\psi_{k}\rightarrow \psi\ \  \mbox{weakly in $L^q([0,T]\times\Omega, \mathbb{V}_2  )$},
			\\& \label{Defsolution25c}
			\mathbf{A}_1(\cdot,\psi_k,\mathcal{L}_{\psi_k}) \rightarrow \mathfrak{A}_1\ \  \mbox{weakly in $ L^{\widehat{p}}([0,T]\times\Omega, \mathbb{V}^*_1)$},
			\\& \label{Defsolution25c2}
			\mathbf{A}_2(\cdot,\psi_k,\mathcal{L}_{\psi_k}) \rightarrow \mathfrak{A}_2\ \  \mbox{weakly in $ L^{\widehat{q}}([0,T]\times\Omega, \mathbb{V}^*_2)$},
			\\&\label{inequalities1}
			\mathfrak{L}_{H,S}(\cdot,\psi_k,
\mathcal{L}_{\psi_k})\rightarrow \widehat{\mathfrak{L}}_{H,S} \ \ \mbox{weakly in $L^2([0,T]			\times\Omega,\mathcal{L}_2(\ell^2,
\mathbb{H}))$},
			\\\label{improved}
			&\mathfrak{L}_{H,S}(\cdot,\psi_k,\mathcal{L}_{\psi_k}
)v\rightarrow \widehat{\mathfrak{L}}_{H,S}v \ \ \mbox{weakly in $L^2([0,T]
				\times\Omega,\mathbb{H})$}, \ \ \forall\ v\in \ell^2,
			\\\label{improvedb}
			&\int_0^\cdot
			\mathfrak{L}_{H,S}(\cdot,\psi_k(r),
\mathcal{L}_{\psi_k(r)})
			d\mathcal{W}(r)
			\rightarrow\int_0^\cdot \widehat{\mathfrak{L}}_{H,S}(r)d\mathcal{W}(r)\ \mbox{weakly in $L^2([0,T], L^2(\Omega,\mathbb{H}))$},
		\end{align}
		and
		\begin{align}
			\label{improvedbcc}
			&\int_0^\cdot
			\mathfrak{L}_{H,S}(\cdot,\psi_k(r),
\mathcal{L}_{\psi_k(r)})
			d\mathcal{W}(r)
			\rightarrow\int_0^\cdot \widehat{\mathfrak{L}}_{H,S}(r)d\mathcal{W}(r)\ \mbox{weak star  in $L^\infty([0,T], L^2(\Omega,\mathbb{H}))$}.
		\end{align}
	\end{subequations}
Passing
to  the limit in \eqref{Desyfsolutjk;jion25djdjd} as $k\rightarrow\infty$, by above
convergence results, it follows that
\begin{align}\label{Desyjd-}
		& \mathbf{E}
		\bigg[\int_{0}^{T}\zeta(t)\langle \psi(t) ,\xi_m\rangle dt\bigg]-
		\mathbf{E}
		\bigg[\int_{0}^{T}\zeta(t) \langle
		\psi_0 ,\xi_m\rangle dt\bigg]
		\nonumber\\&=
		\sum_{i=1}^2\mathbf{E}
		\bigg[\int_{0}^{T}
		\ _{\mathbb{V}^*_i}\Big\langle \mathfrak{A}_i(t) ,\xi_m\int_{t}^{T}
		\zeta(r)dr\Big\rangle_{\mathbb{V}_i}dt\bigg]
		  +\int_{0}^{T}
		\mathbf{E}
		\bigg[
		\Big\langle \zeta(t)\xi_m, \int_{0}^t\widehat{\mathfrak{L}}_{H,S}(r)
d\mathcal{W}(r) \Big\rangle\bigg]
		dt.
	\end{align}
Passing
to the limit in \eqref{Desyjd-} as
$m\rightarrow\infty$, by $\xi_m\rightarrow \xi$  in $\mathbb{H}\bigcap\mathbb{V}_1
\bigcap\mathbb{V}_2$,
we arive at
\begin{align}\label{Desyjd--}
		& \mathbf{E}
		\bigg[\int_{0}^{T}\zeta(t)\langle \psi(t) ,\xi\rangle dt\bigg]-
		\mathbf{E}
		\bigg[\int_{0}^{T}\zeta(t) \langle
		\psi_0 ,\xi\rangle dt\bigg]
		\nonumber\\&=\sum_{i=1}^2
		\mathbf{E}
		\bigg[\int_{0}^{T}
		\ _{\mathbb{V}_i^*}\Big\langle \mathfrak{A}_i(t) ,\xi\int_{t}^{T}
		\zeta(r)dr\Big\rangle_{\mathbb{V}_i}dt\bigg]
		   +\int_{0}^{T}
		\mathbf{E}
		\bigg[
		\Big\langle \zeta(t)\xi, \int_{0}^t\widehat{\mathfrak{L}}_{H,S}(r)
d\mathcal{W}(r) \Big\rangle\bigg]
		dt
\nonumber\\&=\sum_{i=1}^2
		\mathbf{E}
		\bigg[\int_{0}^{T}\zeta(t)
		\ _{\mathbb{V}_i^*}
\Big\langle \int_{0}^{t}\mathfrak{A}_i(r) dr ,\xi
		\Big\rangle_{\mathbb{V}_i}dt\bigg]
		  +\mathbf{E}
		\bigg[\int_{0}^{T} \zeta(t)
		\Big\langle \int_{0}^t\widehat{\mathfrak{L}}_{H,S}(r)
d\mathcal{W}(r),\xi \Big\rangle
		dt\bigg].
	\end{align}
This, in fact, implies that
\begin{align}\label{;jionjd}
		& \psi(\cdot)
		=
		\psi_0  +\int_{0}^{\cdot}
\mathfrak{A}_1(r)dr+\int_{0}^{\cdot}
\mathfrak{A}_2(r)dr
		+\int_{0}^\cdot\widehat{\mathfrak{L}}_{H,S}(r)
d\mathcal{W}(r)\  \mbox{in}\ \big(\mathbb{H}\bigcap\mathbb{V}_1
\bigcap\mathbb{V}_2\big)^*,
	\end{align}
$dt\times\mathbb{P}$-\mbox{a.e.} on $[0,T]\times\Omega$.
Based on \eqref{;jionjd},
we infer that
$ \psi$ has a version (still denoted by
$\psi$) which is  a
continuous   $\mathbb{H}$-valued
	 $\mathfrak{F}_t$-adapted stochastic processes
	such that
	$\psi\in L^2(\Omega, L^\infty([0,T],\mathbb{H}) )$	
and $\mathbf{P}$-a.s.,
	\begin{align}\label{enHRTrFYDDRergy5}	\|\psi(t)\|^2
		&=\|\psi_0\|^2+
		\int_{0}^t\bigg( 2\sum_{i=1}^2\ _{\mathbb{V}^*_i}\langle \mathfrak{A}_i(r) ,\psi(r)\rangle_{\mathbb{V}_i}
		+\|\widehat{\mathfrak{L}}_{H,S}(r)
\|^2_{\mathcal{L}_2(\ell^2,
			\mathbb{H})}\bigg)dr
+2\int_{\tau}^t
		\big\langle (\widehat{\mathfrak{L}}_{H,S}(r))^*\psi(r), d\mathcal{W}(r) \big\rangle_{\ell^2}.
	\end{align}
Since $\psi\in L^2(\Omega, L^\infty([0,T],\mathbb{H}) )$,
from  \eqref{Defia} and \eqref{Defibyy}, it follows that $\psi=\psi^*$.
	
Moreover, from
  \eqref{enHRTrFYDDRergy5} we have
\begin{align}\label{enerudufoogy}
\mathbf{E}\big[\|\psi(t)\|^{2}\big]
		& =\mathbf{E}\big[
\|\psi_0\|^{2}\big]	
+ \mathbf{E}\bigg[\int_{0}^t\bigg( 2\sum_{i=1}^2\ _{\mathbb{V}^*_i}\langle \mathfrak{A}_i(r) ,\psi(r)\rangle_{\mathbb{V}_i}
		+\|\widehat{\mathfrak{L}}_{H,S}(r)
\|^2_{\mathcal{L}_2(\ell^2,
			\mathbb{H})}\bigg)dr\bigg].
	\end{align}

Let $\psi_\varepsilon:=\psi-\varepsilon
	\widehat{\zeta} u$
	with
	$\varepsilon \in [0,1]$,
	$\widehat{\zeta}\in L^\infty([0,T]
	\times\Omega,\mathbb{R})$
	and $u\in\mathbb{H}\bigcap\mathbb{V}_1
\bigcap\mathbb{V}_2$. Let
$\mathbf{G}(t):=
C_4(1,\beta/2,\widehat{\beta}/2)+C_5(1)
+2\big(
\|\phi_6(t)
\|_{\infty}
+\|\widehat{\phi}_6(t)
\|_{\infty}+\|\phi_7(t)
\|_{1}+\|\widehat{\phi}_7(t)
\|_{1}\big) $ for
$t\in \mathbb{R}^+$, where
  $C_4(1,\beta/2,\widehat{\beta}/2)$ and $C_5(1)$
are the
constants
in   \eqref{h2--} for $M=1$, $\varepsilon_1=\beta/2$ and
$\varepsilon_2=\widehat{\beta}/2$.
Then we find that $\psi_\epsilon\in L^2(\Omega, L^\infty([0,T],\mathbb{H}) )$ and
$\mathbf{G}\in
L_{loc}^{1}(\mathbb{R}^+,\mathbb{R}^+)$.
Thus, it follows from \eqref{energy5new} and \eqref{Re1}-\eqref{Re5}
that
\begin{align}&\label{ffzxa2}
		\mathbf{E} \left[e^{-\int_0^{t}
\mathbf{G}(s)ds}
		\|\psi_{k}(t)
\|^2-
		\|\varphi(\cdot/k)\psi_0\|^2\right]
		\nonumber\\
		& =
		\mathbf{E}
		\Bigg[
		\int_0^{t}
		e^{-\int_0^s\mathbf{G}(r)dr}
		\Bigg(2 \sum_{i=1}^2\ _{\mathbb{V}^*_i}\langle \mathbf{A}_i(s,\psi_k(s),\mathcal{L}_{\psi_k(s)}) ,\psi_k(s)\rangle_{\mathbb{V}_i}
 \nonumber\\&\   \ +\|\mathfrak{L}_{H,S}(s,\psi_k(s),
\mathcal{L}_{\psi_k(s)})\|^2_{\mathcal{L
}_2(\ell^2,
			\mathbb{H})}  -
	\mathbf{G}(s)
		\|\psi_{k}(s)\|^2
		\Bigg)ds\Bigg]
		\nonumber\\&=
		\mathbf{E}
		\Bigg[
		\int_0^{t}
		e^{-\int_0^s\mathbf{G}(r)dr} \Bigg(    2 \sum_{i=1}^2\Big(
_{\mathbb{V}_i^*}\langle \mathbf{A}_i(s,
\psi_\varepsilon(s),\mathcal{L}_{\psi_\varepsilon(s)}) ,\psi_k(s)\rangle_{\mathbb{V}_i}
	\nonumber\\&\ \ 	+\ _{\mathbb{V}^*_i}
		\langle \mathbf{A}_i(s,\psi_k(s),\mathcal{L}_{\psi_k(s)})
-\mathbf{A}_i(s,\psi_\varepsilon(s),
\mathcal{L}_{\psi_\varepsilon(s)}) ,\psi_\varepsilon(s)\rangle_{\mathbb{V}_i}\Big)
		\nonumber\\&\ \    -\|\mathfrak{L}_{H,S}(s,\psi_\varepsilon(s),
\mathcal{L}_{\psi_\varepsilon(s)})
		\|^2_{\mathcal{L}_2(\ell^2,
			\mathbb{H})}
		+2 \langle  \mathfrak{L}_{H,S}(s,\psi_k(s),
\mathcal{L}_{\psi_k(s)})
		, \mathfrak{L}_{H,S}(s,\psi_\varepsilon(s),
\mathcal{L}_{\psi_\varepsilon(s)})
		\rangle_{\mathcal{L}_2(\ell^2,
			\mathbb{H})}
		\nonumber\\&\ \   -
		\mathbf{G}(s) \Big(2\langle \psi_k(s)
		, \psi_\varepsilon(s)
		\rangle-
		\|\psi_\varepsilon(s)\|^2
		\Big)
		\Bigg)ds\Bigg]
		\nonumber\\&\ \ +
		\mathbf{E}
		\Bigg[
		\int_0^{t}
e^{-\int_0^s\mathbf{G}(r)dr}
		\Bigg(
		2 \sum_{i=1}^2\ _{\mathbb{V}_i^*}
		\langle \mathbf{A}_i(s,\psi_k(s),\mathcal{L}_{\psi_k(s)}
)-\mathbf{A}_i(s,\psi_\varepsilon(s),
\mathcal{L}_{\psi_\varepsilon(s)}) ,\psi_k(s)-\psi_\varepsilon(s)\rangle_{\mathbb{V}_i}
\nonumber\\&  \  \ +\|\mathfrak{L}_{H,S}(s,\psi_k(s),
\mathcal{L}_{\psi_k(s)})
		-\mathfrak{L}_{H,S}(s,\psi_\varepsilon(s),
\mathcal{L}_{\psi_\varepsilon(s)})
		\|^2_{\mathcal{L}_2(\ell^2,
			\mathbb{H})}   -
		\mathbf{G}(s)
		\|\psi_{k}(s)-\psi_\varepsilon(s)
\|^2
		\Bigg)ds\Bigg].
	\end{align}
Note that the last term in \eqref{ffzxa2}
	is less or equal to zero due to
\eqref{h2--} (for $M=1$, $\varepsilon_1=\beta/2$ and
$\varepsilon_2=\widehat{\beta}/2$), \eqref{h2----0}, Lemma \ref{Monotonicity} and the  property of the
$L^2$-Wasserstein distance \eqref{Wassersteinaa}.
According to the Fubini's theorem and
integration by parts, by \eqref{ffzxa2}, we derive that for all $\zeta\in L^\infty([0,T],\mathbb{R}^+)$,
	\begin{align}&\label{ffzrrrgggxa2}		\mathbf{E}\Bigg[\int_0^{T}
\zeta(t)\Bigg(
		e^{-\int_0^{t}
\mathbf{G}(s)ds}
		\|\psi_{k}(t)
\|^2-
		\|\varphi(\cdot/k)\psi_0\|^2\Bigg)dt\Bigg]
		\nonumber\\&\  \leq
		\mathbf{E}\Bigg[\int_0^{T} \zeta(t)
\Bigg(
		\int_0^{t}
e^{-\int_0^s\mathbf{G}(r)dr} \Bigg( 2\sum_{i=1}^2\Big( _{\mathbb{V}_i^*}
		\langle \mathbf{A}_i(s,
\psi_\varepsilon(s),
\mathcal{L}_{\psi_\varepsilon(s)}) ,\psi_k(s)\rangle_{\mathbb{V}_i}
\nonumber\\&\  	+\ _{\mathbb{V}_i^*}
		\langle \mathbf{A}_i(s,\psi_k(s),\mathcal{L}_{\psi_k(s)})
-\mathbf{A}_i(s,\psi_\varepsilon(s),
\mathcal{L}_{\psi_\varepsilon(s)}) ,\psi_\varepsilon(s)\rangle_{\mathbb{V}_i}\Big)
		\nonumber\\&\    -\|\mathfrak{L}_{H,S}(s,\psi_\varepsilon(s),
\mathcal{L}_{\psi_\varepsilon(s)})
		\|^2_{\mathcal{L}_2(\ell^2,
			\mathbb{H})}
		+2 \langle  \mathfrak{L}_{H,S}(s,\psi_k(s),
\mathcal{L}_{\psi_k(s)})
		, \mathfrak{L}_{H,S}(s,\psi_\varepsilon(s),
\mathcal{L}_{\psi_\varepsilon(s)})
		\rangle_{\mathcal{L}_2(\ell^2,
			\mathbb{H})}
		\nonumber\\&\    -
		\mathbf{G}(s) \Big(2\langle \psi_k(s)
		, \psi_\varepsilon(s)
		\rangle-
		\|\psi_\varepsilon(s)\|^2
		\Big)
		\Bigg)ds\Bigg)dt\Bigg]
\nonumber\\&=
		\mathbf{E}\Bigg[\int_0^{T} \Bigg(\int_{t}^{T}\zeta(r)dr\Bigg)
\Bigg(e^{-\int_0^{t}\mathbf{G}(r)dr} \Bigg(   2\sum_{i=1}^2\Big( _{\mathbb{V}_i^*}
		\langle \mathbf{A}_i(t,
\psi_\varepsilon(t),\mathcal{L}_{
\psi_\varepsilon(t)}) ,\psi_k(t)\rangle_{\mathbb{V}_i}
\nonumber\\&\  +\ _{\mathbb{V}^*_i}
		\langle \mathbf{A}_i(t,\psi_k(t),\mathcal{L}_{\psi_k(t)})
-\mathbf{A}_i(t,\psi_\varepsilon(t),
\mathcal{L}_{\psi_\varepsilon(t)}) ,\psi_\varepsilon(t)\rangle_{\mathbb{V}_i}\Big)
		\nonumber\\&\   -\|\mathfrak{L}_{H,S}(t,\psi_\varepsilon(t),
\mathcal{L}_{\psi_\varepsilon(t)})
		\|^2_{\mathcal{L}_2(\ell^2,
			\mathbb{H})}
		+2 \langle  \mathfrak{L}_{H,S}(t,\psi_k(t),
\mathcal{L}_{\psi_k(t)})
		, \mathfrak{L}_{H,S}(t,\psi_\varepsilon(t),
\mathcal{L}_{\psi_\varepsilon(t)})
		\rangle_{\mathcal{L}_2(\ell^2,
			\mathbb{H})}
		\nonumber\\&\    -
		\mathbf{G}(t)\Big(2\langle \psi_k(t)
		, \psi_\varepsilon(t)
		\rangle-
		\|\psi_\varepsilon(t)\|^2
		\Big)
		\Bigg)\Bigg)dt\Bigg].
	\end{align}

We now want to pass
to  the limit in \eqref{ffzrrrgggxa2} as $k\rightarrow\infty$.
Since
	$ \psi_k
	\to \psi$
	weakly in
	$L^2(\Omega, L^2([0, T],
	\mathbb{H})),
	$
we see  that $
\zeta^{1/2} e^{-\frac{1}{2}\int_0^{\cdot}
\mathbf{G}(s)ds} \psi_k \rightarrow \zeta^{1/2}  e^{-\frac{1}{2}\int_0^{\cdot}
\mathbf{G}(s)ds}\psi$
weakly in
	$L^2(\Omega, L^2([0, T],
	\mathbb{H}))$, for all $\zeta\in L^\infty([0,T],\mathbb{R}^+)$. Based on this weak convergence and \eqref{Defia}-\eqref{improvedb},
	by passing
	to the limit in \eqref{ffzrrrgggxa2} as $k\rightarrow\infty$, we find that
	\begin{align}&\label{ffzrrrgxa2}
\mathbf{E}\Bigg[\int_0^{T}
\zeta(t)\Bigg(
		e^{-\int_0^{t}\mathbf{G}(s)ds}
		\|\psi(t)
\|^2-
		\|\psi_0\|^2\Bigg)dt\Bigg]
\nonumber\\&
\leq\liminf_{k\rightarrow\infty}
\mathbf{E}\Bigg[\int_0^{T}
\zeta(t)\Bigg(
		e^{-\int_0^{t}
\mathbf{G}(s)ds}
		\|\psi_{k}(t)
\|^2-
		\|\varphi(\cdot/k)\psi_0\|^2\Bigg)dt\Bigg]
	\nonumber\\&\leq
		\mathbf{E}\Bigg[\int_0^{T} \zeta(t)
\Bigg(
		\int_0^{t}
e^{-\int_0^s\mathbf{G}(r)dr}\Bigg(   2\sum_{i=1}^2 \Big(  _{\mathbb{V}_i^*}
		\langle \mathbf{A}_i(s,
\psi_\varepsilon(s),\mathcal{L}_{\psi_\varepsilon(s)}) ,\psi(s)\rangle_{\mathbb{V}_i}
		+ _{\mathbb{V}_i^*}
		\langle \mathfrak{A}_i
-\mathbf{A}_i(s,\psi_\varepsilon(s),
\mathcal{L}_{\psi_\varepsilon(s)}) ,\psi_\varepsilon(s)\rangle_{\mathbb{V}_i}\Big)
		\nonumber\\&\ \ \ \ \ \    -\|\mathfrak{L}_{H,S}(s,\psi_\varepsilon(s),
\mathcal{L}_{\psi_\varepsilon(s)})
		\|^2_{\mathcal{L}_2(\ell^2,
			\mathbb{H})}
		+2 \langle \widehat{\mathfrak{L}}_{H,S}(s)
		, \mathfrak{L}_{H,S}(s,\psi_\varepsilon(s),
\mathcal{L}_{\psi_\varepsilon(s)})
		\rangle_{\mathcal{L}_2(\ell^2,
			\mathbb{H})}
		\nonumber\\&\ \    -
		\mathbf{G}(s) \Big(2\langle \psi(s)
		, \psi_\varepsilon(s)
		\rangle-
		\|\psi_\varepsilon(s)\|^2
		\Big)
		\Bigg)ds\Bigg)dt\Bigg],		
	\end{align}
where we treat the convergence in the last line by using the weak star convergence \eqref{Defia}
and the conclusion:
$$\bigg(\int_{\cdot}^{T}\zeta(r)dr\bigg)
e^{-\int_0^{\cdot}\mathbf{G}(r)dr}
\mathbf{G}(\cdot)\psi_\varepsilon\in L^ 2(\Omega,
				L^1([0,T],\mathbb{H} )),$$
which is a consequence of $\psi_\epsilon\in L^2(\Omega, L^\infty([0,T],\mathbb{H}) )$ and
$\mathbf{G}\in
L_{loc}^{1}(\mathbb{R}^+,\mathbb{R}^+)$.

By \eqref{enerudufoogy}, we
infer that for all $\zeta\in L^\infty([0,T],\mathbb{R}^+)$,
	\begin{align*}&	\mathbf{E}\Bigg[\int_0^{T}
\zeta(t)\Bigg(
	e^{-\int_0^{t}
\mathbf{G}(s)ds} \|\psi(t)
\|^2-
\|\psi_0\|^2\Bigg)dt\Bigg]
		\nonumber\\&
		=\mathbf{E}\Bigg[\int_0^{T}\zeta(t)\Bigg(
		\int_0^{t
		}
		e^{-\int_0^s\mathbf{G}(r)
dr}
	\Bigg(2
\sum_{i=1}^2\   _{\mathbb{V}^*_i}
		\langle \mathfrak{A}_i(s) ,\psi(s)\rangle_{\mathbb{V}_i}
		+\|\widehat{\mathfrak{L}}_{H,S}(s)
		\|^2_{\mathcal{L}_2(\ell^2,
			\mathbb{H})}
	-
		\mathbf{G}(s)
		\|\psi(s)\|^2
		\Bigg)ds\Bigg)dt \Bigg],
	\end{align*}
which along with \eqref{ffzrrrgxa2}  implies that \begin{align}&\label{ffzrddroyofo}
		0\geq
		\mathbf{E}\Bigg[\int_0^{T}
\zeta(t)\Bigg(
		\int_0^{t}
e^{-\int_0^s\mathbf{G}(r)dr}
		\Bigg(2\sum_{i=1}^2\ _{\mathbb{V}_i^*}
		\langle \mathfrak{A}_i(s)-
\mathbf{A}_i(s,\psi_\varepsilon(s),
\mathcal{L}_{\psi_\varepsilon(s)}) ,\psi(s)-\psi_\varepsilon(s)
\rangle_{\mathbb{V}_i}
		\nonumber\\&  \ \ +\|\mathfrak{L}_{H,S}(s,\psi_\varepsilon(s),
\mathcal{L}_{\psi_\varepsilon(s)})-
		\widehat{\mathfrak{L}}_{H,S}(s)
		\|^2_{\mathcal{L}_2(\ell^2,
			\mathbb{H})} -
		\mathbf{G}(s)
		\|\psi(s)-\psi_\varepsilon(s)\|^2
		\Bigg)ds\Bigg)dt\Bigg].
	\end{align}

	Letting $\varepsilon =0$
	in   \eqref{ffzrddroyofo},
since $\psi_\varepsilon =\psi$
for $\varepsilon =0$, we find  from
 the arbitrariness of $\zeta\in L^\infty([0,T],\mathbb{R}^+)$ that $\widehat{\mathfrak{L}}_{H,S}=
\mathfrak{L}_{H,S}(\cdot,
\psi,\mathcal{L}_{\psi})$, $dt\times\mathbb{P}$-\mbox{a.e.} on $[0,T]\times\Omega$ as desired.

For $\varepsilon>0$, we first
 divide
	both sides of  \eqref{ffzrddroyofo}  by $\varepsilon$,
	and
	then taking   the limit as $\varepsilon\rightarrow0$,
	by
  the hemicontinuity of
$\textbf{A}_i$ ($i=1,2$) in Lemma \ref{Hemicontinuity}
and
the	Lebesgue dominated convergence theorem,
we obtain
\begin{align}&\label{ffzrdfo}
		0\geq
		\mathbf{E}\Bigg[\int_0^{T}
\zeta(t)\Bigg(
		\int_0^{t}
e^{-\int_0^s\mathbf{G}(r)dr}
		\sum_{i=1}^2\ _{\mathbb{V}_i^*}
		\langle \mathfrak{A}_i(s)-
\mathbf{A}_i(s,\psi(s),
\mathcal{L}_{\psi(s)}) ,u
\rangle_{\mathbb{V}_i}
\widehat{\zeta}(s)
		ds\Bigg)dt\Bigg].
	\end{align}
	Choosing  $\zeta \equiv 1$,
by  Fubini's theorem,   the arbitrariness of
$\widehat{\zeta}$ and $u\in\mathbb{H}\cap\mathbb{V}_1
\cap\mathbb{V}_2$, we  infer
 that $$\mathfrak{A}_1+\mathfrak{A}_2=
\mathbf{A}_1(\cdot,\psi,\mathcal{L}_{\psi}) +\mathbf{A}_2(\cdot,\psi,\mathcal{L}_{\psi})\ \ \mbox{in}\  \Big(\mathbb{V}_1
\bigcap\mathbb{V}_2\Big)^*, \ \ \mbox{ $dt\times\mathbb{P}$-\mbox{a.e.} on $[0,T]\times\Omega$}.
$$
This, along with \eqref{;jionjd} and $\widehat{\mathfrak{L}}_{H,S}=
\mathfrak{L}_{H,S}(\cdot,
\psi,\mathcal{L}_{\psi})$, $dt\times\mathbb{P}$-\mbox{a.e.} on $[0,T]\times\Omega$, shows that $\psi$ is a solution to \eqref{Int1--} in the sense of Definition \ref{Defsolutionsolution}.
By the arguments of Lemma \ref{ffTheorem4kk}, we can obtain  \eqref{G8}.

\emph{Pathwise uniqueness of global solutions}.
Let $\psi_1(t):=\psi_1(t,\psi_{0,1})$ and $\psi_2(t):=\psi_2(t,\psi_{0,2})$ for all $t\in [0, T]$ be two solutions to \eqref{Int1--} in the sense of
	Definition \ref{Defsolutionsolution}.
If $\psi_{0,1}=\psi_{0,2}$, $\textbf{P}$-a.s., then it yields from \eqref{h2--}, \eqref{h2----0} and Lemma \ref{Monotonicity}
	that
	\begin{align*}
	 e^{-\int_0 ^{t}
		\mathbf{G}(r)dr}	\|\psi_1(t)-
		\psi_2(t)\|^2 &\leq
		2
		\int_0^{t}
		e^{-\int^s_{0}
			\mathbf{G}(r)dr}
		\\&\ \ \  \times\big\langle \psi_1(s)-\psi_2(s),
		\big(\mathfrak{L}_{H,S}(s,\psi_1(s),
\mathcal{L}_{\psi_1(r)})
		-\mathfrak{L}_{H,S}(s,\psi_2(s),
\mathcal{L}_{\psi_2(r)})\big)
		d\mathcal{W}(s)\big
\rangle, \ \ \mbox{$\textbf{P}$-a.s.},
	\end{align*}
by  which , one can derive from
\eqref{h2-}, \eqref{ffnewA2a0} and \eqref{G8} that
	\begin{align}\label{wan10301}
	\textbf{E}\bigg [e^{-\int_0 ^{t}
		\mathbf{G}(r)dr}	\|\psi_1(t)-
		\psi_2(t)\|^2\bigg ]=0.
\end{align}	
For every
	$t\in [0, T]$, by \eqref{wan10301}, we find that $\|\psi_1(t)-
		\psi_2(t)\|=0$,
$\mathbb{P}$-a.s.. This, together with the pathwise continuity of
$\psi_1$ and $\psi_2$, implies the pathwise uniqueness of the solutions.

\section{Weak convergence theory for Freidlin-Wentzell and Dembo-Zeitouni uniform large deviations}
For the reader's convenience, we recall the weak convergence theory
for both Freidlin-Wentzell and Dembo-Zeitouni uniform large deviations, see \cite{Budhiraja,Dembo,Freidlin,Salins1,Salins2}.

Let $(\mathbb{X},\|\cdot\|_{\mathbb{X}})$ be a Banach space, and $(\Omega,\mathfrak{F},
 \{\mathfrak{F}_t\}_{t\geq0},\textbf{P})$ be a complete filtered probability space
supporting a cylindrical Wiener process  $\mathcal{W}(t)$ in $\ell^2$.
Then we can find a sparable Hilbert space $(\mathcal{U}, \langle \cdot, \cdot\rangle_{\mathcal{U}})$ such that
the embedding $\ell^2\hookrightarrow\mathcal{U}$ is Hilbert-Schmidt, and hence $\mathcal{W}(t)$ takes values in $\mathcal{U}$. Let $T>0$,
$(\mathfrak{E},\rho)$ and $(\mathfrak{E}_0,\rho_0)$ be two Polish spaces, and $\{Y^\epsilon_y: \epsilon>0,\ y\in\mathfrak{E}_0\}$ be a family of random variables taking values in $\mathfrak{E}$. We suppose that for every $\epsilon>0$ and $y\in\mathfrak{E}_0$, there exists a measurable map $\mathfrak{M}_y^\epsilon: C([0,T],\mathcal{U})\rightarrow\mathfrak{E}$.
such that $Y^\epsilon_y=\mathfrak{M}_y^\epsilon(\mathcal{W})$.
Let $\mathfrak{A}$ be the set of $\ell^2$-valued
progressively measurable processes
$\psi\in L^2([0,T],\ell^2 )$, $\mathbf{P}$-a.s.. For $N>0$,
we set
$$\overline{B}_N(L^2([0,T],\ell^2 )):=\bigg\{\psi\in L^2([0,T],\ell^2 ):
\bigg(\int_0^T\|\psi(s)\|^2_{\ell^2}ds\bigg)^{1/2}\leq N\bigg\}.$$
Note that $\overline{B}_N(L^2([0,T],\ell^2 ))$ is a closed ball in $L^2([0,T],\ell^2 )$ with center zero and radius $N$, and it is a Polish space
endowed with the weak topology. Based on the definitions of $\mathfrak{A}$ and  $\overline{B}_N(L^2([0,T],\ell^2 ))$, we define a subset of $\mathfrak{A}$ by
$$\mathfrak{A}_N:=\big\{\psi\in \mathfrak{A}:
\psi\in\overline{B}_N(L^2([0,T],\ell^2 )),\ \mbox{$\mathbf{P}$-a.s.} \big \}.$$

\begin{definition}(Rate function)
Given $y\in\mathfrak{E}_0$, we
say a map $\mathcal{I}_y:\mathfrak{E}\rightarrow[0,\infty]$
 is a rate
function if it is lower semi-continuous in $\mathfrak{E}$. A rate
function $\mathcal{I}_y:\mathfrak{E}\rightarrow[0,\infty]$
 is called a good one if for every $s\in[0,\infty)$, the level set $I^s_y:=\{ \psi\in \mathfrak{E}: \mathcal{I}_y(\psi)\leq s\}$ is a compact subset of $\mathfrak{E}$.
\end{definition}

\begin{definition}\label{LDP}
(Freidlin-Wentzell uniform LDP, \cite[Section 3.3]{Freidlin})
We say the family $\{Y^\epsilon_y: \epsilon>0,\ y\in\mathfrak{E}_0\}$ satisfies the Freidlin-Wentzell uniform LDP in $\mathfrak{E}$ with rate functions $\mathcal{I}_y:\mathfrak{E}
\rightarrow[0,\infty]$ uniformly
on $\mathfrak{E}_0$ if the following two
conditions are fulfilled:

(i) For every $\delta>0$ and $s>0$, we have
\begin{align*}
\liminf_{\epsilon\rightarrow0}
\inf_{y\in\mathfrak{E}_0}
\inf_{\psi\in I^s_y}\Big(\epsilon \ln
\mathbf{P}\big\{\rho(Y^\epsilon_y,\psi)<\delta\big\}         +\mathcal{I}_y(\psi)\Big)\geq0.
\end{align*}

(ii) For every $\delta>0$ and $s_0>0$, we have
\begin{align*}
\limsup_{\epsilon\rightarrow0}
\sup_{y\in\mathfrak{E}_0}
\sup_{s\in[0,s_0]}\Big(\epsilon \ln
\mathbf{P}\big\{ \dist(Y_y^\epsilon, I_y^s)\geq\delta       \big\}         +s\Big)\leq0.
\end{align*}
\end{definition}

\begin{definition}\label{LDP-}
(Dembo-Zeitouni uniform LDP \cite[Corollary 5.6.15]{Dembo})
We say the family $\{Y^\epsilon_y: \epsilon>0,\ y\in\mathfrak{E}_0\}$ satisfies the Dembo-Zeitouni uniform LDP in $\mathfrak{E}$ with rate functions $\mathcal{I}_y:\mathfrak{E}
\rightarrow[0,\infty]$ uniformly
on $\mathfrak{E}_0$ if the following two
conditions are fulfilled:

(i) For every open subset $G\subseteq\mathfrak{E}$, we have
\begin{align*}
\liminf_{\epsilon\rightarrow0}
\inf_{y\in\mathfrak{E}_0}\Big(\epsilon \ln
\mathbf{P}\big\{ Y^\epsilon_y\in G\big\}         \Big)\geq -\sup_{y\in\mathfrak{E}_0}\inf_{\psi\in G}     \mathcal{I}_y(\psi).
\end{align*}

(ii) For every closed subset $F\subseteq\mathfrak{E}$, we have
\begin{align*}
\limsup_{\epsilon\rightarrow0}
\sup_{y\in\mathfrak{E}_0}\Big(\epsilon \ln
\mathbf{P}\big\{ Y_y^\epsilon\in F      \big\}         \Big)\leq-\inf_{y\in\mathfrak{E}_0}\inf_{\psi\in F}     \mathcal{I}_y(\psi).
\end{align*}
\end{definition}

In general, the Freidlin-Wentzell and  Dembo-Zeitouni uniform LDPs
are not equivalent to each other. However, they are
equivalent to each other when $(\mathfrak{E}_0,\rho_0)$ is a \emph{compact} Polish space and the
level sets of the rate functions $I_y$
are continuous in $y\in\mathfrak{E}_0$ with respect to the Hausdorff metric of  $\mathfrak{E}$.

\begin{lemma}\label{Tattractor1-}
(Equivalence of the Freidlin-Wentzell and  Dembo-Zeitouni uniform LDPs, \cite[Theorem 2.7]{Salins2}) Suppose that $(\mathfrak{E}_0,\rho_0)$ is a compact Polish space, and that the
level sets of the rate functions $I_y$
are continuous in $y\in\mathfrak{E}_0$ with respect to the Hausdorff metric of $\mathfrak{E}$, that is, for every $s\in[0,\infty)$,
$$\lim_{n\rightarrow\infty}\max\bigg\{
\sup_{\psi\in I_y^s}\dist_{\mathfrak{E}}(\psi,I^s_{y_n}\big),
\sup_{\psi\in I_{y_n}^s}\dist_{\mathfrak{E}}(\psi,I^s_{y}\big)    \bigg\}=0,\ \ \mbox{as}\ \ \lim_{n\rightarrow\infty}\rho_0(y_n,y)=0.$$
Then the Freidlin-Wentzell uniform LDP
 and the Dembo-Zeitouni uniform LDP
are equivalent to each other.
\end{lemma}

Next, we recall the classical weak convergence theory for establishing the Freidlin-Wentzell uniform LDPs of $\{Y^\epsilon_y: \epsilon>0,\ y\in\mathfrak{E}_0\}$.

\begin{theorem}\label{Tattractor1}
(Weak convergence theory for the Freidlin-Wentzell uniform LDP, \cite[Theorem 2.13]{Salins2}) Given $y\in\mathfrak{E}_0$, we assume that there exists a measurable map
$\mathfrak{M}_y^0: C([0,T],\mathcal{U})\rightarrow\mathfrak{E}$
 such that
\begin{description}
  \item[(C1)] For every $N\in(0,\infty)$ and $\delta>0$,
\begin{align*}
\lim_{\epsilon\rightarrow0}
\sup_{y\in\mathfrak{E}_0}\sup_{v\in\mathfrak{A}_N}
\mathbf{P}\Bigg\{ \rho\Bigg( \mathfrak{M}_y^\epsilon\bigg(\mathcal{W}+
\frac{1}{\sqrt{\epsilon}}\int_0^{\cdot}v(s)ds\bigg)  , \ \mathfrak{M}_y^0
\bigg(
\int_0^{\cdot}v(s)ds\bigg)\Bigg)>\delta
\Bigg\}=0.
\end{align*}
\item[(C2)] For every $N\in(0,\infty)$ and $y\in\mathfrak{E}_0$, the set $$\Bigg\{\mathfrak{M}_y^0
\bigg(
\int_0^{\cdot}v(s)ds\bigg):v\in \overline{B}_N(L^2([0,T],\ell^2 ))\Bigg\}\ \ \mbox{is a compact subset of $\mathfrak{E}$}.$$
\end{description}
Then the family $\{Y^\epsilon_y: \epsilon>0,\ y\in\mathfrak{E}_0\}$ satisfies the Freidlin-Wentzell uniform LDP in $\mathfrak{E}$ uniformly on $\mathfrak{E}_0$ with good rate functions $\mathcal{I}_y:\mathfrak{E}
\rightarrow[0,\infty]$ defined by
\begin{align*}
\mathcal{I}_y(\psi)=\inf\left\{\frac{1}{2}
\int_{0}^{T}||v(s)||_{\ell^{2}}^{2}ds:
v\in L^{2}([0,T],\ell^{2})~
\text{such~that}~\mathfrak{M}_y^0
\bigg(
\int_0^{\cdot}v(s)ds\bigg)=\psi\right\},\ \ \forall\ \psi\in\mathfrak{E},
\end{align*}
where $\inf\emptyset=+\infty$.
\end{theorem}

\section{Uniform large deviations for Mckean-Vlasov
stochastic fractional $(\alpha,p)$-Laplacian equations on $\mathbb{R}^d$}

In this part, we establish the
Freidlin-Wentzell and Dembo-Zeitouni uniform large deviations
 for the Mckean-Vlasov
stochastic non-local \emph{fractional} $(\alpha,p)$-Laplacian equation
defined \eqref{Int19-} on $\mathbb{R}^d$ when the nonlinear diffusion term has superlinear
growth rates $p^*\in\big[2,\frac{p+2}{2}\big]\subset[2,p)$ and $q^*\in\big[2,\frac{q+2}{2}\big]\subset[2,q)$.

Condition \textbf{B1}. For every $k\in \mathbb{N}$,
$\sigma_k:\mathbb{R}^+\times \mathbb{R}^d\times\mathbb{R}\times
\mathcal{P}_2(\mathbb{H})\rightarrow \mathbb{R}$ satisfies that for
  all $t\in\mathbb{R}^+$,
$x\in\mathbb{R}^d$, $s,s_1,s_2\in\mathbb{R}$ and $\mu,\mu_1,\mu_2
\in\mathcal{P}_2(\mathbb{H})$,
\begin{align}\label{h2new}
|\sigma_k(t,x,s,\mu)|^2\leq \sigma_{1,k}(x)
|s|^{p^*}+\sigma_{2,k}(x)
|s|^{q^*}
+\sigma_{3,k}(x)\big(1+\mu\big(\|\cdot\|^2\big)\big),
\end{align}
and
\begin{align}\label{h2anew}
|\sigma_k(t,x,s_1,\mu_1)-\sigma_k(t,x,s_2,\mu_2)
|^2&\leq
\sigma^2_{4,k}(x)\big( 1+ |s_1|^{p^*-2}+ |s_2|^{p^*-2} +|s_1|^{q^*-2}
\nonumber\\&\ \ +|s_2|^{q^*-2}
\big)|s_1-s_2|^2 +\sigma_{5,k}(x)
d_{\mathcal{P}_2}^2(\mu_1,\mu_2),
\end{align}
where $p^*\in\big[2,\frac{p+2}{2}\big]$ and $q^*\in\big[2,\frac{q+2}{2}\big]$,
$\sigma_1
 \in
 \ell^{1}
 (\mathbb{N},
 L^{\frac{p}{p-p^*}}(\mathbb{R}^d)
 \bigcap
 L^{{\frac{p}{p-2(p^*-1)}}}(\mathbb{R}^d))$,
$\sigma_2
 \in
 \ell^{1}
 (\mathbb{N},L^{\frac{q}{q-q^*}}
 (\mathbb{R}^d)\bigcap
 L^{{\frac{q}{q-2(q^*-1)}}}(\mathbb{R}^d))$, $\sigma_3
 \in
 \ell^{1}
 (\mathbb{N},L^1(\mathbb{R}^d))$, $\sigma_4\in \ell^{1}
 (\mathbb{N},
 L^{\infty}(\mathbb{R}^d)
 \bigcap
 \ell^2
 (\mathbb{N},L^4(\mathbb{R}^d)
 \bigcap
 L^{{\frac{4p}{p-2(p^*-1)}}}
 (\mathbb{R}^d)\bigcap
 L^{{\frac{4q}{q-2(q^*-1)}}}
 (\mathbb{R}^d))$ and $\sigma_5
 \in
 \ell^{1}
 (\mathbb{N},L^{1}(\mathbb{R}^d))$.

One can verify that condition \textbf{B1}
implies condition \textbf{B}, and hence all results in the previous sections are also valid under conditions \textbf{A} and \textbf{B1}.
Since   \eqref{h2-}-\eqref{h2--} are
not sufficient for
establishing  the uniform LDPs of \eqref{Int19-}, we need
to improve both inequalities
under  condition \textbf{B1}.
Indeed,  for every $M>0$,  $T>0$,
  $t\in[0,T]$, $\psi,\psi_1,\psi_2\in
\mathbb{H}\cap L^p(\mathbb{R}^d)\cap \mathbb{V}_2$ and $\mu
\in\mathcal{P}_2(\mathbb{H})$, we
find    that
\begin{align}\label{B1-}
&M\|\mathfrak{L}_{H,S}(t,\psi,\mu)
\|^2_{\mathcal{L}_2(\ell^2
,\mathbb{H})}=
M\|\sigma(t,\cdot,\psi,\mu)\|^2_{
\ell^2(\mathbb{N},\mathbb{H})}
\nonumber\\&\leq
M\sum_{k\in\mathbb{N}}\int_{\mathbb{R}^d}
\Big(
\sigma_{1,k}(x)
|\psi(x)|^{p^*}+
\sigma_{2,k}(x)
|\psi(x)|^{q^*}
+
\sigma_{3,k}(x)
\big(1+\mu\big(\|\cdot\|^2\big)\big)\Big)dx
\nonumber\\
&\leq
 M
 \Bigg(
 \sum_{k\in\mathbb{N}}\|\sigma_{1,k} \|
 _{\frac{2p}{p-2(p^*-1)}}
 \| \psi\|^{p^*-1}_{p}\| \psi\|
 +\sum_{k\in\mathbb{N}}\|\sigma_{2,k} \|
 _{\frac{2q}{q-2(q^*-1)}}
 \| \psi\|^{q^*-1}_{q}\| \psi\|+\sum_{k\in\mathbb{N}}\|\sigma_{3,k} \|
 _{1}\big(1+\mu\big(\|\cdot\|^2
\big)\big)\Bigg)
\nonumber\\
&\leq C_6(M)(1+\| \psi\|^{p/2}_{p}+
 \| \psi\|^{q/2}_{q})\| \psi\|+C_6(M)\big(1+\mu\big(\|\cdot\|^2
\big)\big),
\end{align}
and
\begin{align}\label{B1--}
&M\|\mathfrak{L}_{H,S}(t,\psi_1,\mu)-
\mathfrak{L}_{H,S}(t,\psi_2,\mu)
\|^2_{\mathcal{L}_2(\ell^2
,\mathbb{H})}\nonumber\\&=M\|\sigma(t,\cdot,\psi_1,\mu)-
\sigma(t,\cdot,\psi_2,\mu)
\|^2_{
\ell^2(\mathbb{N},\mathbb{H})}
\nonumber\\&\leq M\sum_{k\in\mathbb{N}}\int_{\mathbb{R}^d}
\sigma^2_{4,k}(x)\big(1+ |\psi_1(x)|^{p^*-2}+ |\psi_2(x)|^{p^*-2} \nonumber\\& \ +|\psi_1(x)|^{q^*-2}
+|\psi_2(x)|^{q^*-2}
\big)|\psi_1(x)-\psi_2(x)|^2 dx
\nonumber\\&\leq
M\sum_{k\in\mathbb{N}}\int_{\mathbb{R}^d}
\sigma^2_{4,k}(x)
\big(|\psi_1(x)|+
|\psi_2(x)|
+ |\psi_1(x)|^{p^*-1}+ |\psi_2(x)|^{p^*-1} \nonumber\\& \ +|\psi_1(x)|^{q^*-1}
  +|\psi_2(x)|^{q^*-1}
 +|\psi_1(x)||\psi_2(x)|^{p^*-2}+
|\psi_2(x)||\psi_1(x)|^{p^*-2}
\nonumber\\& \ +|\psi_1(x)||\psi_2(x)|^{q^*-2}
 +
|\psi_2(x)||\psi_1(x)|^{q^*-2}
\big)|\psi_1(x)-\psi_2(x)|dx
\nonumber\\&\leq
4M\sum_{k\in\mathbb{N}}\int_{\mathbb{R}^d}
\sigma^2_{4,k}(x)
\big(1
+  |\psi_1(x)|^{p^*-1}+ |\psi_2(x)|^{p^*-1} \nonumber\\& \ +|\psi_1(x)|^{q^*-1}
 +|\psi_2(x)|^{q^*-1}
\big)|\psi_1(x)-\psi_2(x)| dx
\nonumber\\&\leq 4M
\bigg(  \sum_{k\in\mathbb{N}}\|\sigma_{4,k}
\|^{2}_{4} +\sum_{k\in\mathbb{N}}\|\sigma_{4,k} \|^2
 _{\frac{4p}{p-2(p^*-1)}}\Big(\|\psi_1\|_p^{p^*-1}
+\|\psi_2\|_p^{p^*-1} \Big)\nonumber\\&+
 \sum_{k\in\mathbb{N}}\|\sigma_{4,k} \|^2
 _{\frac{4q}{q-2(q^*-1)}}\Big(\|\psi_1\|_q^{q^*-1}
+\|\psi_2\|_q^{q^*-1} \Big)
     \bigg)
 \|\psi_1-\psi_2\|
\nonumber\\&\leq C_7(M)
\Big(1+\|\psi_1\|_p^{\frac{p}{2}}
+\|\psi_2\|_p^{\frac{p}{2}}
+\|\psi_1\|_q^{\frac{q}{2}}
 +\|\psi_2\|_q^{\frac{q}{2}}
   \Big)\|\psi_1-\psi_2\|,
\end{align}
where $C_6(M)$ and $C_7(M)$ are positive constants
depending on $M$.

In the rest of this paper, we will write the solution of \eqref{Int19-} and \eqref{Int1--} as  $\psi^\epsilon(t,\psi_0)$ in order to indicate  the dependence of the solution on the noise intensity $\epsilon\in(0,1)$, and we will use
  \eqref{h2-}-\eqref{h2--}
and \eqref{B1-}-\eqref{B1--} to establish the
Freidlin-Wentzell and Dembo-Zeitouni uniform LDPs for the family $\{\psi^\epsilon(\cdot,\psi_0): \epsilon>0,\ \psi_0\in \mathbb{H}\}$
  in $C([0,T],\mathbb{H} )\bigcap L^p([0,T], \mathbb{V}_1  )\bigcap L^q([0,T], \mathbb{V}_2)$.
\subsection{Global-in-time well-posedness of   controlled equations}
Given a control
$u\in L^2([0,T],\ell^2)$, we consider the
 controlled equation corresponding to problem  \eqref{Int1--}:
\begin{equation}\label{Int1--++}
\left\{
  \begin{aligned}
  &\frac{d \psi_u(t)}{dt}
=\sum_{i=1}^2\textbf{A}_i(t,\psi_u(t),\delta_{\psi^0(t)})
+ \mathfrak{L}_{H,S}(t,\psi_u(t),
\delta_{\psi^0(t)})u(t), \ \ t>0, \\
  &\psi_u(0)=\psi_0\in \mathbb{H},\\
  \end{aligned}
\right.
\end{equation}
where $\delta_{\psi^0(t)}$ is the Dirac measure\footnote{The Dirac
 measure $\delta_{x}$ at $x\in \mathbb{H}$ is defined by $\delta_{x}(A)=1$ if $x\in A$; and $\delta_{x}(A)=0$ if $x\notin A$, where $A\in \mathcal{B}(\mathbb{H})$.} at
 the solution $\psi^0(t)$ of the
 deterministic  equation:
\begin{equation}\label{Int1--++++}
\left\{
  \begin{aligned}
  &\frac{d \psi^0(t)}{dt}
=\sum_{i=1}^2\textbf{A}_i(t,\psi^0(t),\delta_{\psi^0(t)})
, \ \ t>0, \\
  &\psi^0(0)=\psi_0\in \mathbb{H}.\\
  \end{aligned}
\right.
\end{equation}

Next, we derive the global-in-time well-posedness and uniform estimates of solutions of \eqref{Int1--++} and \eqref{Int1--++++}.

\begin{theorem}\label{Con1}
Let conditions {\bf A} and {\bf B1} be satisfied.
Then we have:

(i) For every $T>0$ and $\psi_0\in \mathbb{H}$, problem \eqref{Int1--++++} has a unique solution $\psi^0(\cdot,\psi_0)\in
C([0,T],\mathbb{H} )\bigcap L^p([0,T], \mathbb{V}_1  )\\ \bigcap L^q([0,T], \mathbb{V}_2)$. In addition, for every $T>0$ and $R>0$, there exist $C_1(T)>0$ and $C_1(T,R)>0$ such that for all $\psi_{0}\in\mathbb{H}$ and $\psi_{0,1},\psi_{0,2}\in \overline{B}_{R}(\mathbb{H})$,
the solutions $\psi^0(\cdot,\psi_0)$, $\psi^0(\cdot,\psi_{0,1})$
and $\psi^0(\cdot,\psi_{0,2})$ of  \eqref{Int1--++++} satisfy that
\begin{align}\label{Con2}
&\|\psi^0(\cdot,\psi_0)\|^2_{C([0,T],\mathbb{H} )}+
\|\psi^0(\cdot,\psi_0)\|^p_{L^p([0,T],\mathbb{V}_1 )}+\|\psi^0(\cdot,\psi_0)\|^q_{L^q([0,T], \mathbb{V}_2)}\leq C_1(T)(1+\|\psi_0\|^2),
\end{align}
and
\begin{align}\label{Con3}&
\|\psi^0(\cdot,\psi_{0,1})-
\psi^0(\cdot,\psi_{0,2})
\|^2_{C([0,T],\mathbb{H} )}+
\|\psi^0(\cdot,\psi_{0,1})-
\psi^0(\cdot,\psi_{0,2}\|^p_{L^p([0,T], \mathbb{V}_1)}\nonumber\\&   +\|\psi^0(\cdot,\psi_{0,1})-
\psi^0(\cdot,\psi_{0,2}\|^q_{L^q([0,T], \mathbb{V}_2)}\leq C_1(T,R)\|
\psi_{0,1}-
\psi_{0,2}
\|^2.
\end{align}

(ii) For every $T>0$, $\psi_0\in \mathbb{H}$ and $u\in L^2([0,T],\ell^2)$, problem \eqref{Int1--++} has a unique solution $\psi_u(\cdot,\psi_0) \in
C([0,T],\mathbb{H} )\\ \bigcap L^p([0,T], \mathbb{V}_1  ) \bigcap L^q([0,T], \mathbb{V}_2)$. In addition, for every $T>0$ and $R>0$, there exists $C_2(T,R)>0$ such that for all $\psi_{0}\in\mathbb{H}$, $\psi_{0,1},\psi_{0,2}\in \overline{B}_{R}(\mathbb{H})$ and
$u,u_1,u_2\in \overline{B}_{R}(L^2([0,T],\ell^2))$,
the solutions $\psi_u(\cdot,\psi_0)$, $\psi_{u_1}(\cdot,\psi_{0,1})$ and $\psi_{u_2}(\cdot,\psi_{0,2})$ of \eqref{Int1--++} satisfy that
\begin{align}\label{Con4}
&\|\psi_u(\cdot,\psi_0)
\|^2_{C([0,T],\mathbb{H} )}+
\|\psi_u(\cdot,\psi_0)\|^p_{L^p([0,T],\mathbb{V}_1 )}+\|\psi_u(\cdot,\psi_0)\|^q_{L^q([0,T], \mathbb{V}_2)}\leq C_2(T,R)(1+\|\psi_0\|^2),
\end{align}
and
\begin{align}\label{Con5}&
\|\psi_{u_1}(\cdot,\psi_{0,1})-
\psi_{u_2}(\cdot,\psi_{0,2})\|^2_{C([0,T],
\mathbb{H} )}+
\|\psi_{u_1}(\cdot,\psi_{0,1})-
\psi_{u_2}(\cdot,\psi_{0,2})\|^p_{L^p([0,T], \mathbb{V}_1)}\nonumber\\&   +\|\psi_{u_1}(\cdot,\psi_{0,1})-
\psi_{u_2}(\cdot,\psi_{0,2})\|^q_{L^q([0,T], \mathbb{V}_2)}\leq C_2(T,R)
\Big(\|
\psi_{0,1}-
\psi_{0,2}
\|^2+\|u_1-u_2
\|^2_{L^2([0,T],\ell^2)}\Big).
\end{align}\end{theorem}
\begin{proof}
Following the proof of Theorem \ref{Theorem4}, one can prove the existence and uniqueness of solutions of \eqref{Int1--++} and \eqref{Int1--++++} for  every $\psi_0\in \mathbb{H}$ and $u\in L^2([0,T],\ell^2)$.
The proof of \eqref{Con2}-\eqref{Con3} is samilar to that of \eqref{Con4}-\eqref{Con5}, and so we only prove \eqref{Con4}-\eqref{Con5}.

Let $\psi_u(t):=\psi_u(t,\psi_0)$ and
$\psi^0(t):=\psi^0(t,\psi_0)$ for $t\in[0,T]$.
By \eqref{Int1--++}, we find that
\begin{align}\label{qqqqqq1}	\frac{1}{2}\frac{d}{dt}\|\psi_u(t)
\|^2
		&= \underbrace{\sum_{i=1}^2\ _{\mathbb{V}_i^*}\langle \mathbf{A}_i(t,\psi_u(t),\delta_{\psi^0(t)}),
\psi_u(t)\rangle_{\mathbb{V}_i}}_{:=\textbf{L}_1(t)}
     +\underbrace{\langle\mathfrak{L}_{H,S}(t,\psi_u(t),
\delta_{\psi^0(t)})u(t) ,\psi_u(t)\rangle}_{:=\textbf{L}_2(t)}.
\end{align}
By Lemma \ref{Dissipativeness},
we have
\begin{align}\label{qqqqqq2}
\textbf{L}_1(t)&\leq
-\frac{1}{2K}\| \psi_u(t)\|_{\dot{\mathbb{V}}_1}^p-
\beta\| \psi_u(t)\|_{p}^{p}-
\widehat{\beta}\| \psi_u(t)\|_{\mathbb{V}_2}^{q}
 +(\|\phi_1(t)
\|_{1}+\|\widehat{\phi}_1(t)
\|_{1})
\|\psi^0(t)\|^2
+\|\phi_2(t)
\|_{1}+\|\widehat{\phi}_2(t)
\|_{1}.
	\end{align}
By \eqref{h2-}(for $M=1/2$, $\varepsilon_1=\beta/2$ and
$\varepsilon_2=\widehat{\beta}/2$) and \eqref{ffnewA2a0}, we know that
\begin{align}\label{qqqqqq3}
\textbf{L}_2(t)&\leq
\|\mathfrak{L}_{H,S}(t,\psi_u(t),
\delta_{\psi^0(t)})\|_{_{\mathcal{L}_2(\ell^2
 ,\mathbb{H})}}\|u(t)\|_{\ell^2}\|\psi_u(t)\|
 \nonumber\\& \leq\frac{1}{2}\|u(t)\|_{\ell^2}^2
\|\psi_u(t)\|^2 +
\frac{\beta}{2}\| \psi_u(t)\|_{p}^{p}+
\frac{\widehat{\beta}}{2}\| \psi_u(t)\|_{\mathbb{V}_2}^{q}+
 C_2(1/2,\beta/2,\widehat{\beta}/2)
\big(1+\|\psi^0(t)\|^2\big).
	\end{align}
Then it follows from \eqref{Con2} and
\eqref{qqqqqq1}-\eqref{qqqqqq3} that there exists a constant $K_1(T)>0$ such that for all $\psi_{0}\in \mathbb{H}$, $u\in \overline{B}_{R}(L^2([0,T],\ell^2))$ and $t\in [0,T]$,
\begin{align}\label{qqqqqq4}	\frac{d}{dt}\|\psi_u(t)
\|^2+(K^{-1}\wedge\beta)\| \psi_u\|_{\mathbb{V}_1}^p+
\widehat{\beta}\| \psi_u\|_{\mathbb{V}_2}^q&
\leq\|u(t)\|_{\ell^2}^2
\|\psi_u(t)\|^2+K_1(T)(1+\|\psi_0\|^2)\mathbb{G}(t),
\end{align}
where $\mathbb{G}(t):=1+ \|\phi_1(t)
\|_{1}+\|\widehat{\phi}_1(t)
\|_{1}
+\|\phi_2(t)
\|_{1}+\|\widehat{\phi}_2(t)
\|_{1}$ is an integrable function
on $[0,T]$. Then \eqref{Con4} follows from
\eqref{qqqqqq4}.
	
Let
$\psi_{1}(t):=\psi_{u_1}(t,\psi_{0,1})$ and $\psi_{2}(t):=\psi_{u_2}(t,\psi_{0,2})$ for $t\in[0,T]$. Then by \eqref{Int1--++}, we find that
\begin{align}\label{qqqqqq5}	\frac{1}{2}\frac{d}{dt}\|\psi_1(t)-
\psi_2(t)
\|^2
		&= \underbrace{\sum_{i=1}^2\ _{\mathbb{V}_i^*}\langle \mathbf{A}_i(t,\psi_1(t),\delta_{\psi_1^0(t)})
-\mathbf{A}_i(t,\psi_2(t),\delta_{\psi_2^0(t)})
,
\psi_1(t)-\psi_2(t)
\rangle_{\mathbb{V}_i}}_{:=\textbf{L}_3(t)}
 \nonumber\\
 &\
    +\underbrace{\langle \mathfrak{L}_{H,S}(t,\psi_2(t),
\delta_{\psi_2^0(t)})\big(u_1(t)-u_2(t)\big) ,\psi_1(t)-\psi_2(t)
\rangle}_{:=\textbf{L}_4(t)}
\nonumber\\
&+
\underbrace{\langle(
    \mathfrak{L}_{H,S}(t,\psi_1(t),
\delta_{\psi_1^0(t)}) - \mathfrak{L}_{H,S}(t,\psi_2(t),
\delta_{\psi_2^0(t)}))u_1(t) ,\psi_1(t)-\psi_2(t)
\rangle}_{:=\textbf{L}_5(t)}
.
\end{align}
By the strong monotonicity of $\mathbf{A}_1$ and $\mathbf{A}_2$ in Lemma \ref{Monotonicity}, one can deduce that
\begin{align}\label{qqqqqq6}	
\textbf{L}_3(t)\leq&-\frac{1}{2^{p+1}K}\|\psi_1(t)-
\psi_2(t)\|^p_{\dot{\mathbb{V}}_1}-
2^{-p}\beta\|\psi_1(t)-\psi_2(t)\|^p_p
-2^{-q}\widehat{\beta}
\|\psi_1(t)-\psi_2(t)\|_{\mathbb{V}_2}^{q}
\nonumber\\& -\frac{\beta}{4} \int_{\mathbb{R}^d}  \big(|\psi_1(t,x)|^{p-2}+|\psi_2(t,x)|^{p-2}\big)
|\psi_1(t,x)-\psi_2(t,x)|^2dx
\nonumber\\& -\frac{\widehat{\beta}}{4} \int_{\mathbb{R}^d}  \big(|\psi_1(t,x)|^{q-2}+|\psi_2(t,x)|^{q-2}\big)
|\psi_1(t,x)-\psi_2(t,x)|^2dx
\nonumber\\&+(\|\phi_6(t)\|_{\infty}+\|\widehat{\phi}_6(t)\|_{\infty})
\|\psi_1(t)-\psi_2(t)\|^2
+(\|\phi_7(t)\|_{1}+\|\widehat{\phi}_7(t)\|_{1})
\|\psi_1^0(t)-\psi_2^0(t)\|^2.
\end{align}
It follows from \eqref{h2-}(for $M=1/2$, $\varepsilon_1=1/2$ and
$\varepsilon_2=1/2$) and \eqref{ffnewA2a0} that
\begin{align}\label{qqqqqq7}
\textbf{L}_4(t)&\leq
\frac{1}{2}
\| \mathfrak{L}_{H,S}(t,\psi_2(t),
\delta_{\psi^0_2(t)})\|_{\mathcal{L}_2(\ell^2
 ,\mathbb{H})}^2\|\psi_1(t)-\psi_2(t) \|^2
+
\frac{1}{2}\|u_1(t)-u_2(t)\|_{\ell^2}^2
\nonumber\\&\leq
\Big(\frac{1}{2}\| \psi_2(t)\|_{p}^p+
\frac{1}{2}\| \psi_2(t)\|_{\mathbb{V}_2}^{q}+
C_2(1/2,1/2,1/2)
\big(1+\|\psi^0_2(t)\|^2
\big)\Big)\|\psi_1(t)-\psi_2(t) \|^2
\nonumber\\&\ +\frac{1}{2}\|u_1(t)-u_2(t)\|_{\ell^2}^2.
\end{align}
By \eqref{h2--} (for $M=1/2$, $\varepsilon_1=\beta/4$ and
$\varepsilon_2=\widehat{\beta}/4$) and \eqref{h2----0}, we find that
\begin{align}\label{qqqqqq8}
\textbf{L}_5(t)&\leq
\frac{1}{2}\| \mathfrak{L}_{H,S}(t,\psi_1(t),
\delta_{\psi^0_1(t)})-\mathfrak{L}_{H,S}(t,\psi_2(t),
\delta_{\psi^0_2(t)})
\|_{_{\mathcal{L}_2(\ell^2
 ,\mathbb{H})}}^2+\frac{1}{2}
\|u_1(t)\|_{\ell^2}^2\|\psi_1(t)-\psi_2(t) \|^2
\nonumber\\&\leq\frac{\beta}{4} \int_{\mathbb{R}^d}  \big(|\psi_1(t,x)|^{p-2}+|\psi_2(t,x)|^{p-2}\big)
|\psi_1(t,x)-\psi_2(t,x)|^2dx
\nonumber\\& +\frac{\widehat{\beta}}{4} \int_{\mathbb{R}^d}  \big(|\psi_1(t,x)|^{q-2}+|\psi_2(t,x)|^{q-2}\big)
|\psi_1(t,x)-\psi_2(t,x)|^2dx
\nonumber\\&+\bigg(\frac{1}{2}
\|u_1(t)\|_{\ell^2}^2+C_4(1/2,\beta/4,
\widehat{\beta}/4)
\bigg)\|\psi_1(t)-\psi_2(t) \|^2
+C_5(1/2)\|\psi_1^0(t)-\psi_2^0(t) \|^2 .
\end{align}

As a cnsequence of \eqref{Con2}-\eqref{Con3}
 and \eqref{qqqqqq5}-\eqref{qqqqqq8}, we find that there exists $K_2(R, T)>0$ such that
for all $\psi_{0,1},\psi_{0,2}\in \overline{B}_{R}(\mathbb{H})$ and
$u_1,u_2\in \overline{B}_{R}(L^2([0,T],\ell^2))$,
\begin{align}\label{qqqqqq9}	& \frac{d}{dt}\|\psi_1(t)-
\psi_2(t)\|^2+\Big(\frac{1}{2^pK}\wedge\frac{\beta}{2^{p-1}}
\Big)\|\psi_1(t)-
\psi_2(t)\|^p_{\mathbb{V}_1}
+2^{1-q}\widehat{\beta}
\|\psi_1(t)-\psi_2(t)\|_{\mathbb{V}_2}^{q}
\nonumber\\&
 \leq \big(\| \psi_2(t)\|_{\mathbb{V}_1}^p+
\| \psi_2(t)\|_{q}^{q}+\|u_1(t)\|_{\ell^2}^2+
K_2\widehat{\mathbb{G}}(t)\big)\|\psi_1(t)-\psi_2(t) \|^2
\nonumber\\&
+
K_2 \|\psi_{0,1} -\psi_{0,2}
 \|^2
+\|u_1(t)-u_2(t)\|_{\ell^2}^2,
\end{align}
where
$\widehat{\mathbb{G}}(t):=1 +
\|\phi_6(t)\|_{\infty}
+\|\widehat{\phi}_6(t)\|_{\infty}
+
  \| \phi_7(t)\|_1
+ \| \widehat{\phi}_7(t)\|_1
 $ in an integrable function
on $[0,T]$. Then \eqref{Con5} follows from
\eqref{qqqqqq9} and \eqref{Con4}. The proof is completed.
\end{proof}

\subsection{Uniform tail-ends estimates for deterministic controlled equations}
We derive uniform
tail-ends estimates of solutions to
\eqref{Int1--++} for the purpose of showing
the continuity of the solution $\psi_u$ of \eqref{Int1--++} in $u$
from the weak topology of
$L^2([0,T],\ell^2)$ to the strong topology of  $C([0,T],\mathbb{H} )\bigcap L^p([0,T], \mathbb{V}_1  ) \bigcap L^q([0,T], \mathbb{V}_2)$. The idea of uniform tail-ends estimates developed in \cite{Wangphysd} has
been used to discuss other problems with lack of compactness such as the existence and stability of attractors and invariant measures
 of infinite-dimensional systems
 \cite{chenp0, Chenpa,Caraballo4,
Gu2018jde,Kinra24,LIding1,Wang2011Tran,wangjfa,
wangjde2019,WangGUOWANG,Wang-wangb,
Wang-wangb2}.

Let $\rho:\mathbb{R}^d\mapsto[0,1]$ be a smooth function such that
$\rho(x)=0$ for $|x|\leq1/2$; and
$\rho(x)=1$ for $|x|\geq1$. Denote by $\rho_n(x):=\rho(x/n)$ for $x\in\mathbb{R}^d$ and $n\in\mathbb{N}$.
By
the definitions of $\rho_n$ and $\mathcal{K}_{p}^\alpha$, as in \cite{Wang-wangb,Wang-wangb2,RWangsubmitted},
one can verify:

\begin{lemma}\label{Prpgiikgk.jkj}
For the cut-off function $\rho_n$ and the kernel function
$\mathcal{K}_{p}^\alpha$, we have:
\begin{align*}&
\sup_{x\in\mathbb{R}^d}\int_{\mathbb{R}^d}
|\rho_n(x)-
\rho_n(y)|^p
\mathcal{K}_{p}^\alpha(x,y)dy\leq
C_{p,\alpha}n^{-p\alpha},
\\&\sup_{y\in\mathbb{R}^d}
\int_{\mathbb{R}^d}|\rho_n(x)
-\rho_n(y)|^p
\mathcal{K}_{p}^\alpha(x,y)dx\leq C_{p,\alpha}n^{-p\alpha},
\end{align*}
and for any $\psi\in\mathbb{V}_1$,
\begin{align*}&
-\int_{\mathbb{R}^d}\int_{\mathbb{R}^d}
|\psi(x)-\psi(y)|^{p-2}\big
(\psi(x)-\psi(y)\big)
\big(\rho^2_n(x)\psi(x)-
\rho^2_n(y)\psi(y)\big)
\mathcal{K}_p^\alpha(x,y)
\\&\leq C_{p,\alpha}
n^{-\alpha}
\|\psi\|^{p
}_{\mathbb{V}_1}
-\int_{\mathbb{R}^d}\int_{\mathbb{R}^d}
\rho^2_n(x)|\psi(x)-\psi(y)|^{p}
\mathcal{K}_p^\alpha(x,y)
 dxdy,\end{align*}
where $C_{p,\alpha}>0$ is a constant independent of $n$ and $\psi$.
\end{lemma}

By using the results in Lemma \ref{Prpgiikgk.jkj} and the idea of uniform tail-ends
estimates, we next show that the solution $\psi_u$ of \eqref{Int1--++} is uniformly small in $L^2(\mathbb{R}^d)$ when the space variable is sufficiently large.

\begin{lemma}\label{tightness2ddfhtt}
Let conditions  {\bf A} and {\bf B1}
be satisfied. Then for every  $\delta>0$, $T>0$, $R>0$
and compact subset
$\mathcal{K}\subseteq \mathbb{H}$,
there exists $N:=N(\delta,T,R,\mathcal{K})
\in\mathbb{N}$ such that the solution $\psi_u$ of \eqref{Int1--++} satisfies
\begin{align*}&
\sup_{n\geq N}\sup_{t\in[0,T]}\sup_{u\in \overline{B}_{R}(L^2([0,T],\ell^2))}
\sup_{\psi_0\in\mathcal{K}}\int_{ {\mathbb{R}}^d \setminus
\mathcal{Q}_n}|\psi_u(t,\psi_0,x)|^2
dx
<\delta,
\end{align*}
where $\mathcal{Q}_n:=\{x\in\mathbb{R}^d:|x|< n\}$.
\end{lemma}

\begin{proof}
By \eqref{Int1--++}, we find
\begin{align}\label{qqqqqq164644}	\frac{1}{2}\frac{d}{dt}\|\rho_n\psi_u(t)
\|^2
		&= \underbrace{\sum_{i=1}^2\ _{\mathbb{V}_i^*}\langle \mathbf{A}_i(t,\psi_u(t),\delta_{\psi^0(t)}),
\rho^2_n\psi_u(t)\rangle_{\mathbb{V}_i}
}_{:=\textbf{L}_6(t,n)}
     +\underbrace{\langle\mathfrak{L}_{H,S}(t,\psi_u(t),
\delta_{\psi^0(t)})u(t) ,\rho^2_n\psi_u(t)\rangle}_{:=\textbf{L}_7(t,n)}.
\end{align}
By \eqref{f1}, \eqref{ff1} and Lemma \ref{Prpgiikgk.jkj}, one can infer that the term $\textbf{L}_6(t,n)$ with the nonlinear non-local
fractional $(\alpha,p)$-Laplace
operator $\mathfrak{L}_{\mathcal{K}_{p}^\alpha}$ satisfies that
\begin{align}\label{qqqqqq164644-}
\textbf{L}_6(t,n)&\leq-
 \int_{\mathbb{R}^d}\rho^2_n(x)(
 \beta|\psi_u(t,x)|^p+\widehat{\beta}
 |\psi_u(t,x)|^q)dx
\nonumber \\&\ \ -\frac{1}{2}\int_{\mathbb{R}^d}\int_{\mathbb{R}^d}
\rho^2_n(x)|\psi_u(t,x)-\psi_u(t,y)|^{p}
\mathcal{K}_p^\alpha(x,y)dxdy
 \nonumber \\&\
\ +\frac{1}{2} C_{p,\alpha}
n^{-\alpha}
\|\psi_u(t)\|^{p
}_{\mathbb{V}_1}
 +
 \int_{|x|\geq n/2}
 \big(|\phi_1(t,x)|
 +|\widehat{\phi}_1(t,x)| +|\phi_2(t,x)|
 +|\widehat{\phi}_2(t,x)|
 \big)dx\big(1+\|\psi^0(t)\|^2\big).
\end{align}
Using the arguments of  \eqref{h2-} and \eqref{ffnewA2a0}, by \eqref{h2}, we find that the term $\textbf{L}_7(t,n)$ involing the superlinearly growing diffusion term is bounded by
\begin{align}\label{qqqqqq164644--}
\textbf{L}_7(t,n)&\leq\frac{1}{2}\|\rho_n\psi_u(t)
\|^2\|u(t)\|^2_{\ell^2}+\frac{1}{2}\|
\rho_n\mathfrak{L}_{H,S}(t,\psi_u(t),
\delta_{\psi^0(t)})\|^2_{\mathcal{L}_2(\ell^2
 ,\mathbb{H})}
\nonumber \\&=\frac{1}{2}\|\rho_n\psi_u(t)
\|^2\|u(t)\|^2_{\ell^2}+\frac{1}{2}
\|\rho_n\sigma(\cdot,
\psi_u(t),
\delta_{\psi^0(t)})
\|^2_{
\ell^2(\mathbb{N},\mathbb{H})}
\nonumber \\&=\frac{1}{2}\|\rho_n\psi_u(t)
\|^2\|u(t)\|^2_{\ell^2}+\frac{1}{2}
\sum_{k\in\mathbb{N}}
\int_{\mathbb{R}^d}
\rho^2_n(x)
|\sigma_k(x,\psi_u(t,x),\delta_{\psi^0(t)})|^2dx
\nonumber \\& \leq\frac{1}{2}\|\rho_n\psi_u(t)
\|^2\|u(t)\|^2_{\ell^2}+
 \int_{\mathbb{R}^d}\rho^2_n(x)\Big(
 \frac{\beta}{2}|\psi_u(t,x)|^p+
 \frac{\widehat{\beta}}{2}
 |\psi_u(t,x)|^q\Big)dx
 + K_1\mathbb{F}_n
 (1+\|\psi^0(t)\|^2),
\end{align}
where $K_1>0$ is a constant independent of $n$, and
\begin{align*}
\mathbb{F}_n:&=\Bigg(\sum_{k\in\mathbb{N}}\bigg(\int_{|x|\geq n/2}
|\sigma_{1,k}(x)|^{\frac{p}{p-p^*}}
 dx\bigg)^{\frac{p-p^*}{p}}
 \Bigg)^{\frac{p}{p-p^*}}
\nonumber \\& \ +\Bigg(\sum_{k\in\mathbb{N}}\bigg(\int_{|x|\geq n/2}
|\sigma_{2,k}(x)|^{\frac{q}{q-q^*}}
 dx\bigg)^{\frac{q-p^*}{q}}
 \Bigg)^{\frac{q}{q-q^*}}+\sum_{k\in\mathbb{N}}\int_{|x|\geq n/2}
|\sigma_{3,k}(x)|
 dx.
\end{align*}

As a consequence of \eqref{qqqqqq164644}-\eqref{qqqqqq164644--}, we find that there exists a constant $K_2(\mathcal{K})>0$, indepenent of $n$, such that
for any $t\in[0,T]$, $\psi_0\in\mathcal{K}$ and $u\in \overline{B}_{R}(L^2([0,T],\ell^2))$,
\begin{align}\label{qqqqqq164644---}	\frac{d}{dt}\|\rho_n\psi_u(t)
\|^2
&\leq\|u(t)\|^2_{\ell^2}\|\rho_n\psi_u(t)
\|^2+C_{p,\alpha}
n^{-\alpha}
\|\psi_u(t)\|^{p
}_{\mathbb{V}_1}
 \nonumber \\&\ +K_2(\mathcal{K})\mathbb{F}_n+
K_2(\mathcal{K})
\int_{|x|\geq n/2}
 \big(|\phi_1(t,x)|
 +|\widehat{\phi}_1(t,x)|+|\phi_2(t,x)|
 +|\widehat{\phi}_2(t,x)|
 \big)dx.
\end{align}
This further shows that for any $t\in[0,T]$, $\psi_0\in\mathcal{K}$ and $u\in \overline{B}_{R}(L^2([0,T],\ell^2))$,
\begin{align}\label{cutInt3}
 \mathbb{E}\Big[
\|\rho_n\psi_u(t)\|^{2}\Big]
&\leq e^{\int_0^T\|u(t)\|^2_{\ell^2}dt}
\bigg(\|\rho_n\psi_0\|^{2}+T
K_2(\mathcal{K})\mathbb{F}_n
+C_{p,\alpha}
n^{-\alpha}\int_0^T
\|\psi_u(t)\|^{p
}_{\mathbb{V}_1}dt
\nonumber \\&+K_2(\mathcal{K})\int_0^T
\int_{|x|\geq n/2}
 \big(|\phi_1(t,x)|
 +|\widehat{\phi}_1(t,x)|+|\phi_2(t,x)|
 +|\widehat{\phi}_2(t,x)|
 \big)dxdt\bigg).
\end{align}
Given $\delta>0$. By the compactness of $\mathcal{K}\subseteq \mathbb{H}$, we find that there exists $N_1:=N_1(\delta, R,\mathcal{K})\in\mathbb{N}$ such that for all $\psi_0\in\mathcal{K}$, $u\in \overline{B}_{R}(L^2([0,T],\ell^2))$ and $n\geq N_1$,
\begin{align}\label{cutInt4}
e^{\int_0^T\|u(t)\|^2_{\ell^2}dt}
\|\rho_n\psi_0\|^{2}
\leq e^{R^2}\int_{|x|\geq n/2}
|\psi_0(x)|^{2}dx<\delta/4.
\end{align}
Since $\sigma_1\in
 \ell^{1}
 (\mathbb{N},L^{\frac{p}{p-p^*}}(\mathbb{R}^d))$,
$\sigma_2
 \in
 \ell^{1}
 (\mathbb{N},L^{\frac{q}{q-q^*}}(\mathbb{R}^d))$ and
 $\sigma_3
 \in
 \ell^{1}
 (\mathbb{N},L^1(\mathbb{R}^d))$, we find that
 $\mathbb{F}_n\rightarrow0$ as $n\rightarrow\infty$. This shows that
 there exists $N_2:=N_2(\delta, T, R,\mathcal{K})\geq N_1$ such that for all $u\in \overline{B}_{R}(L^2([0,T],\ell^2))$ and $n\geq N_2$,
\begin{align}\label{cutInt4-}
TK_2(\mathcal{K})
e^{\int_0^T\|u(t)\|^2_{\ell^2}dt}
\mathbb{F}_n\leq
TK_2(\mathcal{K})e^{R^2}\mathbb{F}_n<\delta/4.
\end{align}
Using \eqref{Con4},
we find that there exists $N_3:=N(\delta,T,R,\mathcal{K})\geq N_2$ such that for all $\psi_0\in\mathcal{K}$, $u\in \overline{B}_{R}(L^2([0,T],\ell^2))$ and $n\geq N_3$,
\begin{align}\label{cutInt5}
C_{p,\alpha}
n^{-\alpha}
e^{\int_0^T\|u(t)\|^2_{\ell^2}dt}\int_0^T
\|\psi_u(t)\|^{p
}_{\mathbb{V}_1}dt\leq C_{p,\alpha}e^{R^2}
C(T,R,\mathcal{K})
n^{-\alpha}
<\delta/4,
\end{align}
where $C(T,R,\mathcal{K})>0$ is a constant independent of $n$.
Since $\phi_1,\widehat{\phi}_1
 \in  L_{loc}^\infty(\mathbb{R}^+,
 L^1(\mathbb{R}^d))$ and $\phi_2,\widehat{\phi}_2\in L_{loc}^1(\mathbb{R}^+,
   L^1(\mathbb{R}^d))$, we find that  there exists $N_4:=N(\delta,T,R,\mathcal{K})\geq N_3$ such that for all $u\in \overline{B}_{R}(L^2([0,T],\ell^2))$ and $n\geq N_4$,
\begin{align}\label{cutInt6}
K_2(\mathcal{K})
e^{\int_0^T\|u(t)\|^2_{\ell^2}dt}
\int_0^T
\int_{|x|\geq n/2}
 \big(|\phi_1(t,x)|
 +|\widehat{\phi}_1(t,x)|+|\phi_2(t,x)|
 +|\widehat{\phi}_2(t,x)|
 \big)dxdt
<\delta/4.
\end{align}
As a result of \eqref{qqqqqq164644---}-\eqref{cutInt6},
we obtain the  desired uniform estimates..
\end{proof}

\subsection{Weak-to-strong continuity
for deterministic controlled equations
}
In this part, we establish the weak-to-strong
continuity of the solution
$\psi_u$ of \eqref{Int1--++} with respect
to the control
$u$ from the weak toplogy of
$L^2([0,T],\ell^2)$ to the strong toplogy of $C([0,T],\mathbb{H} )\bigcap L^p([0,T], \mathbb{V}_1  ) \bigcap L^q([0,T], \mathbb{V}_2)$. This is a key step to establish the uniform LDPs of fractional $(\alpha,p)$-Laplacian equation \eqref{Int19-} on $\mathbb{R}^d$ with polynomial drift terms and  superlinear diffusion terms.

\begin{lemma}\label{Weak-to-strog}
Let conditions \textbf{A} and \textbf{B1}
be satisfied. Let $\psi_u$ and
$\psi_{u_n}$ ($n\in\mathbb{N}$) be two solutions of \eqref{Int1--++} with controls
$u$ and $u_n$ in $L^2([0,T],\ell^2)$. If
$u_n\rightarrow u$ weakly in $L^2([0,T],\ell^2)$, then \begin{align}\label{weak000}\psi_{u_n}\rightarrow\psi_u \ \mbox{strongly in}\
C([0,T],\mathbb{H} )\bigcap L^p([0,T], \mathbb{V}_1  ) \bigcap L^q([0,T], \mathbb{V}_2).
\end{align}
\end{lemma}
\begin{proof}
Let
$u_n\rightarrow u$ weakly in $L^2([0,T],
\ell^2)$. We show \eqref{weak000}
step by step as follows.

\underline{\emph{Step 1}}: \emph{weak convergence}. Since
$u_n\rightarrow u$ weakly in $L^2([0,T],
\ell^2)$, we know that $\{u_n\}_{n=1}^\infty$ is bounded in $L^2([0,T],
\ell^2)$. From this and Theorem \ref{Con1}, we known that $\{\psi_{u_n}\}_{n=1}^\infty$ is bounded in $C([0,T],\mathbb{H} )\bigcap L^p([0,T], \mathbb{V}_1  )\\ \bigcap L^q([0,T], \mathbb{V}_2)$. This, together with \eqref{h2-}, \eqref{ffnewA2a0}, Lemma \ref{Dissipativeness} and Theorem \ref{Con1}, implies that $\{\mathbf{A}_1(\cdot,\psi_{u_n},
\delta_{\psi^0})\}_{n=1}^\infty$, $\{\mathbf{A}_2(\cdot,\psi_{u_n},
\delta_{\psi^0})\}_{n=1}^\infty$ and $\{\mathfrak{L}_{H,S}(\cdot,\psi_{u_n},
\delta_{\psi^0})\}_{n=1}^\infty$ are bounded in
$L^{\widehat{p}}([0,T], \mathbb{V}_1^*)$, $L^{\widehat{q}}([0,T], \mathbb{V}_2^*)$ and $L^2([0,T],\mathcal{L}_2(\ell^2,
\mathbb{H}))$, respectively. Thus, we can find $z_{_{T}}\in H$, $z\in L^\infty([0,T],\mathbb{H} )\bigcap L^p([0,T], \mathbb{V}_1  ) \bigcap L^q([0,T],
\mathbb{V}_2)$, $\mathfrak{A}_1\in L^{\widehat{p}}([0,T], \mathbb{V}^*_1)$, $\mathfrak{A}_2\in L^{\widehat{q}}([0,T], \mathbb{V}^*_2)$, $\widehat{\mathfrak{L}}_{H,S}
\in L^2([0,T],\mathcal{L}_2(\ell^2,
\mathbb{H}))$, and a subsequence of $\{\psi_{u_n}\}_{n=1}^\infty$ (not relabeled) such that
\begin{subequations}
\begin{align}
\label{weak0}&
\psi_{u_n}(T)\rightarrow z_{_{T}} \ \
\mbox{weakly in $H$},
\\&\label{weak1}
\psi_{u_n}\rightarrow z\ \
			\mbox{weak-star in $L^\infty([0,T],\mathbb{H} )
				$},
			\\&\label{weak2}
			\psi_{u_n}\rightarrow z\ \
\mbox{weakly in $L^p([0,T], \mathbb{V}_1)$},		\\&\label{weak2a}\psi_{u_n}\rightarrow z\ \  \mbox{weakly in $L^q([0,T], \mathbb{V}_2  )$},
			\\& \label{weak3}			\mathbf{A}_1(\cdot,\psi_{u_n},
\delta_{\psi^0}) \rightarrow \mathfrak{A}_1\ \  \mbox{weakly in $ L^{\widehat{p}}([0,T], \mathbb{V}^*_1)$},
\\&\label{weak4}	 \mathbf{A}_2(\cdot,\psi_{u_n},
\delta_{\psi^0}) \rightarrow \mathfrak{A}_2\ \  \mbox{weakly in $ L^{\widehat{q}}([0,T], \mathbb{V}^*_2)$},
		\end{align}
		and
		\begin{align}
			\label{weak5}
		\mathfrak{L}_{H,S}(\cdot,\psi_{u_n},
\delta_{\psi^0})\rightarrow \widehat{\mathfrak{L}}_{H,S} \ \ \mbox{weakly in $L^2([0,T],\mathcal{L}_2(\ell^2,
\mathbb{H}))$}.
		\end{align}
	\end{subequations}

\underline{\emph{Step 2}}: \emph{equicontinuity of $\{\psi_{u_n}\}_{n=1}^\infty$ in $\big(\mathbb{H}\bigcap\mathbb{V}_1
\bigcap\mathbb{V}_2\big)^*$ on $[0,T]$}.
Since $\psi_{u_n}$ is the solution of \eqref{Int1--++}, we find that for all $t\geq0$,
\begin{align*}
&\psi_{u_n}(t)=\psi_0 +\sum_{i=1}^2\int_0^t
 \mathbf{A}_i(s,\psi_{u_n}(s),
\delta_{\psi^0(s)})ds
 + \int_0^t\mathfrak{L}_{H,S}(s,\psi_{u_n}(s),
\delta_{\psi^0(s)})u_n(s)ds\ \ \mbox{in}\ \ \big(\mathbb{H}\cap\mathbb{V}_1
\cap\mathbb{V}_2\big)^*.
\end{align*}
From this equality, we find that for all $0\leq t_1\leq t_2\leq T$,
\begin{align} \label{weak6}
\|\psi_{u_n}(t_2)-\psi_{u_n}(t_1)
\|_{(\mathbb{H}\bigcap\mathbb{V}_1
\bigcap\mathbb{V}_2)^*}&\leq
\sum_{i=1}^2\int_{t_1}^{t_2}
 \|\mathbf{A}_i(s,\psi_{u_n}(s),
\delta_{\psi^0(s)})\|_{\mathbb{V}_i^*}ds
\nonumber\\&\ + \int_{t_1}^{t_2}\|\mathfrak{L}_{H,S}(s,\psi_{u_n}(s),
\delta_{\psi^0(s)})u_n(s)
\|ds
\nonumber\\&\leq(t_2-t_1)^{1/p}
 \bigg(\int_{t_1}^{t_2}\|\mathbf{A}_1(s,\psi_{u_n}(s),
\delta_{\psi^0(s)})\|^{\widehat{p}
}_{\mathbb{V}_1^*}ds\bigg)^{1/\widehat{p}}
\nonumber\\&\ +(t_2-t_1)^{1/q}
\bigg(\int_{t_1}^{t_2}\|\mathbf{A}_2(s,\psi_{u_n}(s),
\delta_{\psi^0(s)})\|^{\widehat{q}
}_{\mathbb{V}_2^*}ds\bigg)^{1/\widehat{q}}
\nonumber\\&\
 + \bigg(\int_{t_1}^{t_2}\|
 \mathfrak{L}_{H,S}(s,\psi_{u_n}(s),
\delta_{\psi^0(s)})\|_{\mathcal{L}_2(\ell^2
 ,\mathbb{H})}^2ds
\bigg)^{1/2}
\bigg(\int_{t_1}^{t_2}\|u_n(s)
\|_{\ell^2}^2ds\bigg)^{1/2}
\nonumber\\&\leq(t_2-t_1)^{1/p}
\|\mathbf{A}_1(\cdot,\psi_{u_n},
\delta_{\psi^0})\|_{L^{\widehat{p}}([0,T], \mathbb{V}^*_1)}
\nonumber\\&\ +(t_2-t_1)^{1/q}
\|\mathbf{A}_2(\cdot,\psi_{u_n},
\delta_{\psi^0})\|_{L^{\widehat{q}}([0,T], \mathbb{V}^*_2)}
\nonumber\\&\
 + \bigg(\underbrace{\int_{t_1}^{t_2}\|
 \mathfrak{L}_{H,S}(s,\psi_{u_n}(s),
\delta_{\psi^0(s)})\|_{\mathcal{L}_2(\ell^2
 ,\mathbb{H})}^2ds
}_{:=\textbf{L}_8(t_1,t_2,n)}\bigg)^{1/2}
\|u_n
\|_{L^2([0,T],
\ell^2)}.
\end{align}
For the term $\textbf{L}_8(t_1,t_2,n)$, by \eqref{h2-} with $M=1$,
$\sigma_1
 \in \ell^1
 (\mathbb{N},L^{\frac{p}{p-p^*}}(\mathbb{R}^d))$,
$\sigma_2
 \in
 \ell^1
 (\mathbb{N},L^{\frac{q}{q-q^*}}(\mathbb{R}^d))$ and
$\sigma_3
 \in
 \ell^1
 (\mathbb{N},L^1(\mathbb{R}^d))$, it follows that
\begin{align} \label{weak6-}
\textbf{L}^2_8(t_1,t_2,n)
&\leq
 \int_{t_1}^{t_2}\Bigg(
 \sum_{k\in\mathbb{N}}
 \|\sigma_{1,k} \|
 _{\frac p{p-p^*}}
 \|\psi_{u_n}(s)\|^{p^*}_{p}
\nonumber\\
 &  \ \ +
 \sum_{k\in\mathbb{N}}\|\sigma_{2,k} \|
 _{\frac q{q-q^*} }
  \| \psi_{u_n}(s)\|^{q^*}_{q}
  +
 \sum_{k\in\mathbb{N}}\|\sigma_{3,k}\|
 _{1}\big(1+\|\psi^0(s)\|^2\big)\Bigg)ds
\nonumber\\
 &\leq
 \Bigg(\int_{t_1}^{t_2} \bigg(\sum_{k\in\mathbb{N}}
 \|\sigma_{1,k} \|
 _{\frac p{p-p^*}}\bigg)^{\frac{p}{p-p^*}}ds
 \Bigg)^{\frac{p-p^*}{p}}
 \Bigg(\int_{t_1}^{t_2} \|\psi_{u_n}(s)\|^{p}_{p}ds
 \Bigg)^{\frac{p^*}{p}}
\nonumber\\
 &\ +
  \Bigg(\int_{t_1}^{t_2} \bigg(\sum_{k\in\mathbb{N}}
 \|\sigma_{2,k}\|
 _{\frac q{q-q^*}}\bigg)^{\frac{q}{q-q^*}}ds
 \Bigg)^{\frac{q-q^*}{q}}
 \Bigg(\int_{t_1}^{t_2} \|\psi_{u_n}(s)\|^{q}_{q}ds
 \Bigg)^{\frac{q^*}{q}}
\nonumber\\
 &\ +
  \Bigg(\int_{t_1}^{t_2} \bigg(\sum_{k\in\mathbb{N}}
 \|\sigma_{3,k} \|
 _{1}\bigg)^{2}ds
 \Bigg)^{1/2}
 \Bigg(\int_{t_1}^{t_2} \big(1+\|\psi^0(s)\|^2\big)^2ds
 \Bigg)^{1/2}
\nonumber\\
 &\leq
 \|\sigma_1\|_{
 \ell^1
 (\mathbb{N},L^{\frac{p}{p-p^*}}(\mathbb{R}^d))}
\|\psi_{u_n}\|^{p^*}_{L^p([0,T], L^p(\mathbb{R}^d))}(t_2-t_1)^{\frac{p-p^*}{p}}
\nonumber\\
 &\ +
 \|\sigma_2\|_{
 \ell^1
 (\mathbb{N},L^{\frac{q}{q-q^*}}(\mathbb{R}^d))}
\|\psi_{u_n}\|^{q^*}_{L^q([0,T], L^q(\mathbb{R}^d))}(t_2-t_1)^{\frac{q-q^*}{q}}
\nonumber\\
 &\ + \|\sigma_3
 \|_{
 \ell^1
 (\mathbb{N},L^1(\mathbb{R}^d))}
 \Big(1+
 \|\psi^0\|_{L^\infty([0,T],\mathbb{H} )}^2\Big)(t_2-t_1).
\end{align}
Since
$\{\psi_{u_n}\}_{n=1}^\infty$,
$\{u_n\}_{n=1}^\infty$, $\psi^0$,
$\{\mathbf{A}_1(\cdot,\psi_{u_n},
\delta_{\psi^0})\}_{n=1}^\infty$ and  $\{\mathbf{A}_2(\cdot,\psi_{u_n},
\delta_{\psi^0})\}_{n=1}^\infty$ are
bounded in $L^p([0,T], \mathbb{V}_1  ) \bigcap L^q([0,T],
\mathbb{V}_2)$, $L^2([0,T],\ell^2 )$, $L^\infty([0,T],\mathbb{H} )$, $L^{\widehat{p}}([0,T], \mathbb{V}^*_1)$ and $L^{\widehat{q}}([0,T], \mathbb{V}^*_2)$, respectively, by \eqref{weak6} and \eqref{weak6-}, we deduce that for all $0\leq t_1\leq t_2\leq T$,
\begin{align} \label{weak7-}
\|\psi_{u_n}(t_2)-\psi_{u_n}(t_1)
\|_{(\mathbb{H}\bigcap\mathbb{V}_1
\bigcap\mathbb{V}_2)^*}&\leq
 K_1\Big( (t_2-t_1)^{1/2}
+(t_2-t_1)^{1/p} +(t_2-t_1)^{1/q}
\nonumber\\ &\ \ \ \   \ \ \ \   \ \ \ \   +(t_2-t_1)^{\frac{p-p^*}{2p}}
+(t_2-t_1)^{\frac{q-q^*}{2q}}     \Big),
\end{align}
where $K_1>0$ is a constant independent of $n$,
$t_1$ and $t_2$. By \eqref{weak7-} we find that  $\{\psi_{u_n}\}_{n=1}^\infty$ is equicontinous  in $\big(\mathbb{H}\bigcap\mathbb{V}_1
\bigcap\mathbb{V}_2\big)^*$ on $[0,T]$.

\underline{\emph{Step 3}}: \emph{precompactness of $
\{\psi_{u_n}(t)\}_{n=1}^\infty$ in
$\big(\mathbb{H}\bigcap\mathbb{V}_1
\bigcap\mathbb{V}_2\big)^*$ for every $t\in[0,T]$.} This is a consequence of (ii) of Theorem \ref{Con1} and the uniform tail-ends estimates in
Lemma \ref{tightness2ddfhtt}. To see this, we first note that
$\psi_{u_n}(t,x)$ can be written as $\psi_{u_n}(t,x)=\widetilde{\psi}_{u_n}^k
(t,x)+
\widehat{\psi}_{u_n}^k(t,x)$ with
$\widetilde{\psi}_{u_n}^k(t,x):=\rho_k(x)
\psi_{u_n}(t,x)$ and $\widehat{\psi}_{u_n}^k(t,x):=(1-\rho_k(x))
\psi_{u_n}(t,x)$ for $t\in[0,T]$, $x\in\mathbb{R}^d$ and $k\in\mathbb{N}$, where
$\rho_k$ is the cut-off function defined in Lemma \ref{Prpgiikgk.jkj}.
By (ii) of Theorem \ref{Con1}, we find that for all $n,k\in\mathbb{N}$,
\begin{align}\label{weak7}
&\|\widehat{\psi}_{u_n}^k
\|^2_{C([0,T],\mathbb{H} )}\leq\|\psi_{u_n}
\|^2_{C([0,T],\mathbb{H} )}\leq K_2,
\end{align}
where $K_2>0$ is a constant independent of $n$ and $k$. Since $\widehat{\psi}_{u_n}^k(t,x)=0$
for $|x|\geq k$, and the Sobolev embeddings
$W^{\alpha,p}(\mathcal{Q}_k)\hookrightarrow L^2(\mathcal{Q}_k)\hookrightarrow (W^{\alpha,p}(\mathcal{Q}_k))^*$ are compact, we infer from \eqref{weak7}
that $\{\widehat{\psi}_{u_n}^k(t)\}_{n=1}^\infty$ is
precompact in  $\mathbb{V}_1^*$, and so
is
 in $\big(\mathbb{H}\bigcap\mathbb{V}_1
\bigcap\mathbb{V}_2\big)^*$. In addition, since $\{u_n\}_{n=1}^\infty$ is bounded in $L^2([0,T],
\ell^2)$, by Lemma \ref{tightness2ddfhtt}, we find that for every  $\delta>0$, there exists
$N:=N(\delta,T,\|\psi_0\|)
\in\mathbb{N}$ such that for all $k\geq N$, $n\in \mathbb{N}$ and $t\in[0,T]$,
\begin{align*}
& \|\widehat{\psi}_{u_n}^k(t)
\|_{(\mathbb{H}\bigcap\mathbb{V}_1
\bigcap\mathbb{V}_2)^*}^2\leq\|\widehat{\psi}_{u_n}^k(t)
\|^2\leq\int_{ {\mathbb{R}}^d \setminus
\mathcal{Q}_{k/2}}|\psi_{u_n}(t,x)|^2
dx<\delta^2/4.
\end{align*}
Therefore,  for every $t\in[0,T]$, we find that $\big\{\psi_{u_n}(t)=
\widetilde{\psi}_{u_n}^k(t)+
\widehat{\psi}_{u_n}^k(t)\big\}_{n=1}^\infty$ has a finite open cover with radius $\delta$ in
$\big(\mathbb{H}\bigcap\mathbb{V}_1
\bigcap\mathbb{V}_2\big)^*$, and hence it is compact in
$\big(\mathbb{H}\bigcap\mathbb{V}_1
\bigcap\mathbb{V}_2\big)^*$.

\underline{\emph{Step 4}}: \emph{$\psi_{u_n}\rightarrow z$ strongly in $C\big([0,T],\big(\mathbb{H}\bigcap\mathbb{V}_1
\bigcap\mathbb{V}_2\big)^*\big)$}. This is a direct consequence of step $2$, step $3$ and
the Arzel\`{a}-Ascoli Theorem.

\underline{\emph{Step 5}}: \emph{$\psi_{u_n}\rightarrow z$ strongly in $L^{2\kappa}([0,T],\mathbb{H})$
for every  $\kappa\in[1,p\wedge q]$}.
Since $\psi_{u_n}\rightarrow z$ strongly in the space
$C\big([0,T],\big(\mathbb{H}\bigcap\mathbb{V}_1
\bigcap\mathbb{V}_2\big)^*\big)$ and $\{\psi_{u_n}\}_{n=1}^\infty$ is bounded in $C([0,T],\mathbb{H} )\bigcap L^p([0,T], \mathbb{V}_1  ) \bigcap L^q([0,T], \mathbb{V}_2)$, by the relation:
\begin{align*}
 \int_0^T\|\psi_{u_n}(t)-z(t)
\|^{2\kappa}dt&= \int_0^T\  _{(\mathbb{H}\bigcap\mathbb{V}_1
\bigcap\mathbb{V}_2)^*}
\big\langle \psi_{u_n}(t)-z(t),
\psi_{u_n}(t)-z(t)\big  \rangle_{\mathbb{H}\bigcap\mathbb{V}_1
\bigcap\mathbb{V}_2}^\kappa dt
\nonumber\\&\leq
\int_0^T\|\psi_{u_n}(t)-z(t)
\|_{\mathbb{H}\bigcap\mathbb{V}_1
\bigcap\mathbb{V}_2}^\kappa dt
\|\psi_{u_n}-z
\|_{C([0,T],(\mathbb{H}\bigcap\mathbb{V}_1
\bigcap\mathbb{V}_2)^*)}^\kappa,
\end{align*}
we find that $\psi_{u_n}\rightarrow z$ strongly in $L^{2\kappa}([0,T],\mathbb{H})$ for every $\kappa\in[1,p\wedge q]$.

\underline{\emph{Step 6}}: $z(0)=\psi_0$, $z(T)=z_{_{T}}$ and $z=\psi_u$.
Since $\psi_{u_n}$ is the solution of \eqref{Int1--++}, it yields that for all
$\phi\in C_0^\infty((0,T))$
and $\xi\in \mathbb{H}\bigcap\mathbb{V}_1
\bigcap\mathbb{V}_2$,
\begin{align}\label{weak9}
-\int_0^T\langle \psi_{u_n}(t) ,\xi\rangle\phi'(t)dt
&=\sum_{i=1}^2
 \int_{0}^T\ _{\mathbb{V}_i^*}\big\langle \mathbf{A}_i(t,\psi_{u_n}(t),
\delta_{\psi^0(t)}),
\xi\big  \rangle_{\mathbb{V}_i}  \phi(t)      dt
\nonumber\\& +\int_{0}^T
 \big\langle \mathfrak{L}_{H,S}(t,\psi_{u_n}(t),
\delta_{\psi^0(t)})u_n(t)),\xi \big\rangle\phi(t) dt.
\end{align}
We claim that
\begin{align}&\label{weak9a}
\lim_{n\rightarrow\infty}\int_{0}^T
 \big\langle \mathfrak{L}_{H,S}(t,z(t),
\delta_{\psi^0(t)})\big(u_n(t)-u(t)\big),\xi \big\rangle \phi(t) dt=0,
\\&
\label{weak9b}
\lim_{n\rightarrow\infty}\int_{0}^T
 \big\langle \big(\mathfrak{L}_{H,S}(t,\psi_{u_n}(t),
\delta_{\psi^0(t)})-\mathfrak{L}_{H,S}(t,z(t),
\delta_{\psi^0(t)})\big)u_n(t),\xi \big\rangle \phi(t) dt=0,
\end{align}
and
\begin{align}\label{weak9c}
\lim_{n\rightarrow\infty}\int_{0}^T
 \big\langle \mathfrak{L}_{H,S}(t,\psi_{u_n}(t),
\delta_{\psi^0(t)})u_n(t)),\xi \big\rangle \phi(t) dt=\int_{0}^T
 \big\langle \mathfrak{L}_{H,S}(t,z(t),
\delta_{\psi^0(t)})u(t)),\xi \big\rangle\phi(t) dt.
\end{align}
Note that \eqref{weak9a}-\eqref{weak9b} implies
\eqref{weak9c}. We first prove \eqref{weak9a}.
Define an operator
$\mathcal{I}:L^2([0,T],\ell^2 )\rightarrow \mathbb{H} $ by
\begin{align}\label{weak9d}
\mathcal{I}(\widehat{u})=\int_{0}^T
  \mathfrak{L}_{H,S}(s,z(s),
\delta_{\psi^0(s)})\widehat{u}(s)\phi(s) ds,\ \
\forall\ \widehat{u}\in L^2([0,T],\ell^2 ).
\end{align}
Since $z\in L^p([0,T], \mathbb{V}_1  ) \bigcap L^q([0,T],
\mathbb{V}_2)$, as in \eqref{qqqqqq3}, we
find  $\mathfrak{L}_{H,S}(\cdot,z,
\delta_{\psi^0})
\in L^2([0,T],\mathcal{L}_2(\ell^2,
\mathbb{H}))$. Thus, for every
$\widehat{u}\in L^2([0,T],\ell^2 )$,
\begin{align}\label{weak9e}
\|\mathcal{I}(\widehat{u})\|&=\bigg\|\int_{0}^T
  \mathfrak{L}_{H,S}(s,z(s),
\delta_{\psi^0(s)})\widehat{u}(s)\phi(s) ds\bigg\|
\nonumber\\&\leq\int_{0}^T
  \|\mathfrak{L}_{H,S}(s,z(s),
\delta_{\psi^0(s)})\widehat{u}(s)\phi(s)\| ds
\nonumber\\&\leq \int_{0}^T
  \|\mathfrak{L}_{H,S}(s,z(s),
\delta_{\psi^0(s)})\|_{\mathcal{L}_2(\ell^2
 ,\mathbb{H})} \|\widehat{u}(s)\|_{\ell^2}|\phi(s)|ds
\nonumber\\&\leq \max_{s\in[0,T]} |\phi(s)| \|\widehat{u}\|_{L^2([0,T],\ell^2 )}
\|\mathfrak{L}_{H,S}(\cdot,z,
\delta_{\psi^0})\|_{L^2([0,T],\mathcal{L}_2(\ell^2,
\mathbb{H}))}
\nonumber\\& \leq K_3(T,\psi_0) \|\widehat{u}\|_{L^2([0,T],\ell^2 )},
\end{align}
where $K_3(T,\psi_0)>0$ is a constant.
By \eqref{weak9d} and \eqref{weak9e}, we find that the operator $\mathcal{I}:L^2([0,T],\ell^2 )\rightarrow \mathbb{H} $ is linear and bounded, and hence is weakly continuous.
 Since $u_n\rightarrow u$ weakly in $L^2([0,T],
\ell^2)$, we then find that
$\mathcal{I}(u_n)\rightarrow
\mathcal{I}(u)$ weakly in $
\mathbb{H}$. From this weak convergence, the definition of $\mathcal{I}$ in \eqref{weak9d}, and $\xi\in
\mathbb{H}$, we obtain
\eqref{weak9a}.

We  now
 prove \eqref{weak9b}. Since $\psi_{u_n}\rightarrow z$ strongly in $L^2([0,T],\mathbb{H})$, by \eqref{B1--}, we deduce that
\begin{align}\label{apr--}&
\int_{0}^T
 \big\langle \big(\mathfrak{L}_{H,S}(t,\psi_{u_n}(t),
\delta_{\psi^0(t)})-
\mathfrak{L}_{H,S}(t,z(t),
\delta_{\psi^0(t)})\big)u_n(t)
,\xi \big\rangle\phi(t) dt
\nonumber\\&\leq\|\xi\|\max_{t\in[0,T]}|\phi(t)|
\int_{0}^T \big\|\mathfrak{L}_{H,S}(t,\psi_{u_n}(t),
\delta_{\psi^0(t)})-
\mathfrak{L}_{H,S}(t,z(t),
\delta_{\psi^0(t)})\big\|_{\mathcal{L}_2(\ell^2
 ,\mathbb{H})}
\|u_n(t)\|
dt
\nonumber\\&\leq\|\xi\|\max_{t\in[0,T]}|\phi(t)|
\|u_n\|_{L^2([0,T],\ell^2)}\bigg(\int_{0}^T \big\|\mathfrak{L}_{H,S}(t,\psi_{u_n}(t),
\delta_{\psi^0(t)})-
\mathfrak{L}_{H,S}(t,z(t),
\delta_{\psi^0(t)})
\big\|^2_{\mathcal{L}_2(\ell^2
 ,\mathbb{H})}dt\bigg)^{1/2}
\nonumber\\
&\leq
K_4(T)
\bigg(\int_{0}^T
\Big(1+\|\psi_{u_n}(t)\|_p^{\frac{p}{2}}
+\|z(t)\|_p^{\frac{p}{2}}
+\|\psi_{u_n}(t)\|_q^{\frac{q}{2}}
 +\|z(t)\|_q^{\frac{q}{2}} \Big)\|\psi_{u_n}(t)-z(t)\|dt\bigg)^{1/2}
\nonumber\\
&\leq 2 ^{\frac{1}{2}} K_4(T)
\Bigg(\int_{0}^T
\Big(1+\|\psi_{u_n}(t)\|_p^{p}
+\|z(t)\|_p^{p}
 +\|\psi_{u_n}(t)\|_q^{q}
+\|z(t)\|_q^{q}
   \Big)dt\Bigg)^{1/4}  \|\psi_{u_n}-z
   \|^{\frac{1}{2}}_{L^2([0,T],\mathbb{H})}
\nonumber\\&   \leq K_5(T,\psi_0) \|\psi_{u_n}-z
   \|^{\frac{1}{2}}_{L^2([0,T],\mathbb{H})}
   \rightarrow0\ \ \mbox{as}\ \ n\rightarrow\infty,
\end{align}where $K_4(T)>0$ and $K_5(T,\psi_0)>0$ are constants independent of $n$.
Thus, \eqref{weak9b} follows from \eqref{apr--}.

Letting $n\rightarrow\infty$ in \eqref{weak9},
by \eqref{weak1}, \eqref{weak3}-\eqref{weak4} and \eqref{weak9c}, we find that for all
$\phi\in C_0^\infty(0,T)$
and $\xi\in \mathbb{H}\bigcap\mathbb{V}_1
\bigcap\mathbb{V}_2$,
\begin{align*}
-\int_0^T\langle z(t) ,\xi\rangle\phi'(t)dt
&=\sum_{i=1}^2
 \int_{0}^T\ _{\mathbb{V}_i^*}\big\langle \mathfrak{A}_i(t),
\xi\big  \rangle_{\mathbb{V}_i}  \phi(t)      dt
 +\int_{0}^T
 \big\langle \mathfrak{L}_{H,S}(t,z(t),
\delta_{\psi^0(t)})u(t),\xi \big\rangle\phi(t) dt.
\end{align*}
This indicates that for all $\xi\in \mathbb{H}\bigcap\mathbb{V}_1
\bigcap\mathbb{V}_2$, the equality
\begin{align}\label{weak10}
\frac{d}{dt} \langle z(t) ,\xi\rangle
=\sum_{i=1}^2
\ _{\mathbb{V}_i^*}\big\langle \mathfrak{A}_i(t),
\xi\big  \rangle_{\mathbb{V}_i}
 +\big\langle \mathfrak{L}_{H,S}(t,z(t),
\delta_{\psi^0(t)})u(t),\xi \big\rangle
\end{align}
holds in the sense of scalar distribution on $(0,T)$.

Given $\phi\in C^\infty([0,T])$
and $\xi\in \mathbb{H}\bigcap\mathbb{V}_1
\bigcap\mathbb{V}_2$. Since $\psi_{u_n}$ is the solution of \eqref{Int1--++}, we find that
\begin{align}\label{weak11}
&\langle \psi_{u_n}(T) ,\xi\rangle\phi(T)
-\langle \psi_0 ,\xi\rangle\phi(0)-\int_0^T\langle
\psi_{u_n}(t) ,\xi\rangle\phi'(t)dt
\nonumber\\&=\sum_{i=1}^2
 \int_{0}^T\ _{\mathbb{V}_i^*}\big\langle \mathbf{A}_i(t,\psi_{u_n}(t),
\delta_{\psi^0(t)}),
\xi\big  \rangle_{\mathbb{V}_i}  \phi(t)      dt
 +\int_{0}^T
 \big\langle \mathfrak{L}_{H,S}(t,\psi_{u_n}(t),
\delta_{\psi^0(t)})u_n(t),\xi \big\rangle\phi(t) dt.
\end{align}
Letting $n\rightarrow\infty$ in \eqref{weak11},
by \eqref{weak0}-\eqref{weak1}, \eqref{weak3}-\eqref{weak4} and \eqref{weak9c},  it follows that
\begin{align}\label{weak12}&
\langle z_{_{T}} ,\xi\rangle\phi(T)
-\langle \psi_0 ,\xi\rangle\phi(0)-\int_0^T\langle
z(t) ,\xi\rangle\phi'(t)dt
\nonumber\\&=\sum_{i=1}^2
 \int_{0}^T\ _{\mathbb{V}_i^*}\big\langle \mathfrak{A}_i(t),
\xi\big  \rangle_{\mathbb{V}_i}  \phi(t)      dt
 +\int_{0}^T
 \big\langle \mathfrak{L}_{H,S}(t,z(t),
\delta_{\psi^0(t)})u(t),\xi \big\rangle\phi(t) dt.
\end{align}
In addition, by \eqref{weak10}, we obtain that \begin{align}\label{weak13}&
\langle z(T) ,\xi\rangle\phi(T)
-\langle z(0) ,\xi\rangle\phi(0)
-\int_0^T\langle
z(t) ,\xi\rangle\phi'(t)dt
\nonumber\\&=\sum_{i=1}^2
 \int_{0}^T\ _{\mathbb{V}_i^*}\big\langle \mathfrak{A}_i(t),
\xi\big  \rangle_{\mathbb{V}_i}  \phi(t)      dt
 +\int_{0}^T
 \big\langle \mathfrak{L}_{H,S}(t,z(t),
\delta_{\psi^0(t)})u(t)),\xi \big\rangle\phi(t) dt.
\end{align}
As a result of \eqref{weak12}-\eqref{weak13}, we get
\begin{align}\label{weak14}&
\langle z_{_{T}} ,\xi\rangle\phi(T)
-\langle \psi_0 ,\xi\rangle\phi(0)=\langle z(T) ,\xi\rangle\phi(T)
-\langle z(0) ,\xi\rangle\phi(0).
\end{align}
By choosing $\phi\in C^\infty([0,T])$ with  $\phi(0)=1$ and $\phi(T)=0$, we find from \eqref{weak14} that $\langle \psi_0 ,\xi\rangle=\langle z(0) ,\xi\rangle$, and hence  $z(0)=\psi_0$.
Similarly,
By choosing $\phi\in C^\infty([0,T])$ with  $\phi(T)=1$ and $\phi(0)=0$,
we get  $z(T)=z_{_{T}}$. This, along with \eqref{weak0}, implies that
\begin{align}\label{weak15}
&\|z(T)\|=\|z_{_{T}}\|\leq
\liminf_{n\rightarrow\infty}
\|\psi_{u_n}(T)\|.
\end{align}

\underline{\emph{Step 7}}: \emph{$z$
is a solution of \eqref{Int1--++}}.
Since $\psi_{u_n}$ is the solution of \eqref{Int1--++}, we get the energy equation:
\begin{align}\label{weak16}
\|\psi_{u_n}(T)\|^2=\|\psi_0\|^2 &+2\sum_{i=1}^2\int_0^T
\ _{\mathbb{V}_i^*}\big\langle \mathbf{A}_i(t,\psi_{u_n}(t),
\delta_{\psi^0(t)}),
\psi_{u_n}(t)\big  \rangle_{\mathbb{V}_i}       dt
\nonumber\\& + 2 \int_0^T\big\langle\mathfrak{L}_{H,S}
 (t,\psi_{u_n}(t),
\delta_{\psi^0(t)})u_n(t) , \psi_{u_n}(t)\big\rangle dt.
\end{align}
In order to deal with the last term in \eqref{weak16}, we first prove:
\begin{align}\label{weak16a}
&\lim_{n\rightarrow\infty} \int_0^T\big\langle\mathfrak{L}_{H,S}
 (t,\psi_{u_n}(t),
\delta_{\psi^0(t)})u_n(t) , \psi_{u_n}(t)-z(t)\big\rangle dt=0,
\\&
\label{weak16b}
\lim_{n\rightarrow\infty}\int_{0}^T
 \big\langle \mathfrak{L}_{H,S}(t,z(t),
\delta_{\psi^0(t)})(u_n(t)-u(t)), z(t) \big\rangle dt=0,
\\\label{weak16c}
&\lim_{n\rightarrow\infty} \int_0^T\big\langle\big(\mathfrak{L}_{H,S}
 (t,\psi_{u_n}(t),
\delta_{\psi^0(t)})-\mathfrak{L}_{H,S}
 (t,z(t),
\delta_{\psi^0(t)})\big)u_n(t) , z(t)\big\rangle dt=0,
\end{align}
and
\begin{align}\label{weak16d}
&\lim_{n\rightarrow\infty} \int_0^T\big\langle\mathfrak{L}_{H,S}
 (t,\psi_{u_n}(t),
\delta_{\psi^0(t)})u_n(t) , \psi_{u_n}(t)\big\rangle dt= \int_0^T\big\langle\mathfrak{L}_{H,S}
 (t,z(t),
\delta_{\psi^0(t)})u(t) , z(t)\big\rangle dt.
\end{align}
Note that \eqref{weak16a}-\eqref{weak16c} implies \eqref{weak16d}. Since $z\in
L^\infty([0,T],\mathbb{H})$, following  the proof of \eqref{weak9b}, one can
obtain \eqref{weak16c}.
Then we only  need to prove \eqref{weak16a}-\eqref{weak16b}.

Note that $\psi_{u_n}\rightarrow z$ strongly in $L^{4}([0,T],\mathbb{H})$ by step 5. This,
together with \eqref{B1-}, implies that
\begin{align}\label{weak16e}
& \int_0^T\big\langle\mathfrak{L}_{H,S}
 (t,\psi_{u_n}(t),
\delta_{\psi^0(t)})u_n(t) , \psi_{u_n}(t)-z(t)\big\rangle dt
\nonumber\\
&\leq\int_0^T
\|\mathfrak{L}_{H,S}
 (t,\psi_{u_n}(t),
\delta_{\psi^0(t)})\|_{\mathcal{L}_2(\ell^2
,\mathbb{H})}\|u_n(t)\|_{\ell^2}
\| \psi_{u_n}(t)-z(t)\| dt
\nonumber\\
&\leq\|u_n
   \|_{L^2([0,T],\ell^2 )}
\bigg(\int_0^T
\|\mathfrak{L}_{H,S}
 (t,\psi_{u_n}(t),
\delta_{\psi^0(t)})\|^2_{\mathcal{L}_2(\ell^2
,\mathbb{H})}\|\psi_{u_n}(t)-z(t)\|^2 dt\bigg)^{1/2}
\nonumber\\
&\leq\sqrt{C_6(1)}\|u_n
   \|_{L^2([0,T],\ell^2 )}
\bigg(\int_0^T
\Big(\big(1+\|\psi_{u_n}(t)\|^{p/2}_{p}+
 \| \psi_{u_n}(t)\|^{q/2}_{q}\big)\|\psi_{u_n}(t)\|
\nonumber\\
& \ \ \ \ \ \ \ \ \ \ \ \ \    \   \ \ \ \ \ \ \ \ \ \ \ \ \ \ \   \   \ \ \ \ \ \ \ \ \ \ \ \ \ \ \  +\big(1+\|\psi^0(t)\|^2\big)\Big)
\|\psi_{u_n}(t)-z(t)\|^2  dt\bigg)^{1/2}
\nonumber\\
&\leq\sqrt{C_6(1)}\|u_n
   \|_{L^2([0,T],\ell^2 )}\big(1+\|\psi^0
   \|_{L^\infty([0,T],\mathbb{H})}+\|\psi_{u_n}
   \|^{1/2}_{L^\infty([0,T],\mathbb{H})}\big)
\nonumber\\
&\ \times\bigg(\int_0^T
\big(2+\|\psi_{u_n}(t)\|^{p/2}_{p}+
 \| \psi_{u_n}(t)\|^{q/2}_{q}\big)
\|\psi_{u_n}(t)-z(t)\|^2  dt\bigg)^{1/2}
\nonumber\\
&\leq\sqrt{\sqrt{3}C_6(1)}\|u_n
   \|_{L^2([0,T],\ell^2 )}\big(1+\|\psi^0
   \|_{L^\infty([0,T],\mathbb{H})}+\|\psi_{u_n}
   \|^{1/2}_{L^\infty([0,T],\mathbb{H})}\big)
\nonumber\\
&\ \times\|\psi_{u_n}
-z\|_{L^4([0,T],\mathbb{H})}\bigg(\int_0^T
\big(4+\|\psi_{u_n}(t)\|^{p}_{p}+
 \| \psi_{u_n}(t)\|^{q}_{q}\big)
 dt\bigg)^{1/4}
\nonumber\\
&\leq K_6(T,\psi_0)\|\psi_{u_n}
-z\|_{L^4([0,T],\mathbb{H})}\ \ \mbox{as}\ \ n\rightarrow\infty,
\end{align}where $K_6(T,\psi_0)>0$ is a constant independent of $n$. Therefore, \eqref{weak16a} follows from \eqref{weak16e}.

To  prove \eqref{weak16b}, we define an operator
$\mathcal{J}:L^2([0,T],\ell^2 )\rightarrow \mathbb{R}$ by
\begin{align}\label{weak9d-}
\mathcal{J}(\widehat{u})=\int_{0}^T\big\langle
  \mathfrak{L}_{H,S}(s,z(s),
\delta_{\psi^0(s)})\widehat{u}(s),z(s)\big\rangle ds,\ \
\forall\ \widehat{u}\in L^2([0,T],\ell^2 ).
\end{align}
As in \eqref{weak9e}, we infer that  for every
$\widehat{u}\in L^2([0,T],\ell^2 )$,
\begin{align}\label{weak9e-}
|\mathcal{J}(\widehat{u})|&=\Big|
\int_{0}^T\big\langle
  \mathfrak{L}_{H,S}(s,z(s),
\delta_{\psi^0(s)})\widehat{u}(s),z(s)\big\rangle ds\Big|
\nonumber\\&\leq \int_{0}^T
  \|\mathfrak{L}_{H,S}(s,z(s),
\delta_{\psi^0(s)})\|_{\mathcal{L}_2(\ell^2
 ,\mathbb{H})} \|\widehat{u}(s)\|_{\ell^2}\|z(s)\|ds
\nonumber\\&\leq \|z\|_{L^\infty([0,T],\mathbb{H})} \|\widehat{u}\|_{L^2([0,T],\ell^2 )}
\|\mathfrak{L}_{H,S}(\cdot,z,
\delta_{\psi^0})\|_{L^2([0,T],\mathcal{L}_2(\ell^2,
\mathbb{H}))}
\nonumber\\& \leq K_7(T,\psi_0)\|\widehat{u}\|_{L^2([0,T],\ell^2 )},
\end{align}
where $K_7(T,\psi_0)>0$ is a constant.
 From \eqref{weak9d-} and \eqref{weak9e-}, one can find that the operator $\mathcal{J}:L^2([0,T],\ell^2 )\rightarrow \mathbb{R} $ is a  bounded linear
functional on $L^2([0,T],\ell^2 )$. Since $u_n\rightarrow u$ weakly in $L^2([0,T],
\ell^2)$, we find that $\mathcal{J}(u_n)\rightarrow
\mathcal{J}(u)$ in $\mathbb{R}$. This implies
\eqref{weak16b}.

Taking the inferior limit in \eqref{weak16} as $n\rightarrow\infty$, we infer from \eqref{weak15} and \eqref{weak16d} that
\begin{align}\label{weak17}
\|z(T)\|^2\leq
\liminf_{n\rightarrow\infty}
\|\psi_{u_n}(T)\|^2=\|\psi_0\|^2 &+2\liminf_{n\rightarrow\infty}\sum_{i=1}^2\int_0^T
\ _{\mathbb{V}_i^*}\big\langle \mathbf{A}_i(t,\psi_{u_n}(t),
\delta_{\psi^0(t)}),
\psi_{u_n}(t)\big  \rangle_{\mathbb{V}_i}       dt
\nonumber\\& + 2 \int_0^T\big\langle\mathfrak{L}_{H,S}
 (t,z(t),
\delta_{\psi^0(t)})u(t) , z(t)\big\rangle dt.
\end{align}
Moreover, by \eqref{weak10}, \eqref{weak1}-\eqref{weak2a} and $z(0)=\psi_0$, it follows that
\begin{align}\label{weak18}
 \|z(T)\|^2=\|\psi_0\|^2
+2\sum_{i=1}^2\int_0^T
\ _{\mathbb{V}_i^*}\big\langle \mathfrak{A}_i(t),
z(t)\big  \rangle_{\mathbb{V}_i}dt
 +2\int_0^T\big\langle\mathfrak{L}_{H,S}
 (t,z(t),
\delta_{\psi^0(t)})u(t) , z(t)\big\rangle dt.
\end{align}
As a result of \eqref{weak17}-\eqref{weak18}, we obtain the inequality:
\begin{align}\label{weak19}
\sum_{i=1}^2\int_0^T
\ _{\mathbb{V}_i^*}\big\langle \mathfrak{A}_i(t),
z(t)\big  \rangle_{\mathbb{V}_i}dt\leq
\liminf_{n\rightarrow\infty}\sum_{i=1}^2\int_0^T
\ _{\mathbb{V}_i^*}\big\langle \mathbf{A}_i(t,\psi_{u_n}(t),
\delta_{\psi^0(t)}),
\psi_{u_n}(t)\big  \rangle_{\mathbb{V}_i}       dt.
\end{align}

Next, we use the pseudo monotone technique (see e.g., \cite{Zeidler}) to show that
\begin{align}&\label{weak23a}
\mathfrak{A}_1+\mathfrak{A}_2=
\mathbf{A}_1(\cdot,z,
\delta_{\psi^0})+\mathbf{A}_2(\cdot,z,
\delta_{\psi^0}).
\end{align}
By the monotonicity of $\mathbf{A}_1$ and $\mathbf{A}_2$ in Lemma \ref{Monotonicity}, we find that for all $v\in
 L^\infty ([0,T], \mathbb{H} )\bigcap
 L^p([0,T], \mathbb{V}_1  )\bigcap L^q([0,T], \mathbb{V}_2  )$,
\begin{align}\label{weak20}
&
\int_0^T
(\| \phi_6(t)\|_\infty +
\| \widehat{\phi} _6(t)\|_\infty
) \| \psi_{u_n}(t)-v(t)\|^2 dt
\nonumber\\
&
\geq\sum_{i=1}^2\int_0^T\ _{\mathbb{V}_i^*}\big\langle \mathbf{A}_i(t,\psi_{u_n}(t),\delta_{\psi^0(t)})
-\mathbf{A}_i(t,v(t),\delta_{\psi^0(t)})
,
\psi_{u_n}(t)-v(t)\big  \rangle_{\mathbb{V}_i}
\nonumber\\&=
\sum_{i=1}^2\int_0^T\bigg(
\ _{\mathbb{V}_i^*}\big\langle \mathbf{A}_i(t,\psi_{u_n}(t),\delta_{\psi^0(t)}),
\psi_{u_n}(t)\big  \rangle_{\mathbb{V}_i} -
\ _{\mathbb{V}_i^*}\big\langle \mathbf{A}_i(t,v(t),
\delta_{\psi^0(t)}),
\psi_{u_n}(t)\big  \rangle_{\mathbb{V}_i}
\nonumber\\&\ - \ _{\mathbb{V}_i^*}\big\langle
\mathbf{A}_i(t,\psi_{u_n}(t),
\delta_{\psi^0(t)}),
v(t)\big  \rangle_{\mathbb{V}_i}
+\ _{\mathbb{V}_i^*}\big\langle
\mathbf{A}_i(t,v(t),
\delta_{\psi^0(t)}),
v(t)\big  \rangle_{\mathbb{V}_i}\bigg)dt,
\end{align}
where
$\mathbf{A}_1(\cdot,v,\delta_{\psi^0})
\in L^{\widehat{p}}([0,T], \mathbb{V}^*_1  )$ and
$\mathbf{A}_2(\cdot,v,\delta_{\psi^0})
\in L^{\widehat{q}}([0,T], \mathbb{V}^*_2  )$ by Lemma \ref{Dissipativeness}.

Letting $n\rightarrow\infty$ in \eqref{weak20}, by  the convergence in Step 5,
the boundednes of $\{\psi_{u_n}\}_{n=1}^\infty$
in $C([0,T], \mathbb{H})$,
\eqref{weak19} and \eqref{weak2}-\eqref{weak4}, we obtain that
\begin{align}
 \label{weak21}
&
\int_0^T
(\| \phi_6(t)\|_\infty +
\| \widehat{\phi} _6(t)\|_\infty
) \| z(t)-v(t)\|^2 dt
\nonumber\\
&
\geq
\sum_{i=1}^2\int_0^T
\ _{\mathbb{V}_i^*}\big\langle \mathfrak{A}_i(t)-\mathbf{A}_i(t,v(t),
\delta_{\psi^0(t)}),
z(t)-v(t)\big  \rangle_{\mathbb{V}_i}dt.
\end{align}
Letting $v:=z-k^{-1}\phi (t) \xi $ with $k\in \mathbb{N}$,
$\phi \in C^\infty_0 ([0,T]) $
and $\xi \in  \mathbb{H}
\bigcap
 \mathbb{V}_1   \bigcap   \mathbb{V}_2   $ in \eqref{weak21}, multiplying   both sides by $k$, and letting $k\rightarrow\infty$, by the Lebesgue dominated convergence theorem, and
the hemicontinuity of $\mathbf{A}_1$ and $\mathbf{A}_2$ in
Corollary
 \ref{cHemicontinuity}, one obtains that
\begin{align}&\label{weak22}
0\geq
\sum_{i=1}^2\int_0^T
\ _{\mathbb{V}_i^*}\big\langle \mathfrak{A}_i(t)-\mathbf{A}_i(t,z(t),
\delta_{\psi^0(t)}),
\phi (t) \xi \big  \rangle_{\mathbb{V}_i}dt.
\end{align}
Replacing $\xi $ by $-\xi$  in \eqref{weak22}, one arrives at
\begin{align}&\label{weak23}
\sum_{i=1}^2\int_0^T
\ _{\mathbb{V}_i^*}\big\langle \mathfrak{A}_i(t)-\mathbf{A}_i(t,z(t),
\delta_{\psi^0(t)}),
\phi (t) \xi \big  \rangle_{\mathbb{V}_i}dt=0.
\end{align}
By the arbitrariness of $\phi$ and $\xi$,
  in
\eqref{weak23}, we obtin  \eqref{weak23a}.
This, together with \eqref{weak10},
implies that $z$
is a solution of \eqref{Int1--++}.

\underline{\emph{Step 8}}:\emph{ \emph{$\psi_{u_n}\rightarrow \psi_{u}$ strongly in $L^{2\kappa}([0,T],\mathbb{H})$ with $\kappa\in[1,p\wedge q]$}}. This is a direct consequence
Step 5, Step 7 and the uniqueness of
solutions of \eqref{Int1--++}.

\underline{\emph{Step 9}}: \emph{proof of  \eqref{weak000}}. Since  $\psi_{u_n}$ and $\psi_u$
are two solutions of \eqref{Int1--++},   by
 Lemma \ref{Monotonicity} and the argument of
  \eqref{weak16e}, we deduce that for all $t\in[0,T]$,
\begin{align}\label{weak24}&\|\psi_{u_n}(t)-
\psi_u(t)
\|^2+\Big(\frac{1}{2^pK}\wedge\frac{\beta}{2^{p-1}}
\Big)\int_0^t\|\psi_{u_n}(r)-
\psi_u(r)\|^p_{\mathbb{V}_1}dr
+2^{1-q}\widehat{\beta}
\int_0^t\|\psi_{u_n}(r)-
\psi_u(r)\|_{\mathbb{V}_2}^{q}dr
\nonumber\\&\leq
2\int_0^T\langle
    \mathfrak{L}_{H,S}(r,\psi_{u_n}(r),
\delta_{\psi^0(r)})u_n(r) - \mathfrak{L}_{H,S}(r,\psi_{u}(r),
\delta_{\psi^0(r)})u(r) ,\psi_{u_n}(r)-
\psi_u(r)
\rangle dr
\nonumber\\&+
2 \int_0^T
\big(\|\phi_6 (t)\| _{L^\infty(\mathbb{R}^d)}+
\|\widehat{\phi}_6 (t) \|_{ L^\infty(\mathbb{R}^d)}\big)
\|\psi_{u_n} (t) -
\psi_u  (t)
\|^2 dt
\nonumber\\&\leq K_8(T,\psi_0)
\|\psi_{u_n}
-z\|_{L^4([0,T],\mathbb{H})}
+
2 \int_0^T
\big(\|\phi_6 (t)\| _{L^\infty(\mathbb{R}^d)}+
\|\widehat{\phi}_6 (t) \|_{ L^\infty(\mathbb{R}^d)}\big)
\|\psi_{u_n} (t) -
\psi_u  (t)
\|^2 dt  ,
\end{align}
where $K_8(T,\psi_0)>0$ is a constant independent of $n$.
Since
$\psi_{u_n}\rightarrow \psi_{u}$ strongly in $L^{4}([0,T],\mathbb{H})$, By the Lebesgue
dominated convergence theorem, we infer that
the right-hand side of \eqref{weak24} converges
to zero as $n\to \infty$,
which implies
  \eqref{weak000},
  and thus completes the proof.
\end{proof}

\subsection{Uniform convergence in probability of solutions of stochastic controlled equations}
Let  $\psi^\epsilon(\cdot,\psi_0)$ be the solution of
the  stochastic
equation  \eqref{Int1--}
on $[0,T]$ with initial value $\psi_0\in \mathbb{H}$.
As in \cite[Theorem 3.6]{LR4} and
\cite[Lemma 4.1]{Chenzhang4b}
it follows from Theorem \ref{Theorem4},
 that  for every $\epsilon>0$ and $\psi_0\in \mathbb{H}$, there exists a Borel measurable map \begin{align}&\label{WW2}
\mathfrak{M}_{\psi_0,\mathcal{L}_{\psi^\epsilon}}^\epsilon: C([0,T],\mathcal{U})\rightarrow C([0,T],\mathbb{H} )\bigcap L^p([0,T], \mathbb{V}_1  )\bigcap L^q([0,T], \mathbb{V}_2)
\end{align}
such that $\psi^\epsilon(\cdot,\psi_0)$ can be written as
\begin{align}\label{WW3}
\psi^\epsilon(\cdot,\psi_0)=
\mathfrak{M}_{\psi_0,\mathcal{L}_{\psi^\epsilon}}^\epsilon(\mathcal{W})
, \ \ \mbox{$\mathbf{P}$-a.s.}.
\end{align}
In addition, for every control $u\in L^2([0,T],\ell^2)$,
 $\psi^\epsilon_u=\mathfrak{M}_{\psi_0,\mathcal{L}_{\psi^\epsilon}}^\epsilon
\Big(\mathcal{W}+\frac{1}{\sqrt{\epsilon}}
\int_0^\cdot u(t)dt\Big)$ is the unique solution of the stochastic   equation:
\begin{equation}\label{Int1---}
\left\{
  \begin{aligned}
  &d \psi_u^\epsilon(t)+ \mathfrak{L}_{\mathcal{K}_{p}^\alpha} \psi_u^\epsilon(t) dt
=\sum_{i=1}^2\textbf{A}_i(t,\psi_u^\epsilon(t),
\mathcal{L}_{
\psi^\epsilon(t)})dt\\&+
\mathfrak{L}_{H,S}
(t,\psi_u^\epsilon(t),
\mathcal{L}_{
\psi^\epsilon(t)})u(t)dt+ \sqrt{\epsilon}\mathfrak{L}_{H,S}
(t,\psi_u^\epsilon(t),
\mathcal{L}_{
\psi^\epsilon(t)})d\mathcal{W}(t), \ \ t>0, \\
  &\psi_u^\epsilon(0)=\psi_0\in \mathbb{H}.\\
  \end{aligned}
\right.
\end{equation}

We also consider an operator
\begin{align}&\label{WW4}
\mathfrak{M}_{\psi_0}^0: C([0,T],\mathcal{U})\rightarrow C([0,T],\mathbb{H} )\bigcap L^p([0,T], \mathbb{V}_1  )\bigcap L^q([0,T], \mathbb{V}_2),
\end{align}
which is defined by, for $\upsilon\in C([0,T],\mathcal{U})$,
\begin{align}\label{WW5}
\mathfrak{M}_{\psi_0}^0(\upsilon)=\left\{
          \begin{array}{ll}
            \psi_u (\cdot, \psi_0), & \mbox{if}\ \upsilon=\int_0^\cdot u(t)dt\ \mbox{for some $u\in L^2([0,T],\ell^2)$}\ ;  \\
            0, & \mbox{otherwise},
          \end{array}
        \right.
        \end{align}
where $\psi_u(\cdot,\psi_0)$ is the solution of problem \eqref{Int1--++}
on $[0,T]$ with initial value $\psi_0\in \mathbb{H}$ and control $u\in L^2([0,T],\ell^2)$.

We first derive uniform estimates of solutions of \eqref{Int1---} in  $L^2(\Omega, C([0,T],\mathbb{H}))\bigcap L^p([0,T]\times\Omega, \mathbb{V}_1)\bigcap L^q([0,T]\times\Omega, \mathbb{V}_2)$.

\begin{lemma}\label{SCE1}
Let conditions  {\bf A} and {\bf B1}
be satisfied. Then for every $T>0$, $R>0$ and $N>0$, there exists $C_3(T,R,N)>0$ such that for all $\psi_{0}\in \overline{B}_{R}(\mathbb{H})$, $u\in \mathfrak{A}_N$ and $\epsilon\in(0,1)$,
the solution $\psi^\epsilon_u=\mathfrak{M}_{\psi_0,\mathcal{L}_{\psi^\epsilon}}^\epsilon
\Big(\mathcal{W}+\frac{1}{\sqrt{\epsilon}}
\int_0^\cdot u(t)dt\Big)$ of \eqref{Int1---} satisfies
\begin{align*}
&\mathbf{E}\Big[\|\psi^\epsilon_u
\|^2_{C([0,T],\mathbb{H} )}\Big]+\mathbf{E}\Big[
\|\psi^\epsilon_u\|^p_{L^p([0,T],\mathbb{V}_1 )}\Big]+\mathbf{E}\Big[\|\psi^\epsilon_u\|^q_{L^q([0,T], \mathbb{V}_2)}\Big]\leq C_3(T,R, N),
\end{align*}
\end{lemma}
\begin{proof}
The proof is similar to that of Theorem \ref{Con1}  and  Lemma  \ref{ffTheorem4kk},
and hence omitted here.
\end{proof}

We then show that, as $\epsilon\rightarrow0$, the solution  $\psi^\epsilon$ of \eqref{Int1--} converges to the solution $\psi^0$ of \eqref{Int1--++++} in $L^2(\Omega, C([0,T],\mathbb{H}) )\bigcap L^p([0,T]\times\Omega, \mathbb{V}_1)\bigcap L^q([0,T]\times\Omega, \mathbb{V}_2)$.
\begin{lemma}\label{SCE2}
Let conditions  {\bf A} and {\bf B1}
be satisfied. Then for every $T>0$ and $R>0$,
the solution  $\psi^\epsilon$ of \eqref{Int1--} and the solution $\psi^0$
of \eqref{Int1--++++}  satisfy that, for all $\psi_{0}\in \overline{B}_{R}(\mathbb{H})$ and $\epsilon\in(0,1)$,
\begin{align*}
&\mathbf{E}\Big[\|\psi^\epsilon-\psi^0
\|^2_{C([0,T],\mathbb{H} )}\Big]
+\mathbf{E}\Big[
\|\psi^\epsilon-\psi^0\|^p_{L^p([0,T],\mathbb{V}_1 )}\Big]+\mathbf{E}\Big[\|\psi^\epsilon-\psi^0\|^q_{L^q([0,T], \mathbb{V}_2)}\Big]
\leq \epsilon C_4(T,R),
\end{align*}
where $C_4(T,R)>0$ is a constant
independent of $\epsilon$.
\end{lemma}
\begin{proof}
The proof is is standard,  and we   do not
repeat  the details  here.
\end{proof}

By Lemmas \ref{SCE1} and \ref{SCE2}, we finally show that, as $\epsilon\rightarrow0$, the solution $\psi^\epsilon_u$ of \eqref{Int1---} converges in probability to the solution $\psi_u$ of problem \eqref{Int1--++} in $ C([0,T],\mathbb{H}) \bigcap L^p([0,T], \mathbb{V}_1)\bigcap L^q([0,T], \mathbb{V}_2)$.
\begin{lemma}\label{SCE3}
Let conditions  {\bf A} and {\bf B1}
be satisfied. Then for every $T>0$, $\delta>0$, $R>0$ and $N>0$, we have
\begin{align*}
&\lim_{\epsilon\rightarrow0}\sup_{\psi_{0}\in \overline{B}_{R}(\mathbb{H})}
\sup_{u\in \mathfrak{A}_N}\mathbf{P}\Big\{ \|\psi^\epsilon_u-\psi_u\|_{C([0,T],\mathbb{H}) \bigcap L^p([0,T], \mathbb{V}_1)\bigcap L^q([0,T], \mathbb{V}_2)}>\delta\Big\}=0,
\end{align*}
where $\psi^\epsilon_u=
\mathfrak{M}_{\psi_0,\mathcal{L}_{
\psi^\epsilon}}^\epsilon
\big(\mathcal{W}+\frac{1}{\sqrt{\epsilon}}
\int_0^\cdot u(t)dt\big)$ and $\psi_u=
\mathfrak{M}_{\psi_0}^0
\big(
\int_0^\cdot u(t)dt\big)$ are solutions of \eqref{Int1---} and \eqref{Int1--++}, respectively.
\end{lemma}
\begin{proof}
Since $\psi^\epsilon_u=
\mathfrak{M}_{\psi_0,\mathcal{L}_{
\psi^\epsilon}}^\epsilon
\big(\mathcal{W}+\frac{1}{\sqrt{\epsilon}}
\int_0^\cdot u(t)dt\big)$ and $\psi_u=
\mathfrak{M}_{\psi_0}^0
\big(
\int_0^\cdot u(t)dt\big)$ are solutions of \eqref{Int1---} and \eqref{Int1--++}, respectively,
it follows from  It\^{o}'s
formula and
 Lemma \ref{Monotonicity} that for all $t\in[0,T]$, $\psi_{0}\in \overline{B}_{R}(\mathbb{H})$,
$u\in \mathfrak{A}_N$ and $\epsilon\in(0,1)$,
\begin{align}\label{SCE5}
&\|\psi^\epsilon_u(t)
-\psi_u(t)
\|^2+2^{-p}K^{-1}\int_0^{t}
\|\psi^\epsilon_u(r)
-\psi_u(r)\|^p_{\dot{\mathbb{V}}_1}dr
\nonumber\\&\ +2^{1-p}\beta
\int_0^{t}
\|\psi^\epsilon_u(r)
-\psi_u(r)\|_{p}^{p}dr
+2^{1-q}\widehat{\beta}
\int_0^{t}\|\psi^\epsilon_u(r)
-\psi_u(r)\|_{\mathbb{V}_2}^{q}dr
\nonumber\\&\ +\frac{\beta}{2}\int_0^{t} \int_{\mathbb{R}^d}  \big(|\psi^\epsilon_u(r,x)|^{p-2}
+|\psi_u(r,x)|^{p-2}\big)
|\psi^\epsilon_u(r,x)-\psi_u(r,x)|^2dxdr
\nonumber\\&\ +\frac{\widehat{\beta}}{2}
\int_0^{t} \int_{\mathbb{R}^d}  \big(|\psi^\epsilon_u(r,x)|^{q-2}
+|\psi_u(r,x)|^{q-2}\big)
|\psi^\epsilon_u(r,x)-\psi_u(r,x)|^2dxdr
\nonumber\\&\leq
\underbrace{2
\int_0^{t}\big\langle
    \big(\mathfrak{L}_{H,S}(r,\psi^\epsilon_u(r),
\mathcal{L}_{
\psi^\epsilon(t)}) - \mathfrak{L}_{H,S}(r,\psi_u(r),
\delta_{\psi^0(r)})\big)u(r) ,\psi^\epsilon_u(r)
-\psi_u(r)
\big\rangle dr}_{\textbf{L}^\epsilon_9(t)}
\nonumber\\&\ +\underbrace{\epsilon\int_0^{t}
\|\mathfrak{L}_{H,S}(r,\psi^\epsilon_u(r),
\mathcal{L}_{
\psi^\epsilon(r)})
		\|^2_{\mathcal{L}_2(\ell^2,
			\mathbb{H})}dr}_{\textbf{L}^\epsilon_{10}(t)}
\nonumber\\&\ +\underbrace{2\sqrt{\epsilon}
\int_0^{t}\langle
    \psi^\epsilon_u(r)
-\psi_u(r),\mathfrak{L}_{H,S}(r,
\psi^\epsilon_u(r),
\mathcal{L}_{
\psi^\epsilon(r)})d\mathcal{W}(r)
\rangle}_{\textbf{L}^\epsilon_{11}(t)}
\nonumber\\&\ +
2\int_0^{t}
\big(\|\phi_6(r)\|_{\infty}+
\|\widehat{\phi}_6(r)\|_{\infty}\big)
\| \psi^\epsilon_u(r)
-\psi_u(r)
\|^2dr
\nonumber\\&\ +2\int_0^{t}
\big(\|\phi_7(r)\|_{1}+
\|\widehat{\phi}_7(r)\|_{1}\big)
d_{\mathcal{P}_2}^2
(\mathcal{L}_{
\psi^\epsilon(r)},\delta_{\psi^0(r)})dr.
\end{align}
For   $\textbf{L}^\epsilon_{9}(t)$ in \eqref{SCE5}, by \eqref{h2--} (for $M=1$, $\varepsilon_1=\beta/2$ and
$\varepsilon_2=\widehat{\beta}/2$) and \eqref{h2----0},
we deduce that for all $t\in[0,T]$, $\psi_{0}\in \overline{B}_{R}(\mathbb{H})$,
$u\in \mathfrak{A}_N$ and $\epsilon\in(0,1)$,
\begin{align}
\label{SCE6}
\textbf{L}^\epsilon_{9}(t)
&\leq
\int_0^{t}
\|u(r)\|^2\|\psi^\epsilon_u(r)
-\psi_u(r)
\|^2dr\nonumber\\&\ +
\int_0^{t}
    \big\|\mathfrak{L}_{H,S}(r,\psi^\epsilon_u(r),
\mathcal{L}_{
\psi^\epsilon(t)}) - \mathfrak{L}_{H,S}(r,\psi_u(r),
\delta_{\psi^0(r)})
\big\|^2 dr
\nonumber\\&\leq K_1
\int_0^{t}
\big(1
+\|u(r)\|^2\big)\|\psi^\epsilon_u(r)
-\psi_u(r)
\|^2dr
+ K_1\int_0^{t}
d_{\mathcal{P}_2}^2
(\mathcal{L}_{
\psi^\epsilon(r)},\delta_{\psi^0(r)})dr
\nonumber\\&\ +\frac{\beta}{2}\int_0^{t} \int_{\mathbb{R}^d}  \big(|\psi^\epsilon_u(r,x)|^{p-2}
+|\psi_u(r,x)|^{p-2}\big)
|\psi^\epsilon_u(r,x)-\psi_u(r,x)|^2dxdr
\nonumber\\&\ +\frac{\widehat{\beta}}{2}
\int_0^{t} \int_{\mathbb{R}^d}  \big(|\psi^\epsilon_u(r,x)|^{q-2}
+|\psi_u(r,x)|^{q-2}\big)
|\psi^\epsilon_u(r,x)-\psi_u(r,x)|^2dxdr,
\end{align}
where $K_1>0$ is a constant independent of $\epsilon$.

By \eqref{SCE5} and \eqref{SCE6}, we have
\begin{align}\label{SCE5-1}
&\sup_{r\in [0,t]}\|\psi^\epsilon_u(r)-\psi_u(r)\|^2
+2^{-p}K^{-1}\int_0^{t}
\|\psi^\epsilon_u(r)
-\psi_u(r)\|^p_{\dot{\mathbb{V}}_1}dr
\nonumber\\&\ +2^{1-p}\beta
\int_0^{t}
\|\psi^\epsilon_u(r)
-\psi_u(r)\|_{p}^{p}dr
+2^{1-q}\widehat{\beta}
\int_0^{t}\|\psi^\epsilon_u(r)
-\psi_u(r)\|_{\mathbb{V}_2}^{q}dr \nonumber\\
&\leq
2\int_0^{t}
\big(K_1
+K_1\|u(r)\|^2
+2\|\phi_6(r)\|_{\infty}+
2\|\widehat{\phi}_6(r)\|_{\infty}\big)\|\psi^\epsilon_u(r)
-\psi_u(r)
\|^2dr
\nonumber\\&\ + 2\int_0^{t}
\big(K_1+2 \|\phi_7(r)\|_{1}+
2\|\widehat{\phi}_7(r)\|_{1}\big)
d_{\mathcal{P}_2}^2
(\mathcal{L}_{
\psi^\epsilon(r)},\delta_{\psi^0(r)})dr
+2\textbf{L}^\epsilon_{10}(t)
+2\sup_{r\in[0,t]}|\textbf{L}^\epsilon_{11}(r)|.
\end{align}
Then
\begin{align}\label{SCE5-2}
&\sup_{r\in [0,t]}\|\psi^\epsilon_u(r)-\psi_u(r)\|^2
+2^{-p}K^{-1}\int_0^{t}
\|\psi^\epsilon_u(r)
-\psi_u(r)\|^p_{\dot{\mathbb{V}}_1}dr
\nonumber\\&\ +2^{1-p}\beta
\int_0^{t}
\|\psi^\epsilon_u(r)
-\psi_u(r)\|_{p}^{p}dr
+2^{1-q}\widehat{\beta}
\int_0^{t}\|\psi^\epsilon_u(r)
-\psi_u(r)\|_{\mathbb{V}_2}^{q}dr \nonumber\\
&\leq
e^{2\int_0^{t}
\big(K_1
+K_1\|u(r)\|^2
+2\|\phi_6(r)\|_{\infty}+
2\|\widehat{\phi}_6(r)\|_{\infty}\big)dr}
\nonumber\\&\ \ \cdot\Big( 2\int_0^{t}
\big(K_1+2 \|\phi_7(r)\|_{1}+
2\|\widehat{\phi}_7(r)\|_{1}\big)
d_{\mathcal{P}_2}^2
(\mathcal{L}_{
\psi^\epsilon(r)},\delta_{\psi^0(r)})dr
+2\textbf{L}^\epsilon_{10}(t)
+2\sup_{r\in[0,t]}|\textbf{L}^\epsilon_{11}(r)|
\Big)
\nonumber\\
&\leq
e^{2
\big(K_1 T
+K_1 N^2
+2 \int_0^T \left(\|\phi_6(r)\|_{\infty}+
\|\widehat{\phi}_6(r)\|_{\infty}\right)dr \big)}
\nonumber\\&\ \ \cdot \Big(2\int_0^{t}
\big(K_1+2 \|\phi_7(r)\|_{1}+
2\|\widehat{\phi}_7(r)\|_{1}\big)dr
\mathbf{E}\Big[
\| \psi^\epsilon-\psi^0\|^2_{C([0,T],\mathbb{H} )}\Big]
+2\textbf{L}^\epsilon_{10}(t)
+2\sup_{r\in[0,t]}|\textbf{L}^\epsilon_{11}(r)|
\Big).
\end{align}

For the  terms $\textbf{L}^\epsilon_{10}(t)$ and $\sup_{r\in[0,t]}|\textbf{L}^\epsilon_{11}(r)|$ in \eqref{SCE5-2}, by \eqref{h2-}, \eqref{ffnewA2a0}
and the BDG inequality,
we infer that for all $t\in[0,T]$, $\psi_{0}\in \overline{B}_{R}(\mathbb{H})$,
$u\in \mathfrak{A}_N$ and $\epsilon\in(0,1)$,
\begin{align}&
\label{SCE7}
2 e^{2
\big(K_1 T
+K_1 N^2
+2 \int_0^T \left(\|\phi_6(r)\|_{\infty}+
\|\widehat{\phi}_6(r)\|_{\infty}\right)dr \big)}
\left( \mathbf{E}[
\textbf{L}^\epsilon_{10}(t)]+
\mathbf{E}\Big[\sup_{r\in[0,t]}
|\textbf{L}^\epsilon_{11}(r)|\Big]
\right)
\nonumber\\&\leq\frac{1}{2}
\mathbf{E}\Big[\sup_{r\in[0,t]}
\|\psi^\epsilon_u(r)
-\psi_u(r)
\|^2\Big]+\epsilon K_2\mathbf{E}\bigg[\int_0^{t}
\|\mathfrak{L}_{H,S}(r,\psi^\epsilon_u(r),
\mathcal{L}_{
\psi^\epsilon(r)})
		\|^2_{\mathcal{L}_2(\ell^2,
			\mathbb{H})}dr\bigg]
\nonumber\\&\leq \frac{1}{2}\mathbf{E}\Big[\sup_{r\in[0,t]}
\|\psi^\epsilon_u(r)
-\psi_u(r)
\|^2\Big]+\epsilon K_3
\int_0^{t}
\Big(1+\textbf{E}[\|\psi^\epsilon_u(r)\|_{p
}^{p}]+\textbf{E}[\|\psi^\epsilon_u(r)
\|_{\mathbb{V}_2}^{q}]+
\textbf{E}[\|\psi^\epsilon(r)\|^2]\Big)dr,
\end{align}where $K_2$ and $K_3$ are constants independent of $\epsilon$.
It follows from \eqref{SCE5-2} and \eqref{SCE7}
 that for all $t\in[0,T]$, $\psi_{0}\in \overline{B}_{R}(\mathbb{H})$,
$u\in \mathfrak{A}_N$ and $\epsilon\in(0,1)$,
 \begin{align}\label{SCE8}
&\mathbf{E}\Big[\sup_{r\in[0,t]}
\|\psi^\epsilon_u(r)
-\psi_u(r)
\|^2\Big]+2^{1-p}K^{-1}\mathbf{E}
\bigg[\int_0^{t}
\|\psi^\epsilon_u(r)
-\psi_u(r)\|^p_{\dot{\mathbb{V}}_1}dr\bigg]
\nonumber\\&\ +2^{2-p}\beta\mathbf{E}
\bigg[
\int_0^{t}
\|\psi^\epsilon_u(r)
-\psi_u(r)\|_{p}^{p}dr\bigg]
+2^{2-q}\widehat{\beta}\mathbf{E}
\bigg[
\int_0^{t}\|\psi^\epsilon_u(r)
-\psi_u(r)\|_{\mathbb{V}_2}^{q}dr\bigg]
\nonumber\\&\leq
K_4 \int_0^{t}
\big(K_1+2 \|\phi_7(r)\|_{1}+
2\|\widehat{\phi}_7(r)\|_{1}\big)dr
\mathbf{E}\Big[
\| \psi^\epsilon-\psi^0\|^2_{C([0,T],\mathbb{H} )}\Big]
\nonumber\\&\ +\epsilon K_4
\int_0^{t}
\Big(1+\textbf{E}[\|\psi^\epsilon_u(r)\|_{p
}^{p}]+\textbf{E}[\|\psi^\epsilon_u(r)
\|_{\mathbb{V}_2}^{q}]+
\textbf{E}[\|\psi^\epsilon(r)\|^2]\Big)dr,
\end{align}
where $K_4>0$ is a constant independent of $\epsilon$. Then we find from \eqref{SCE8}, Theorem \ref{Theorem4} and Lemma \ref{SCE1} that
for all $t\in [0,T]$,
 \begin{align}\label{SCE9}
&\mathbf{E}\Big[\sup_{r\in[0,t]}
\|\psi^\epsilon_u(r)
-\psi_u(r)
\|^2\Big]+2^{1-p}K^{-1}\mathbf{E}
\bigg[\int_0^{t}
\|\psi^\epsilon_u(r)
-\psi_u(r)\|^p_{\dot{\mathbb{V}}_1}dr\bigg]
\nonumber\\&\ +2^{2-p}\beta\mathbf{E}
\bigg[
\int_0^{t}
\|\psi^\epsilon_u(r)
-\psi_u(r)\|_{p}^{p}dr\bigg]
+2^{2-q}\widehat{\beta}\mathbf{E}
\bigg[
\int_0^{t}\|\psi^\epsilon_u(r)
-\psi_u(r)\|_{\mathbb{V}_2}^{q}dr\bigg]
\nonumber\\&\leq
K_5\mathbf{E}\Big[\|\psi^\epsilon-\psi^0
\|^2_{C([0,T],\mathbb{H} )}\Big]
+\epsilon K_5,
\end{align}
where $K_5=K_5(T,N,R)>0$  is a constant independent of $\epsilon$.

By \eqref{SCE9} and Lemma \ref{SCE2}, we obtain
\begin{align}\label{SCE10}
&\lim_{\epsilon\rightarrow0}\sup_{\psi_{0}\in \overline{B}_{R}(\mathbb{H})}
\sup_{u\in \mathfrak{A}_N}
\mathbf{E}\Big[\sup_{r\in[0,T]}
\|\psi^\epsilon_u(r)
-\psi_u(r)
\|^2\Big]=0,
\\&\label{SCE11}\lim_{\epsilon\rightarrow0}
\sup_{\psi_{0}\in \overline{B}_{R}(\mathbb{H})}
\sup_{u\in \mathfrak{A}_N}\mathbf{E}
\bigg[\int_0^{T}
\|\psi^\epsilon_u(r)
-\psi_u(r)\|^p_{\mathbb{V}_1}dr\bigg]
=0,
\end{align}
and
\begin{align}\label{SCE12}
\lim_{\epsilon\rightarrow0}\sup_{\psi_{0}\in \overline{B}_{R}(\mathbb{H})}
\sup_{u\in \mathfrak{A}_N}\mathbf{E}
\bigg[
\int_0^{T}\|\psi^\epsilon_u(r)
-\psi_u(r)\|_{\mathbb{V}_2}^{q}dr\bigg]=0.
\end{align}
As a result of \eqref{SCE10}-\eqref{SCE12}, we find that
$$\lim_{\epsilon\rightarrow0}\sup_{\psi_{0}\in \overline{B}_{R}(\mathbb{H})}
\sup_{u\in \mathfrak{A}_N}
\mathbf{E}\Big[ \|\psi^\epsilon_u-\psi_u\|^2_{C([0,T],\mathbb{H}) \bigcap L^p([0,T], \mathbb{V}_1)\bigcap L^q([0,T], \mathbb{V}_2)}\Big]=0.$$
From this and Chebyshev inequality, we conclude the proof.
\end{proof}

\subsection{Freidlin-Wentzell uniform LDPs of \eqref{Int1--}}

With the results in previous subsections in mind, we now prove the Freidlin-Wentzell uniform
LDPs of  the stochastic equation
\eqref{Int1--} in
$C([0,T],\mathbb{H} )\bigcap L^p([0,T], \mathbb{V}_1  )\bigcap L^q([0,T], \mathbb{V}_2)$
in the sense of Definition \ref{LDP} with
the good rate function $\mathcal{I}_{\psi_0}:C([0,T],\mathbb{H} )\bigcap L^p([0,T], \mathbb{V}_1  )\bigcap L^q([0,T], \mathbb{V}_2)
\rightarrow[0,+\infty]$ defined by, for $\psi\in C([0,T],\mathbb{H} )\bigcap L^p([0,T], \mathbb{V}_1  )\bigcap L^q([0,T], \mathbb{V}_2)$,
\begin{align}
\label{main1}
\mathcal{I}_{\psi_0}(\psi)=\inf\left\{\frac{1}{2}
\int_{0}^{T}||u(s)||_{\ell^{2}}^{2}ds:
u\in L^{2}([0,T],\ell^{2})~
\text{such~that}~ \psi_u=\psi\right\},
\end{align}
where $\inf\emptyset=+\infty$, and
$\psi_u=
\mathfrak{M}_{\psi_0}^0
\big(
\int_0^\cdot u(t)dt\big)$ is the solution of \eqref{Int1--++}.
\\

\textbf{Proof of Theorem \ref{MianI}}.
  The proof
consists of  the following  three steps:

\emph{Step 1}. Prove that, for every $N\in(0,\infty)$, the set
\begin{align}
\label{main2}\bigg\{\psi_u=
\mathfrak{M}_{\psi_0}^0
\bigg(
\int_0^\cdot u(t)dt\bigg):u\in \overline{B}_N(L^2([0,T],\ell^2 ))\bigg\}:=\mathfrak{S}_{N,\psi_0}
\end{align}
is compact in $C([0,T],\mathbb{H} )\bigcap
L^p([0,T], \mathbb{V}_1  ) \bigcap L^q([0,T], \mathbb{V}_2)$. Let $\{\psi_{u_{n}}\}_{n=1}^{\infty}$ be a sequence   with $\{u_{n}\}_{n=1}^{\infty}\subseteq \overline{B}_N(L^2([0,T],\ell^2 ))$. Then, there exists $u\in \overline{B}_N(L^2([0,T],\ell^2 ))$ and a subsequence $\{u_{n_{j}}\}_{j=1}^{\infty}$ such that $u_{n_{j}}\to u$ weakly in $L^{2}([0,T],\ell^2)$. This, together with
Lemma \ref{Weak-to-strog}, shows that $\psi_{u_{n_{j}}}\to \psi_{u}$ in $C([0,T],\mathbb{H} )\bigcap
L^p([0,T], \mathbb{V}_1  ) \bigcap L^q([0,T], \mathbb{V}_2)$, and hence the set in \eqref{main2} is compact $C([0,T],\mathbb{H} )\bigcap
L^p([0,T], \mathbb{V}_1  ) \bigcap L^q([0,T], \mathbb{V}_2)$.

\emph{Step 2}. Prove that the function $\mathcal{I}_{\psi_0}:\mathbb{H}
\rightarrow[0,\infty]$ in \eqref{main1} is a good rate function for every $\psi_0\in \mathbb{H}$.
By the definitions of $\mathcal{I}_{\psi_0}$
and $\mathfrak{M}_{\psi_0}^0$, one can verify that
\begin{align}
\label{main3}&
I^N_{\psi_0}:=\big\{\psi\in C([0,T],\mathbb{H} )\cap
L^p([0,T], \mathbb{V}_1  ) \cap L^q([0,T], \mathbb{V}_2):
\mathcal{I}_{\psi_0}(\psi)
\leq N\big\}=\mathfrak{S}_{\sqrt{2N},\psi_0}.
\end{align}
By step 1, we find that
the level set is compact in $C([0,T],\mathbb{H} )\bigcap
L^p([0,T], \mathbb{V}_1  ) \bigcap L^q([0,T], \mathbb{V}_2)$. This means that $\mathcal{I}_{\psi_0}:\mathbb{H}
\rightarrow[0,\infty]$ is a good rate function.

\emph{Step 3}. By steps 1 and 2 and Lemma \ref{SCE3}, we complete the proof by Theorem \ref{Tattractor1}.

\subsection{Dembo-Zeitouni  uniform LDPs of \eqref{Int1--}}

We now prove the Dembo-Zeitouni uniform
LDPs of problem \eqref{Int1--} in
$C([0,T],\mathbb{H} )\bigcap L^p([0,T], \mathbb{V}_1  )\bigcap L^q([0,T], \mathbb{V}_2)$
in the sense of Definition \ref{LDP-}. To that end, we first show the continuity of level sets
$I^s_{\psi_0}$ (with $s\geq0$)
of the rate function $\mathcal{I}_{\psi_0}:\mathbb{H}
\rightarrow[0,\infty]$.

\begin{lemma}\label{continuityoflevelsets}
Let conditions  {\bf A} and {\bf B1}
be satisfied. If $\psi_{0,n}\rightarrow\psi_{0}$
in $\mathbb{H}$, then for every $s\geq0$, the family of level sets $\{I^s_{\psi_{0,n}}\}_{n=1}^\infty$
converges to the level set $I^s_{\psi_{0}}$ with respect to the Hausdorff metric of $C([0,T],\mathbb{H} )\bigcap L^p([0,T], \mathbb{V}_1  )\\ \bigcap L^q([0,T], \mathbb{V}_2)$:
\begin{align}
\label{main4}\lim_{n\rightarrow\infty}
\sup_{\psi\in I^s_{\psi_{0}}}\dist_{C([0,T],\mathbb{H} )\bigcap L^p([0,T], \mathbb{V}_1  )\bigcap L^q([0,T], \mathbb{V}_2)}\big(\psi,I^s_{\psi_{0,n}}\big)=0,
\end{align}
and
\begin{align}
\label{main5}\lim_{n\rightarrow\infty}
\sup_{\psi\in I_{\psi_{0,n}}^s}\dist_{C([0,T],\mathbb{H} )\bigcap L^p([0,T], \mathbb{V}_1  )\bigcap L^q([0,T], \mathbb{V}_2)}\big(\psi,I^s_{\psi_{0}}\big)=0.
\end{align}
\end{lemma}

\begin{proof}
Given $s\geq0$, and suppose that $\psi_{0,n}\rightarrow\psi_{0}$
in $\mathbb{H}$. Given $\psi\in I^s_{\psi_{0}}$.  Then by \eqref{main2}-\eqref{main3}, we find that $\psi\in\mathfrak{S}_{\sqrt{2s},\psi_{0}}$, and hence there exists
$u\in\overline{B}_{\sqrt{2s}}(L^2([0,T],\ell^2 ))$ such that $\psi=\psi_u(\cdot,\psi_0)$, where $\psi_u(\cdot,\psi_0)$ is the solution of \eqref{Int1--++}
on $[0,T]$ with initial value $\psi_0\in \mathbb{H}$ and control $u$.
Since $\psi_u(\cdot,\psi_{0,n})$ is the solution of \eqref{Int1--++}
on $[0,T]$ with initial value $\psi_{0,n}\in \mathbb{H}$ and control $u\in\overline{B}_{\sqrt{2s}}(L^2([0,T],\ell^2 ))$,
by \eqref{main2}-\eqref{main3}, we find that $\psi_u(\cdot,\psi_{0,n})\in \mathfrak{S}_{\sqrt{2s},\psi_{0,n}}
=I^{s}_{\psi_{0,n}}$. Therefore, it follows
from Theorem \ref{Con1} that
\begin{align}&\label{main6}
\sup_{\psi\in I^s_{\psi_{0}}}\dist_{C([0,T],\mathbb{H} )\bigcap L^p([0,T], \mathbb{V}_1  )\bigcap L^q([0,T], \mathbb{V}_2)}\big(\psi,I^s_{\psi_{0,n}}\big)
\nonumber\\&\leq\sup_{\psi\in I^s_{\psi_{0}}}\dist_{C([0,T],\mathbb{H} )\bigcap L^p([0,T], \mathbb{V}_1  )\bigcap L^q([0,T], \mathbb{V}_2)}\big(\psi_u(\cdot,\psi_0),
\psi_u(\cdot,\psi_{0,n})\big)
\nonumber\\&\leq M(T,\psi_{0})\big(\|\psi_{0,n}-\psi_{0}\|
+\|\psi_{0,n}-\psi_{0}\|^{p/2}
+\|\psi_{0,n}-\psi_{0}\|^{q/2}\big),
\end{align}
where $M(T,\psi_{0})>0$ is a constant independent of $n$. Since $\psi_{0,n}\rightarrow\psi_{0}$
in $\mathbb{H}$, by \eqref{main6}, we
find \eqref{main4}. By the same method, we can also prove \eqref{main5}.
\end{proof}

\textbf{Proof of Theorem \ref{MianII}}:

\begin{proof}
The proof is a direct consequence of Lemma \ref{continuityoflevelsets} and Theorem \ref{MianI} based on the theoretical results in Lemma \ref{Tattractor1-}.
\end{proof}

\section*{
Competing interests and declarations}
We declare that the authors has no competing interests or other interests that might be perceived to influence the results and/or discussion reported in this paper. We declare that
this manuscript has no associated data.

\section*{Acknowledgements}
Renhai Wang was supported by National Natural Science
Foundation of China(No. 12301299), the Guiyang City Science and Technology Plan Project (No. [2024]2-17), and the Natural Science Research Project of Guizhou Provincial Department of Education (No. QJJ[2024]322).  Zhang Chen was supported by the NNSF of China (11971260, 12471167).

\end{document}